\documentclass[10pt,a4paper]{article}

\usepackage{amsfonts,amsmath,mathrsfs,amssymb,amsthm,hyperref}
\usepackage{verbatim}
\usepackage{array}
\usepackage{setspace}

\usepackage[pdftex]{graphics}
\usepackage[hang]{subfigure}

\usepackage{tikz}

\newcounter{thm}
\newcounter{ex}
\newcounter{re}
\newtheorem{Theorem}[thm]{Theorem}
\newtheorem{Lemma}[thm]{Lemma}

\newtheorem{Proposition}[thm]{Proposition}

\newtheorem{Example}[ex]{Example}
\newtheorem{Remark}[re]{Remark}
\newtheorem{Corollary}[thm]{Corollary}
\newtheorem{Definition}[thm]{Definition}
\newtheorem{Conjecture}[thm]{Conjecture}

\definecolor{revision}{rgb}{1,0,0}

\def\SG{$\cal{S}$}
\def\RG{$\cal{R}$}

\newcommand{\te}{\tilde{e}}
\newcommand{\tu}{\tilde{u}}
\newcommand{\tv}{\tilde{v}}
\newcommand{\bp}{\mathbf{p}}
\newcommand{\bu}{\mathbf{u}}

\makeatletter
\def\Ddots{\mathinner{\mkern1mu\raise\p@
\vbox{\kern7\p@\hbox{.}}\mkern2mu
\raise4\p@\hbox{.}\mkern2mu\raise7\p@\hbox{.}\mkern1mu}}
\makeatother

\begin{document}
\title {Symmetry adapted Assur Decompositions}

\author{Tony Nixon\\Department of Mathematics and Statistics, York University\\ 4700 Keele Street, Toronto, ON M£J1P3, Canada\\Bernd Schulze\\Department of Mathematics and Statistics, Lancaster University\\ Lancaster, LA1 4YF, U.K.\\
Adnan Sljoka\\Department of Physics, Ryerson University, Toronto\\ ON M5B 2K3, Canada and\\
Department of Psychology and Neuroscience, University of Colorado\\ Boulder, CO 80309, USA
and\\ Walter Whiteley\footnote{Supported by a grant from NSERC (Canada).}\\
 Department of Mathematics and Statistics, York University\\ 4700 Keele Street, Toronto, ON M3J1P3, Canada
}

\maketitle

\begin{abstract}  Assur graphs are a tool originally developed by mechanical engineers to decompose mechanisms for simpler analysis and synthesis. Recent work has connected these graphs to strongly directed graphs, and decompositions of the pinned rigidity matrix.  Many mechanisms have initial configurations which are symmetric, and other recent work has exploited the orbit matrix as a symmetry adapted form of the rigidity matrix.  This paper explores how the decomposition and analysis of symmetric frameworks and their symmetric motions can be supported by the new symmetry adapted tools.

\end{abstract}

\section{Introduction}\label{sec:intro}

Assur decompositions of mechanisms dates back to the work of the engineer Leonid Assur \cite{Assur} as a tool to simplify the analysis and synthesis of mechanisms. These techniques are widely used in the kinematical community. Several recent mathematical papers have reworked this approach using tools from rigidity theory, extending results from the plane to 3-space (as well as higher $d$-dimensional space) and providing algorithms for decomposing pinned isostatic frameworks into these minimal components \cite{SSW1,SSW2,3directed,pebbleassur,adnanthesis}.   Other recent papers have developed mathematical tools and algorithms to analyze the behaviour of symmetric frameworks \cite{KG1,BS6,BSWWorbit,jkt}. Given the examples of symmetric mechanisms, with fully symmetric motions, it is natural to consider how the symmetry-adapted tools can be used to decompose symmetric mechanisms for analysis, for synthesis and for control of these mechanisms.

Figure~\ref{fig:Mechanisms} shows two common mechanisms.  The Stewart Platform of Figure~\ref{fig:Mechanisms}(a,b,c) is a pinned isostatic structure which is widely studied in mechanical engineering and robotics, and has appeared in previous papers on symmetry and rigidity \cite{gsw,schtan}.  The symmetry analysis adds further information about when coordinated drivers give a fully symmetric motion.  The Grab Bucket (see, for example, \cite[P270]{kinematics}) depicted in Figure~\ref{fig:Mechanisms}(d) is a common machine which can be analyzed as a plane structure (see Section \ref{sec:variants}); this structure is pinned  with mirror symmetry and does not fit the previous work on Assur decompositions of plane mechanisms \cite{SSW1,SSW2}. 
\begin{figure}[ht]
    \begin{center}
\subfigure[] {\includegraphics [width=.25\textwidth]{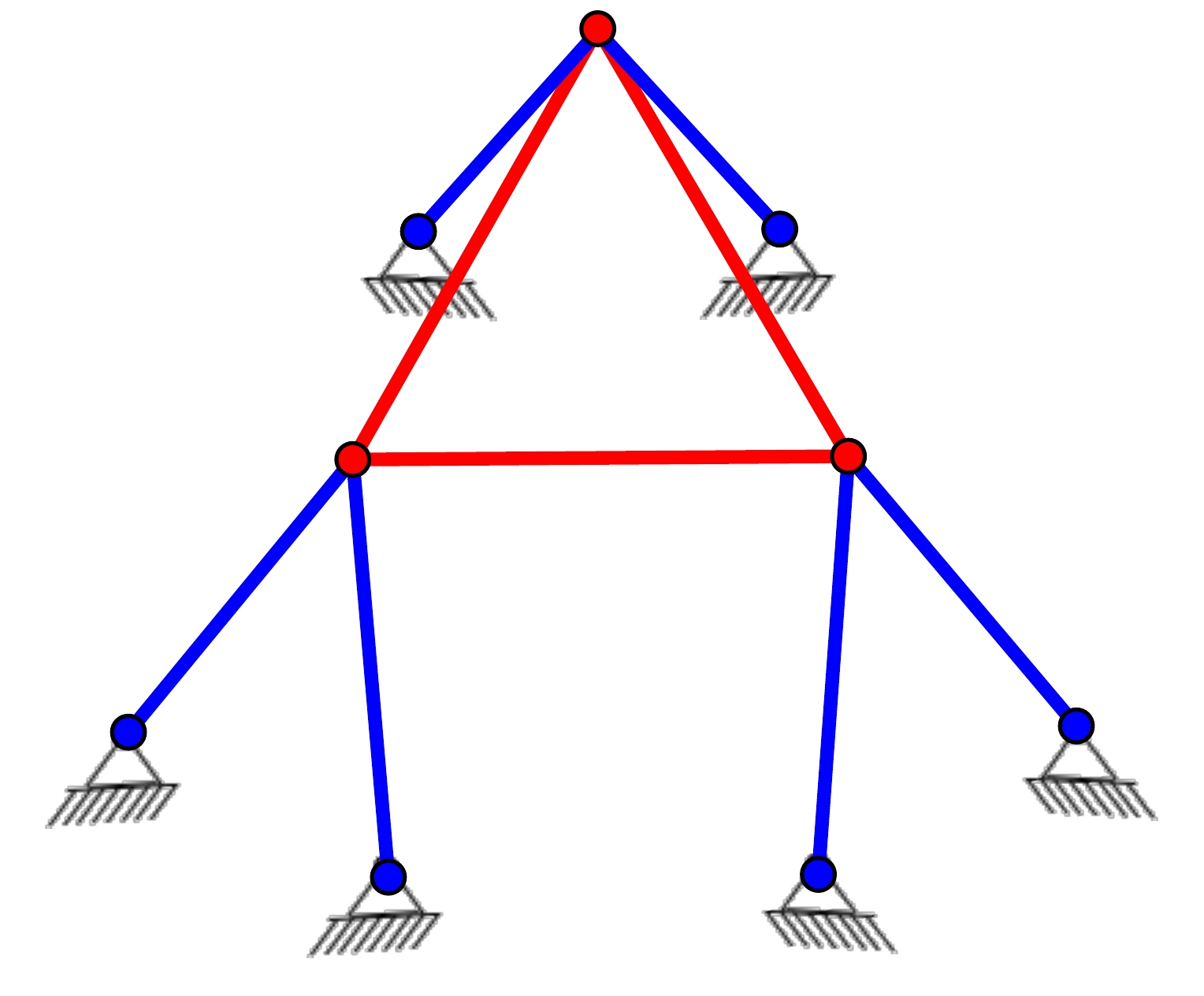}}
\subfigure[] {\includegraphics [width=.25\textwidth]{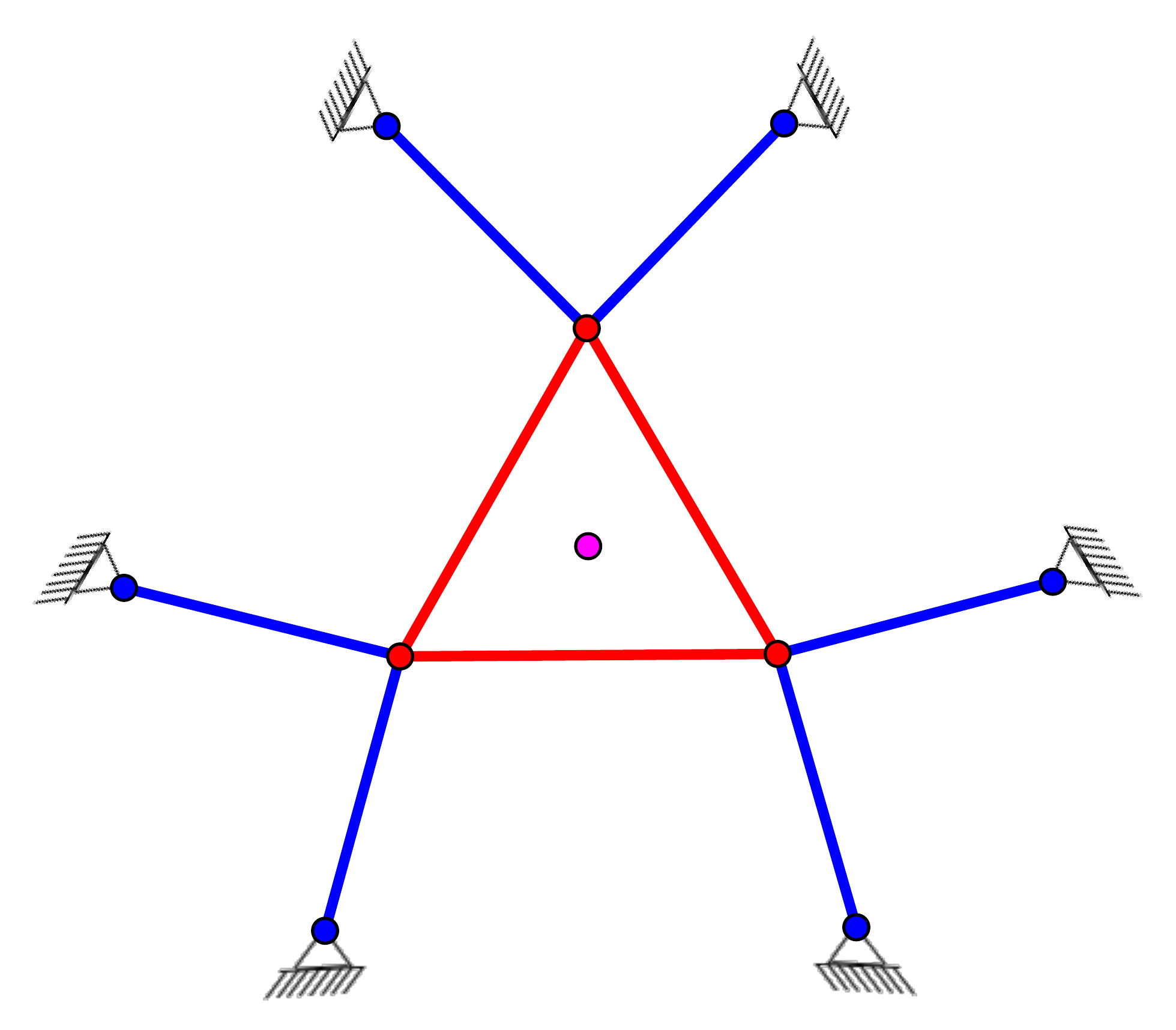}}
  \subfigure[] {\includegraphics [width=.25\textwidth]{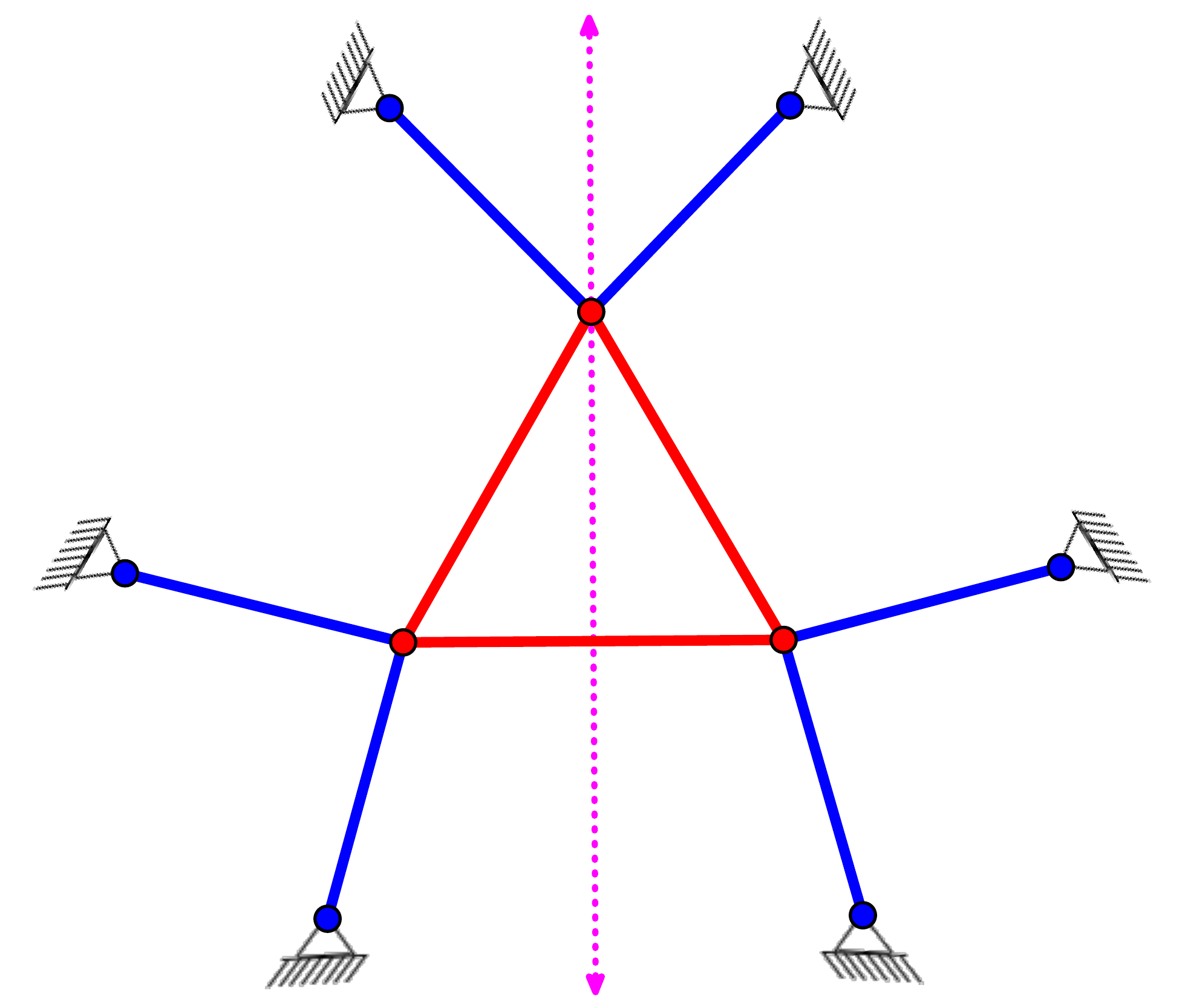}}\quad
    \subfigure[] {\includegraphics [width=.20\textwidth]{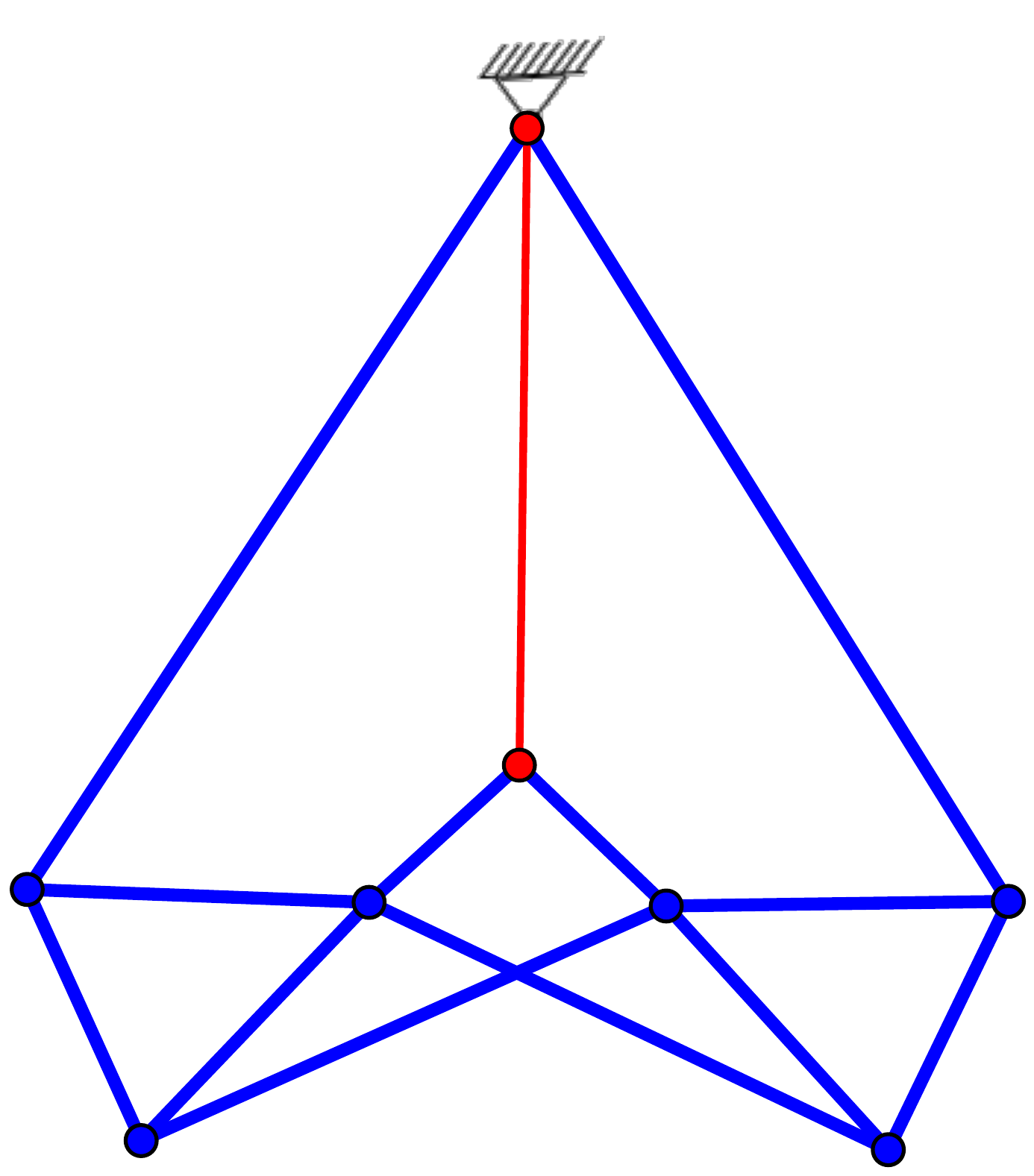}} \quad
    \end{center}
    \caption{Schematics of some sample mechanisms which are built with symmetry: (a) the Stewart Platform which can have $\mathcal{C}_3$ (i.e., $3$-fold rotational) symmetry (b) or mirror symmetry (c); the Grab Bucket (d) has mirror symmetry.}
    \label{fig:Mechanisms}
    \end{figure}

We will apply the techniques developed for isostatic pinned graphs (which generate square matrices) to orbit rigidity matrices which have independent rows and maximal rank among symmetry regular configurations.  Consider the pinned framework in Figure~\ref{fig:TwoMirror3D1}(a).  Although this has the graph of a generically rigid framework in 3-space, with the symmetry of the two mirrors, generating a dihedral symmetry group, the inner points not on the ground all move continuously, preserving this symmetry (b).  In mechanical engineering, one can ask which drivers (say pistons) can be inserted to control this motion.   Figure~\ref{fig:TwoMirror3D1}(c) shows one set of drivers which, if they expand in a synchronized way, drive the framework along the path of the symmetric motion.  The graph of this extended framework, or symmetric scheme, is generically redundant, but in the symmetry analysis of the associated orbit matrix, this is minimally rigid for symmetric motions, or \SG-isostatic.

\begin{figure}[ht]
    \begin{center}
  \subfigure[] {\includegraphics [width=.29\textwidth]{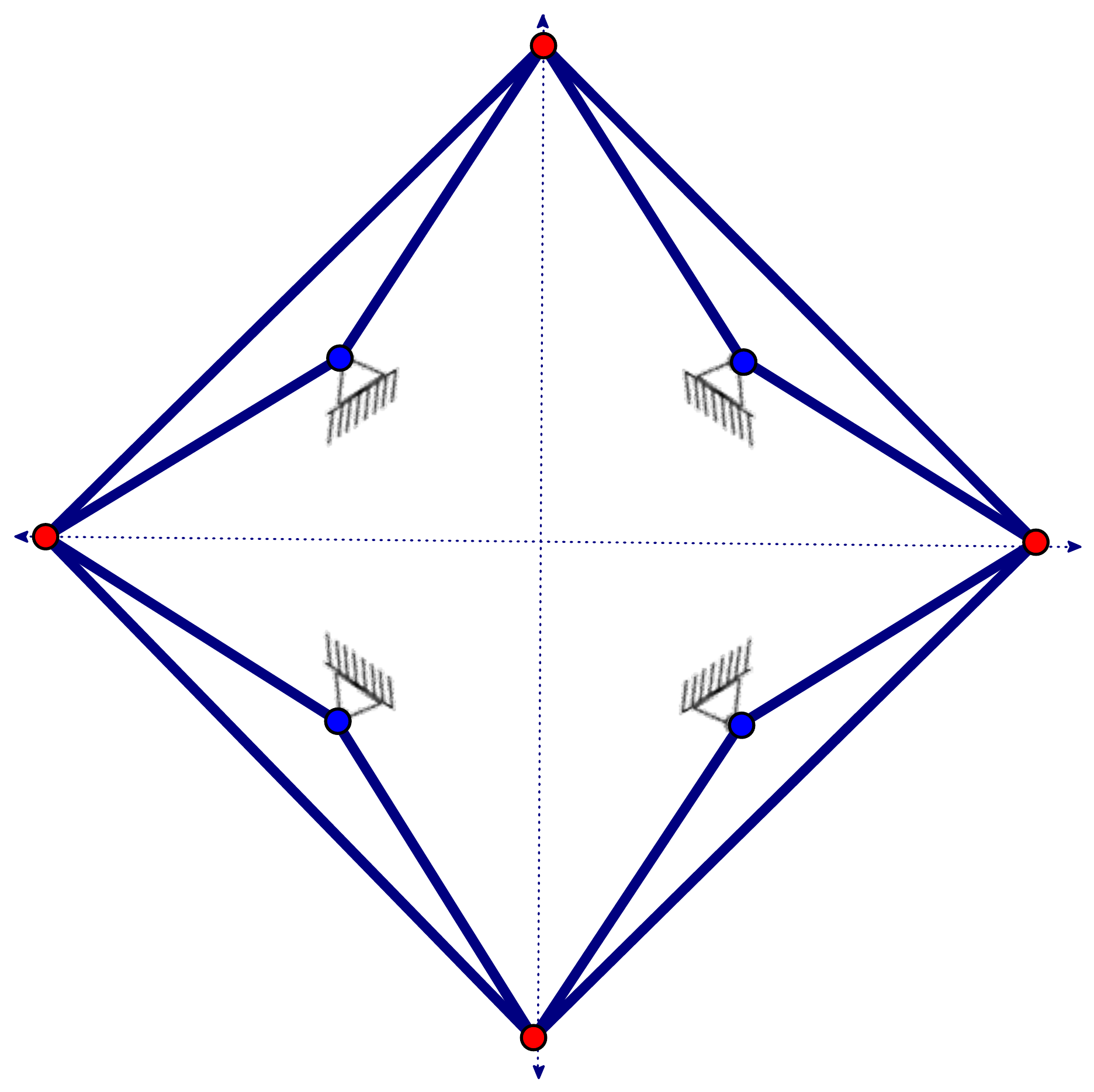}} \quad
\subfigure[] {\includegraphics [width=.33\textwidth]{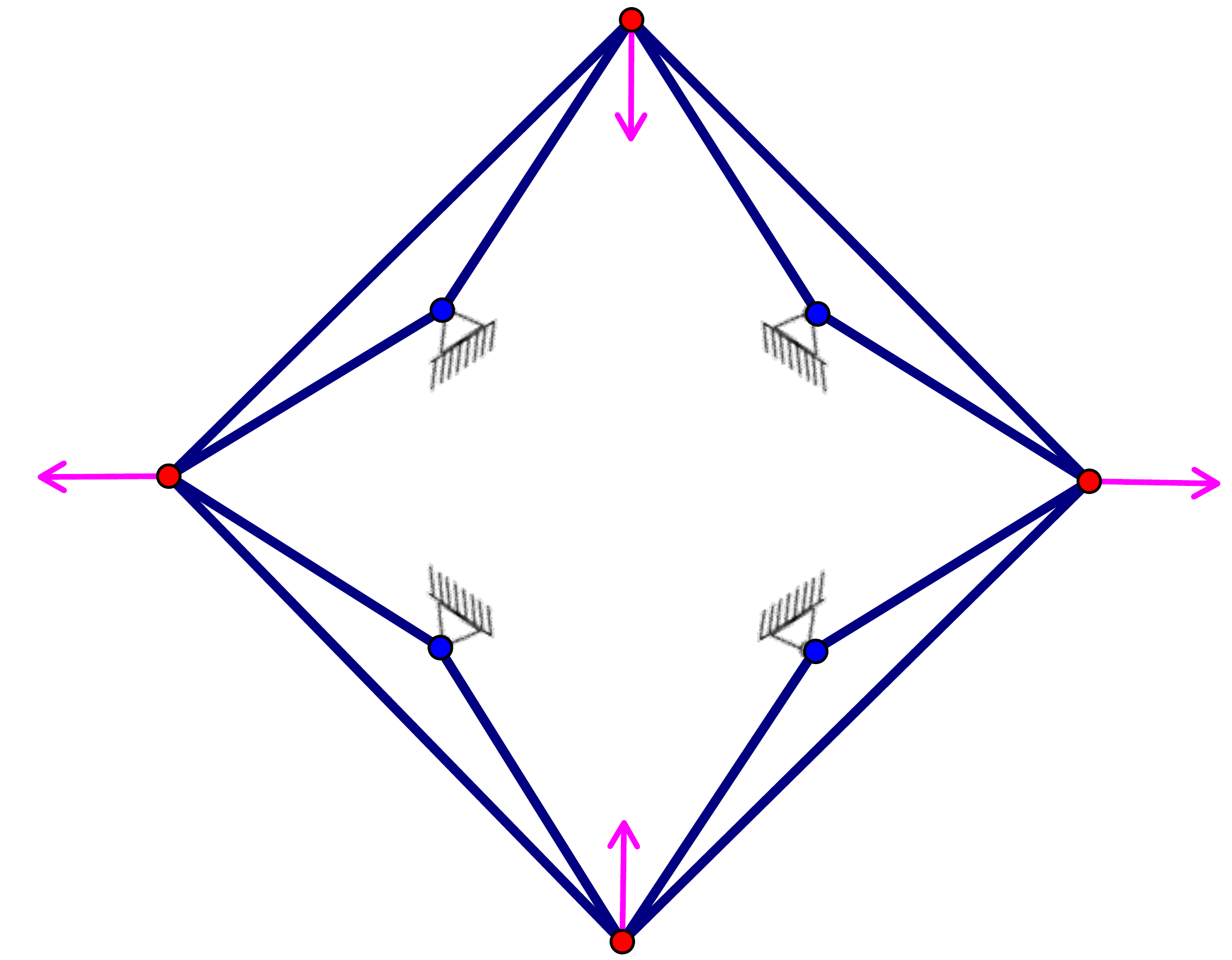}}\quad
\subfigure[] {\includegraphics [width=.29\textwidth]{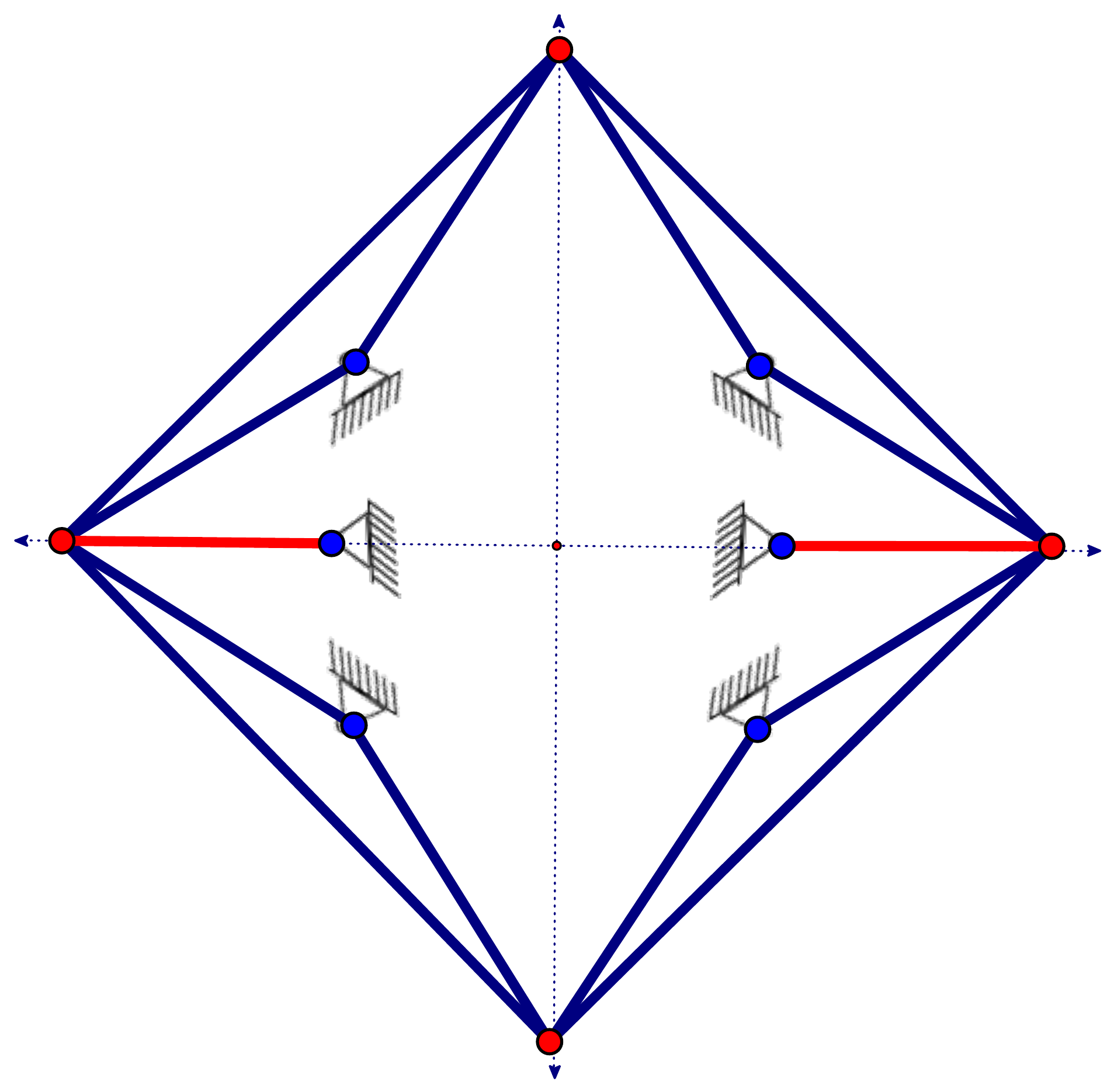}}\quad
      \end{center}
    \caption{A framework in $3$-space  with dihedral symmetry which is generically rigid but flexible with this symmetry (a), (b).  When symmetric drivers (in red) are added (c), changing the lengths of the added edges (as coordinated pistons) moves the framework in a controlled symmetric fashion.}
    \label{fig:TwoMirror3D1}
    \end{figure}


In our context, given a symmetry group \SG, we can consider a framework which is minimally rigid for symmetric motions or  {\it pinned \SG-isostatic} (Figures~\ref{fig:Mechanisms}, \ref{fig:TwoMirror3D1}(c), \ref{fig:desargues}(a)).  Graphs which are minimal pinned \SG-isostatic will be called {\it \SG-Assur}.  Selecting any orbit of edges (i.e. each edge together with all its symmetric copies) under the action of the symmetry group as coordinated drivers will now generate a symmetric motion. There is a residual question whether this is the only residual motion - and we will return to such questions below.



Overall, the paper extends the basic techniques of Assur decompositions, and the impact of selected drivers to a new context, whenever the associated pinned orbit matrix is square and full rank.  This analysis confirms that, if the underlying graph is also pinned isostatic, we get the equivalent Assur decomposition to an analysis on the underlying graph, confirming how symmetric orbits of drivers can drive symmetric motions.  The symmetry analysis is simpler, because the group-labeled quotient graph (also called gain graph) and the resulting algebra of the orbit matrix are smaller.  The symmetry based analysis also extends other structures where the underlying graph is not pinned isostatic, extending the insight of Assur decompositions to new mechanisms and posing some new possibilities for synthesis of mechanisms with symmetric motions.

We finish the introduction by outlining the paper.
\begin{itemize}
\item In Section \ref{sec:back} we review the key concepts and definitions on pinned frameworks, Assur graphs and symmetric frameworks.
\item In Section \ref{sec:pinsymfw} we develop combinatorial conditions on the pinned \SG-gain graph that ensure \SG-isostaticity.
\item In Section \ref{sec:sassur} we consider the broad class of `pinned \SG-isostatic graphs'.  With this square invertible pinned orbit matrix, we are guaranteed an appropriate set of directions on the edges, of out-degree equal to the number of columns of that vertex in this pinned orbit matrix.  Any such \SG-directed orientation gives a strongly-connected directed graph decomposition, and the components, along with their outgoing edges, are the \SG-Assur graphs.
\item Section \ref{sec:isostatic} examines the special subclass where the underlying graph is also generically isostatic, and remains isostatic at some (almost all) \SG-regular configurations.  This extra assumption allows us to map between the $d$-Assur decomposition and the \SG-Assur decomposition.
\item Section  \ref{sec:variants} examines examples of other  types of pinned \SG-isostatic frameworks where the underlying graph is not pinned isostatic, but redundant and rigid or is flexible. In those cases, only the \SG-Assur decomposition is possible, and this new decomposition of the underlying graph provides additional insight for analysis and synthesis.
\item In the final two sections we consider inductive constructions for \SG-Assur graphs (Section \ref{sec:Inductions}) and outline some extensions of our work based on the philosophy of applying the decomposition techniques to constraint systems that generate square matrices (Section \ref{sec:conclusions}).
\end{itemize}

\section{Background}\label{sec:back}

In this section we recap relevant ideas from the literature on pinned frameworks, Assur graphs and symmetric frameworks.

\subsection{Pinned frameworks}


We set $\hat G=(I,P;E)$ where $\hat G$ is the (finite simple) graph with vertex set $I\cup P$ and edge set $E$. Vertices in $I$ are referred to as \emph{inner} and vertices in $P$ are referred to as \emph{pinned}.
A \emph{pinned framework} $(\hat G,\bp)$ is the combination of a pinned graph $\hat G$  with a map $\bp:I \cup P \rightarrow \mathbb{R}^d$. For simplicity, we will denote $\bp(v)$ by $\bp_v$ for $v\in I \textrm{ or } P$  (Figures~\ref{fig:Mechanisms}(a,d), \ref{fig:TwoMirror3D1}(c), \ref{fig:desargues}(a)).


For a pinned framework $(\hat G,\bp)$ in $\mathbb{R}^d$, we define the \emph{pinned rigidity matrix} to be the $|E|\times d|I|$ matrix with one row per edge and d columns per inner vertex as follows:
\setcounter{MaxMatrixCols}{20}
\[R(\hat G,\bp)=
\begin{bmatrix}
\ddots & \vdots & & & & \vdots & & & &  \vdots&\Ddots\\
0 & \dots & 0& (\bp_i-\bp_j)&0 & \dots& 0& (\bp_j-\bp_i)& 0& \dots& 0\\
 & \vdots & & & & \vdots& & & & \vdots& \\
0 & \dots & 0& (\bp_i-\bp_k)& 0& \dots& 0& 0& 0& \dots& 0 \\
\Ddots & \vdots& & & &\vdots & & & & \vdots &\ddots
\end{bmatrix}\]
In the matrix the displayed rows correspond to edges $ij$ and $ik$ where $i,j \in I$ and $k\in P$.
Unlike the standard rigidity matrix, note that the pinned rigidity matrix only has columns for inner vertices.

We call solutions $U$ to the equation $R(\hat G, \bp) \times U^T=0$ \emph{pinned infinitesimal motions} of $(\hat G,\bp)$. If the only such motion is the zero motion then $(\hat G,\bp)$ is said to be \emph{pinned infinitesimally rigid}. Equivalently $(\hat G,\bp)$ is pinned infinitesimally rigid if rank $R(\hat G,\bp)=d|I|$.
Moreover $(\hat G,\bp)$ is \emph{pinned $d$-independent} if the rows of $R(\hat G,\bp)$ are linearly independent and $(\hat G,\bp)$ is \emph{pinned $d$-isostatic} if $(\hat G,\bp)$ is pinned infinitesimally rigid and pinned $d$-independent.

A \emph{pinned self-stress} $\mathbf{\omega}$ is an assignment of real weights to the edges of $\hat G$ such that the following equilibrium condition holds:
$$\sum_{j} \omega_{ij}(\bp_i-\bp_j)=0$$
where $\mathbf{\omega}$ is not the zero vector, $\omega_{ij}=\mathbf{\omega}(ij)$ is taken to be  equal to $0$ if $ij\not\in E$ and the inner vertices are denoted by  $1\leq i \leq |I|$.

For a given pinned graph $\hat G$, let $\bp$ vary over all of $\mathbb{R}^{d(|I|+|P|)}$. For all $\mathbf{q}$ in an open subset $U$, $R(\hat G,\mathbf{q})$ has maximal possible rank. Any pinned framework $(\hat G,\bp)$ which achieves this rank is said to be \emph{regular}.

We refer the reader to \cite{SSW1, SSW2, 3directed, adnanthesis} for more detailed definitions and discussions on pinned frameworks.


\subsection{Assur decompositions}\label{subsec:Assur}

In this subsection, we review Assur decompositions of pinned $d$-isostatic frameworks. The reader can refer to \cite{3directed} for further details, equivalent definitions and more examples.

A pinned  graph $\hat G$ is \emph{pinned $d$-isostatic} if there exists a pinned $d$-isostatic realisation of $\hat G$ in $d$-space. A \emph{$d$-Assur graph} (in mechanical engineering also known as Assur group) is a minimal pinned $d$-isostatic graph, where minimal means that no proper subgraph (containing at least one inner vertex) is also a pinned $d$-isostatic graph.
A \emph{$d$-directed orientation} of a pinned graph $\hat G$ is an assignment of directions to the edges of $\hat G$ such that every inner vertex has out-degree exactly $d$ and every pinned vertex (ground vertex) has out-degree exactly 0 (a sink in the directed graph). In  \cite[Theorem 3.6]{3directed} the natural necessary counting conditions were given for a pinned graph to be isostatic and it was noted that any graph satisfying these counts has a $d$-directed orientation. However, there exist examples where a pinned graph has a $d$-directed orientation, even for $d=2$, but is not pinned $d$-isostatic.

\begin{figure}[ht]
    \begin{center}
\subfigure[] {\includegraphics [width=.32\textwidth]{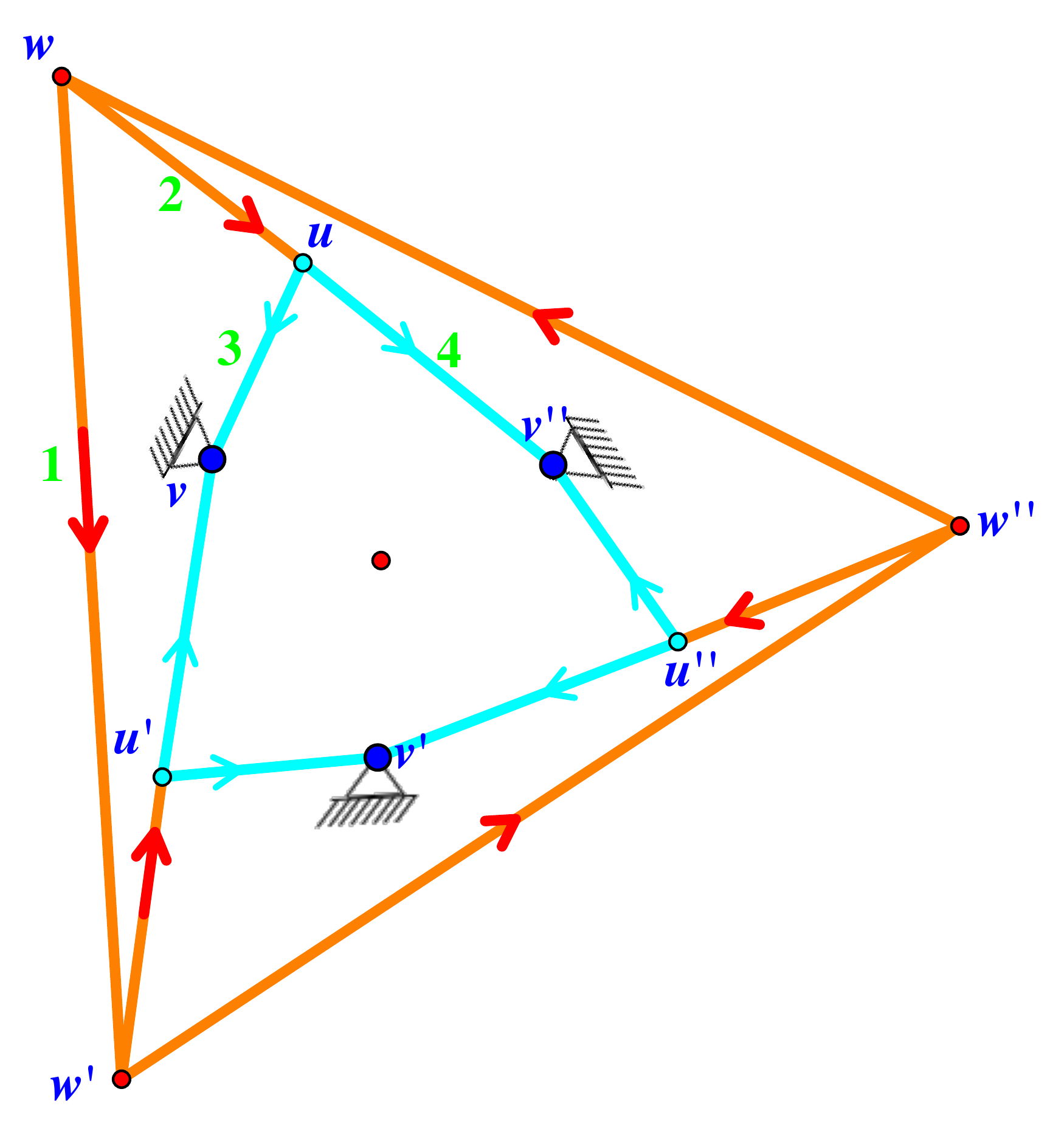}}\quad
\subfigure[] {\includegraphics [width=.35\textwidth]{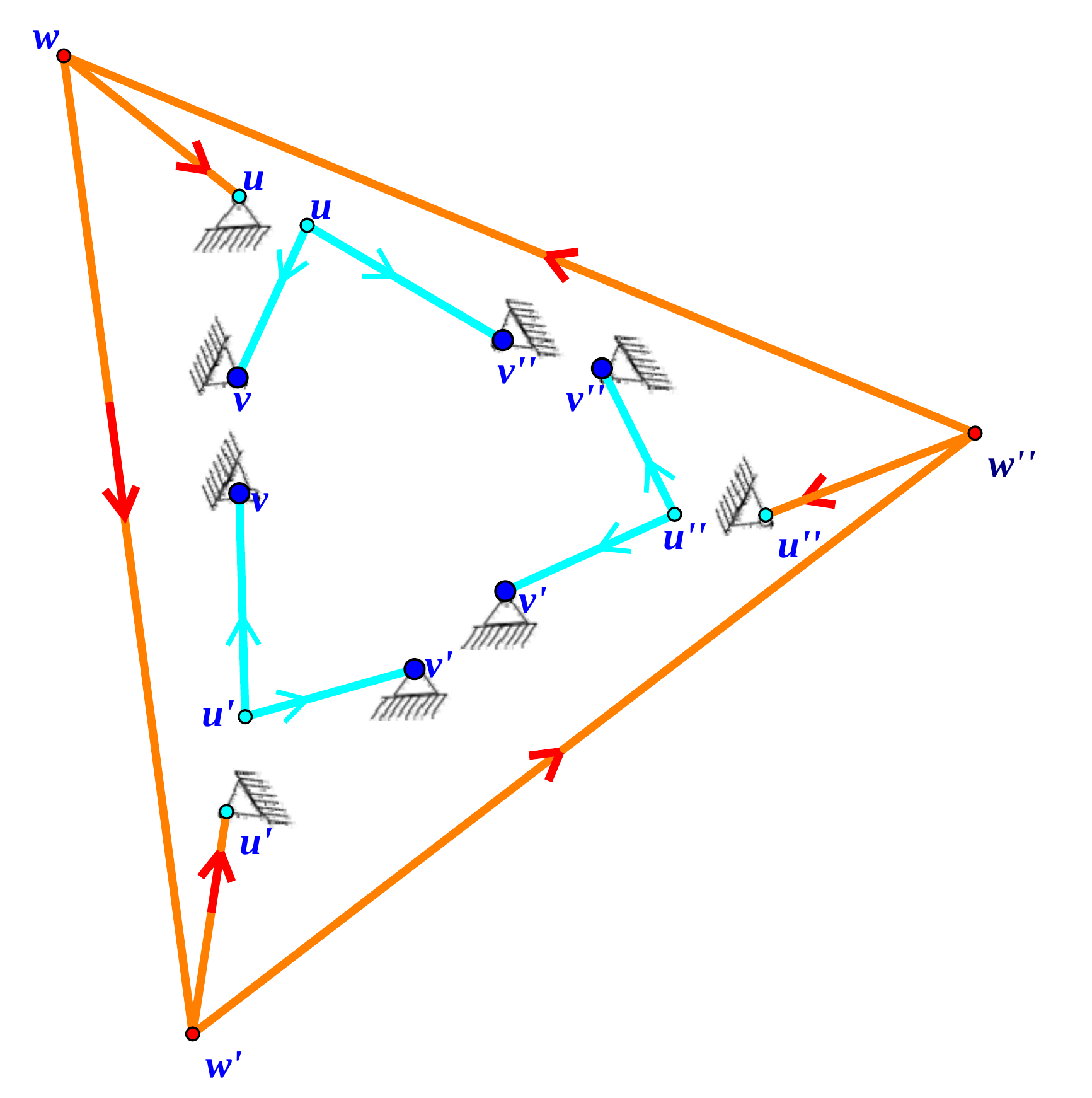}}\quad
  \subfigure[] {\includegraphics [width=.22\textwidth]{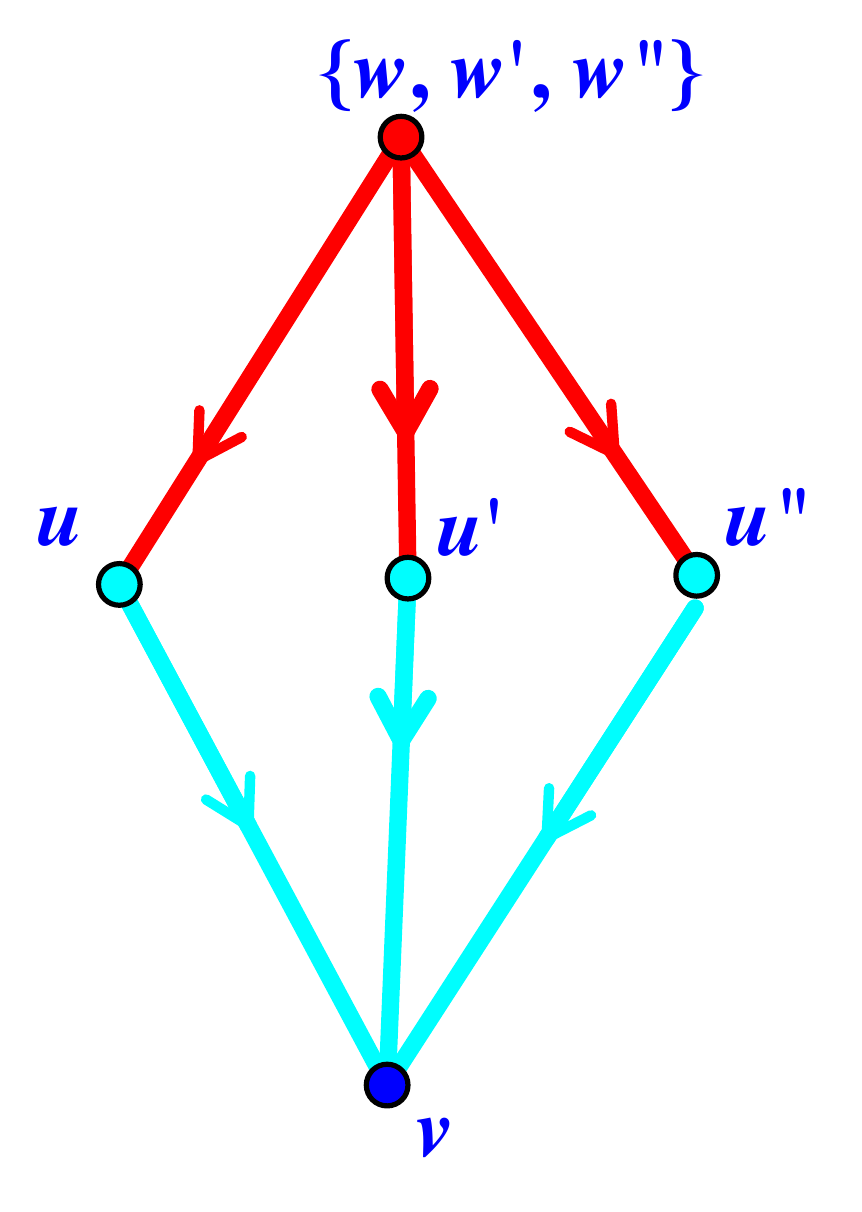}}\quad
\subfigure[] {\includegraphics [width=.35\textwidth]{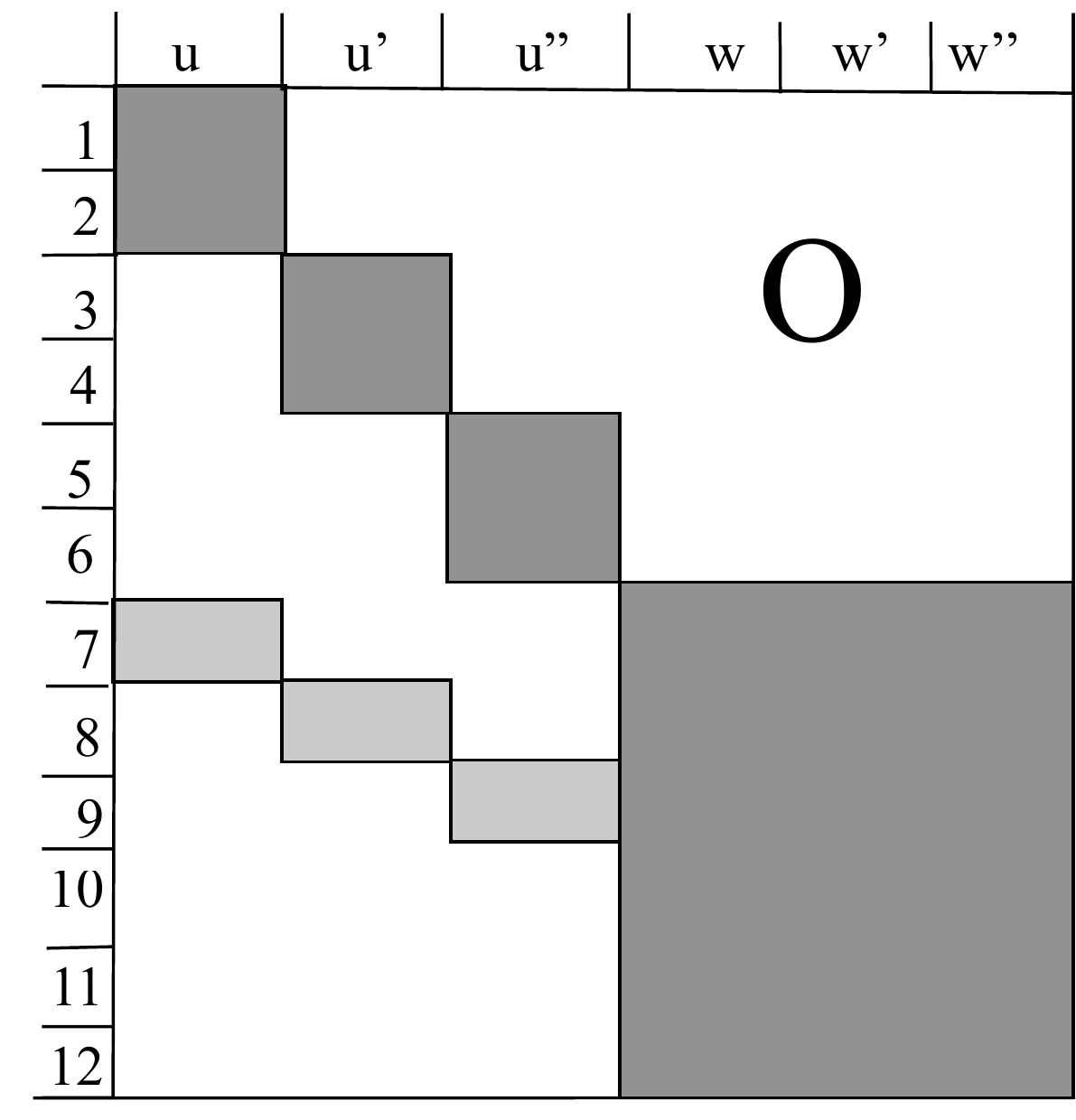}}\quad
  \subfigure[] {\includegraphics [width=.35\textwidth]{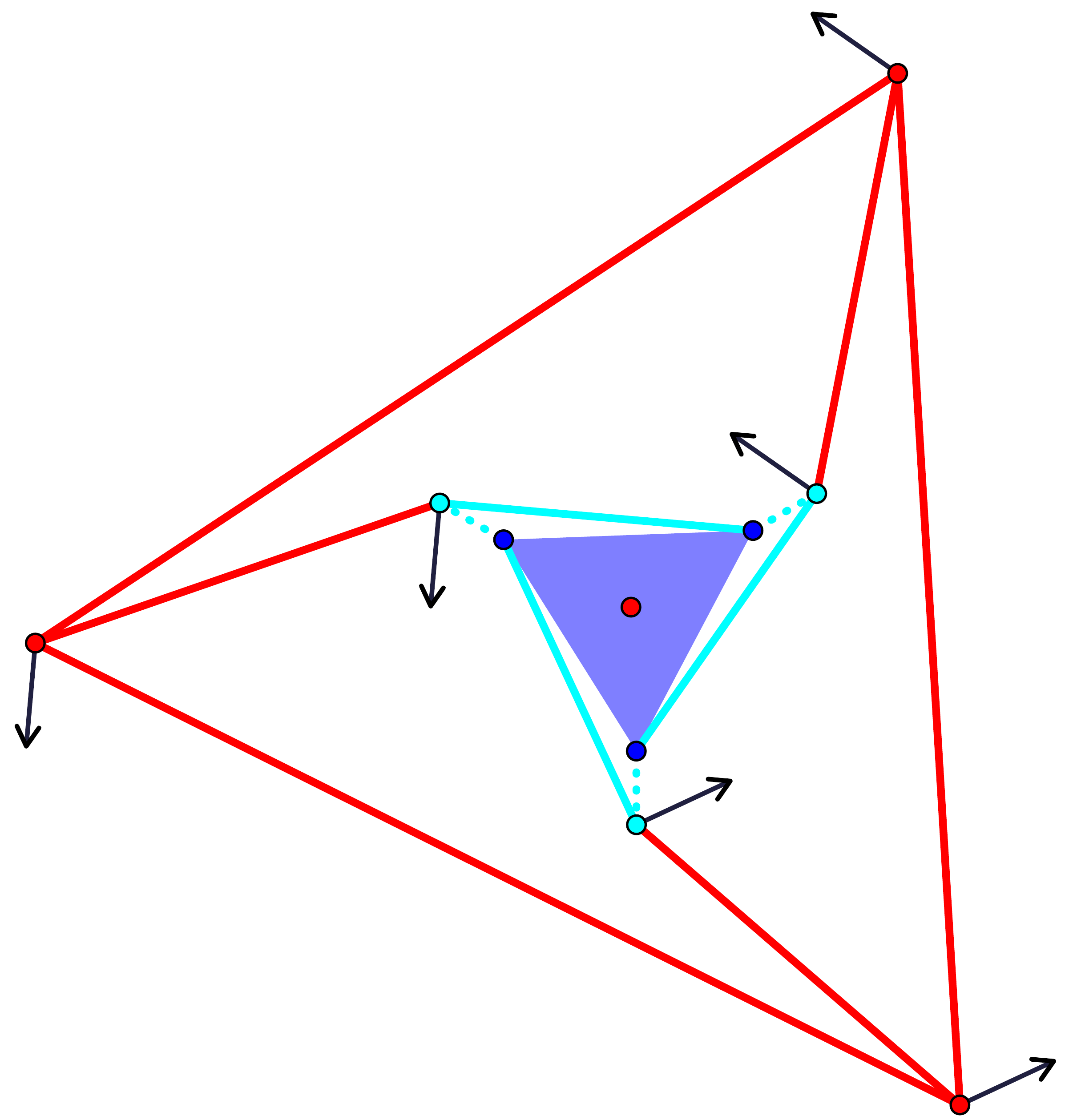}}\quad
    \end{center}
    \caption{A $\mathcal{C}_3$-symmetric pinned $2$-isostatic framework in the plane  with a  $2$-orientation (a). The associated $2$-Assur decomposition gives four $2$-Assur components (b), and an associated Assur partial order (c). The pinned rigidity matrix has a lower triangular block-decomposition as shown in (d). (e) depicts a set of velocities if the dotted edges become drivers.}
    \label{fig:desargues}
    \end{figure}

Starting with a pinned $d$-isostatic graph $\hat G$, a key initial step in the Assur decomposition is to generate a $d$-directed graph of $\hat G$ (for instance via the pebble game \cite{pebbleassur, adnanthesis}), where edges are directed toward the ground. The $d$-directed graph is then decomposed into its strongly connected components (i.e. maximally strongly connected subgraphs) \cite{3directed}. More specifically, a \emph{strongly connected component decomposition} of a pinned $d$-directed graph of $\hat G$ condenses all pins into a sink (ground) vertex, identifies the strongly connected components and condenses each such component to single vertices to obtain a directed acyclic graph and a partial order. Note if any multiple edges arise during this identification process we discard additional copies.

In \cite{3directed} it was noted that if there are two orientations of a graph in which corresponding vertices have the same out-degree then the two directed graphs differ at most by reversing orientations of cycles.  This implies that the strongly directed decomposition is the same for any such pair of orientations of the graph.

\begin{Corollary}[\cite{3directed}]
\label{cor:strong}
Given two equivalent orientations of a graph, the strongly connected components are the same in both orientations.
\end{Corollary}


An \emph{Assur decomposition} is a decomposition of a pinned $d$-isostatic graph where the individual components are the $d$-Assur graphs. This is exactly the strongly connected component decomposition. Each strongly connected component and its outgoing edges with ends becoming pins (we will refer to these as \emph{extended components}) are the $d$-Assur components (see Figure~\ref{fig:desargues}).

In \cite{3directed} additional connections and equivalent properties of $d$-Assur decompositions were presented via a lower triangular block decomposition of the pinned rigidity matrix, by permuting rows and columns following the partial order (see Theorem~\ref{3DirectedAssur}).
%

For any pinned isostatic graph $\hat G$, the pinned rigidity matrix $R(\hat G,\bp)$ will be square. If it is not possible to permute the rows and columns into lower triangular blocks then we say that $R(\hat G,\bp)$ is \emph{indecomposable}. In general we consider a \emph{lower triangular block decomposition} of the rigidity matrix $R(\hat G,\bp)$, and hence the graph $\hat G$, by permuting the rows and columns into indecomposable blocks. This is illustrated in Figure \ref{fig:desargues}.


A \emph{driver} (in mechanical engineering) is an edge in a pinned isostatic graph that is removed and hence its length is allowed to change. This drives a unique motion in the framework (e.g. a piston) and hence can easily be controlled.
Such motions can be understood by looking at the $d$-Assur decomposition.

We  say that a $d$-Assur graph is \emph{strongly} $d$-Assur if removing any edge puts all inner vertices into motion
. For $d=2$ this coincides with the definition of $2$-Assur but for $d\geq 3$ it is an inequivalent notion (see \cite{3directed} for examples that are $3$-Assur but not strongly $3$-Assur). These examples are based on the fact that for $d\geq 3$, the length of a non-edge may be determined by a non-rigid component (i.e. the lack of a combinatorial (counting) characterization of rigidity).

Removal of an edge (driver) from a component $C$ of a strongly $d$-Assur graph makes every vertex in every component above $C$ (i.e. above in the partial order - see below) go into motion and keeps every component below $C$ fixed; such edges in \cite{3directed} were called \emph{regular drivers}. Note, however, that removal of an edge (i.e. \emph{weak driver} \cite{3directed}) from a component $C$ of a $d$-Assur graph may leave entire components above $C$ fixed.


We finish this subsection with a key result on $d$-Assur graphs. Since we will provide a different style of proof, to \cite{3directed}, for the symmetric analogue in Section \ref{sec:decomp} we first prove a proposition that allows us to state \cite[Theorems 3.3 and 3.4]{3directed} in a single theorem.

Given a maximal lower triangular block decomposition of $R(\hat G,\bp)$, the {\it induced directed block graph} has one vertex per block plus a vertex $Z$ for the ground.  There is a directed block graph edge if there is a directed edge $(A,B)$ that goes from the block $A$ to a block $B$ which is upper left from it, i.e., if there is an edge with start vertex in $A$ and end vertex in $B$.  There is a directed edge $(A,Z)$ to the ground if there is an edge in block $A$ which goes to a pinned vertex.  See Figure \ref{fig:desargues}(c), (d).

\begin{Proposition} Given a pinned isostatic graph $\hat G$ in dimension $d$ and a maximal lower triangular block decompositon of $R(\hat G,\bp)$, the induced block graph is an acyclic directed graph, with the ground $Z$ on the bottom. Therefore it forms a partial order.
\label{Proposition:inducedblockgraph}
\end{Proposition}

\begin{proof}   Because the block-decomposition is lower triangular, there is a linear order of the blocks - by the position of their columns.  Place the ground vertex at the bottom of the linear order.   Observe that all directed block edges point down this linear order, because of the block lower-triangular pattern of the matrix.   Therefore, there cannot be a cycle in this graph.   The graph is a partial order, and the linear order of the blocks is an extension of this partial order.
\end{proof}

\begin{Remark}
\label{rem:inducedblockgraph}
The lower triangular block decomposition is not unique.  Given the partial order, any linear extension produces a lower-triangular block decomposition.    Any two blocks which are incomparable in the partial order can be switched in the linear order.  They can also be switched in the block decomposition by permuting the corresponding rows and columns.
\end{Remark}

\begin{Theorem}[ \cite{3directed}]\label{3DirectedAssur}
For a pinned $d$-isostatic graph $\hat G$ and any $d$-directed orientation of $\hat G$ the following are equivalent:
\begin{enumerate}
\item the $d$-Assur decomposition of $\hat G$;
\item the decomposition into strongly connected components associated with the $d$-directed orientation;
\item the induced block graph from a maximal block-triangular decomposition of the pinned rigidity matrix.
\end{enumerate}
\end{Theorem}

We call the shared partial order from the three equivalent decompositions of Theorem~\ref{3DirectedAssur} the \emph{$d$-Assur block graph}.

\subsection{Symmetric graphs}

For a finite simple graph $G=(V,E)$, we let $\textrm{Aut}(G)$ denote the automorphism group of $G$. An \emph{action} of a group $\mathcal{S}$ on $G$ is a group homomorphism $\theta: \mathcal{S} \rightarrow \textrm{Aut}(G)$.
An action $\theta$ is called \emph{free} on the vertices (resp., edges)
if  $\theta(x)(v)\neq v$ for every $v\in V$ (resp., $\theta(x)(e)\neq e$ for every $e\in E$) and
every  non-trivial element $x\in \mathcal{S}$.
We say that a graph $G$ is \emph{$\mathcal{S}$-symmetric} (with respect to $\theta$)
if $\mathcal{S}$ acts on $G$ via $\theta$.
Throughout the paper, we will omit to specify the action $\theta$ if it is clear from the context. In that case we also denote $\theta(x)(v)$ by  $x v$.

For an $\mathcal{S}$-symmetric graph $G=(V,E)$, the \emph{quotient graph} $G/\mathcal{S}$ is a multigraph whose vertex set is the set $V/\mathcal{S}$
of vertex orbits and whose edge set is the set $E/\mathcal{S}$ of edge orbits. Note that an edge orbit may be represented by a loop in $G/\mathcal{S}$.

While several distinct graphs may have the same quotient graph, a gain labeling makes this relation one-to-one if
 the underlying group action is free on $V$ \cite{schtan}.
To see this,  choose a  representative vertex $v$ from each vertex orbit under the group action.
Then each orbit is of the form $\mathcal{S} v=\{xv\mid x\in \mathcal{S}\}$.
If the action is free,
an edge orbit connecting $\mathcal{S} u$ and $\mathcal{S} v$ in $G/\mathcal{S}$
can be written as $\{\{xu,xyv\}\mid x\in \mathcal{S} \}$
for a unique $y\in \mathcal{S}$.
We then orient the edge orbit from $\mathcal{S} u$ to $\mathcal{S} v$ in $G/\mathcal{S}$
and assign to it the gain $y$.
This yields the \emph{quotient $\mathcal{S}$-gain graph}  $(G/\mathcal{S},\psi)$ of $G$, which
 is unique up to choices of representative vertices.  The map $\psi$ is also called the \emph{gain function} of $(G/\mathcal{S},\psi)$. Note that a gain graph is a directed graph, but its orientation is only
 used as a reference orientation, and may be changed, provided that we also modify $\psi$ so that if an edge has gain $x$ in one orientation, then it has gain $x^{-1}$ in the other direction.

 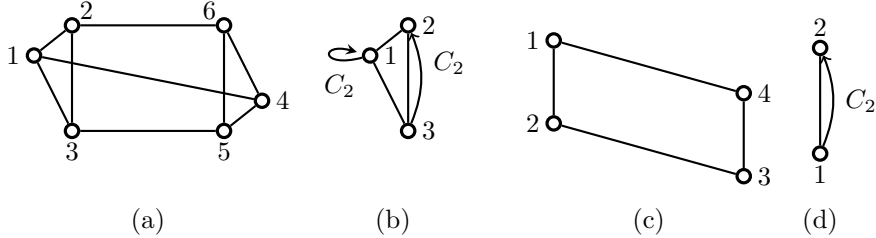
\begin{figure}[htp]
\begin{center}
\begin{tikzpicture}[very thick,scale=1]
       \tikzstyle{every node}=[circle, draw=black, fill=white, inner sep=0pt, minimum width=5pt];
   \path (0,-1.2) node (p3) [label = below: $3$] {} ;
    \path (2,-1.2) node (p5) [label = below: $5$] {} ;
    \path (2,0.2) node (p6) [label = above left: $6$] {} ;
   \path (0,0.2) node (p2) [label = above right: $2$] {} ;
   \path (-0.5,-0.2) node (p1) [label = left: $1$] {} ;
   \path (2.5,-0.8) node (p4) [label = right: $4$] {} ;
   \draw[thick] (p2) -- (p3);
     \draw[thick] (p1) -- (p2);
     \draw[thick] (p1) -- (p3);
     \draw[thick] (p1) -- (p4);
     \draw[thick] (p2) -- (p6);
     \draw[thick] (p5) -- (p3);
     \draw[thick] (p6) -- (p4);
     \draw[thick] (p6) -- (p5);
     \draw[thick] (p5) -- (p4);
       \node [rectangle, draw=white, fill=white] (b) at (0,-2.2) {$\quad$};
    \node [rectangle, draw=white, fill=white] (b) at (1,-2.4) {(a)};
        \end{tikzpicture}
          \hspace{0.1cm}
     \begin{tikzpicture}[very thick,scale=1]
\tikzstyle{every node}=[circle, draw=black, fill=white, inner sep=0pt, minimum width=5pt];
   \path (0,-1.2) node (p3) [label = right: $3$] {} ;
   \path (0,0.2) node (p2) [label = right: $2$] {} ;
   \path (-0.5,-0.2) node (p1) [label = right: $1$] {} ;
\path[thick]
(p1) edge (p3);
\path[thick]
(p2) edge (p3);
\path[thick]
(p2) edge (p1);
\path[thick]
(p1) edge [loop left,->, >=stealth,shorten >=2pt,looseness=26] (p1);
\path[thick]
(p3) edge [->,bend right=22] (p2);
\node [rectangle, draw=white, fill=white] (b) at (0.54,-0.3) {$C_2$};
\node [rectangle, draw=white, fill=white] (b) at (-0.9,-0.6) {$C_2$};

\node [rectangle, draw=white, fill=white] (b) at (-0.2,-2.4) {(b)};
\end{tikzpicture}
       \hspace{0.5cm}
    \begin{tikzpicture}[very thick,scale=1]
\tikzstyle{every node}=[circle, draw=black, fill=white, inner sep=0pt, minimum width=5pt];
        \path (0,0) node (p1) [label = left: $1$] {} ;
       \path (0,-1.1) node (p2) [label = left: $2$] {} ;
   \path (2.5,-1.8) node (p3) [label = right: $3$] {} ;
   \path (2.5,-0.7) node (p4) [label = right: $4$] {} ;
   \draw[thick] (p1) -- (p4);
      \draw[thick] (p3) -- (p4);
     \draw[thick] (p2) -- (p3);
      \draw[thick] (p2) -- (p1);
\node [rectangle, draw=white, fill=white] (a) at (1.25,-0.7) {};
\node [rectangle, draw=white, fill=white] (b) at (1.25,-2.4) {(c)};
         \end{tikzpicture}
               \hspace{0.1cm}
\begin{tikzpicture}[very thick,scale=1]
   \tikzstyle{every node}=[circle, draw=black, fill=white, inner sep=0pt, minimum width=5pt];
   \path (0,-1.5) node (p1) [label = below: $1$] {} ;
   \path (0,-0.1) node (p2) [label = above: $2$] {} ;
  \draw[thick] (p1) -- (p2);
\path[thick]
(p1) edge [->,bend right=22] (p2);
\node [rectangle, draw=white, fill=white] (b) at (0.54,-0.8) {$C_2$};
\node [rectangle, draw=white, fill=white] (b) at (0,-2.4) {(d)};
\end{tikzpicture}\end{center}
\vspace{-0.3cm}
\caption{$\mathcal{C}_2$-symmetric graphs ((a), (c)) and their quotient $\mathcal{C}_2$-gain graphs ((b), (d)), where $\mathcal{C}_2$ denotes half-turn symmetry.}
\label{c2gaingraphs}
\end{figure}

The map $c:G\to H$ defined by $c(xv)=v$ and $c(\{xu,x\psi (e) v\})=(u,v)$ is called the \emph{covering map}. In order to avoid confusion, throughout the paper, a vertex or an edge in a quotient gain graph $H$
is denoted with the mark tilde (e.g., $\tilde{v}$ or $\tilde{e}$) and the vertex and edge set of $H$ is denoted by $\tilde{V}$ and  $\tilde{E}$, respectively.
Then the fiber $c^{-1}(\tilde{v})$ of a vertex $\tilde{v}\in \tilde{V}$ and the fiber $c^{-1}(\tilde{e})$
of an edge $\tilde{e}\in \tilde{E}$
coincide with a vertex orbit and an edge orbit, respectively, in $G$. For simplicity, the gain $\psi(\te)$ of an edge $\te \in \tilde{E}$ will sometimes also be denoted by $\psi_{\te}$.

Finally, a graph $\hat G=(I,P;E)$, where $I$ is the set of inner vertices and $P$ is the set of pinned vertices, is called \SG-symmetric  (with respect to $\theta$) if for every $x\in \mathcal{S}$, $\theta(x)$ maps pinned vertices to pinned vertices and inner vertices to inner vertices. The construction of the
quotient \SG-gain graph of $\hat G$ is of course completely analogous to the construction described above.

\subsection{Symmetric frameworks and orbit matrices}

In this subsection we introduce the basic terminology for symmetric frameworks and summarize the key results concerning `symmetry-forced' rigidity of frameworks. We begin with a discussion of `unpinned' symmetric frameworks.

A \emph{symmetry operation} of a framework $(G,\bp)$ in $\mathbb{R}^d$, where $G=(V,E)$, is an isometry $x$ of $\mathbb{R}^d$ such that for some $\alpha_x\in \textrm{Aut}(G)$, we have
$x\big(\bp(v))=\bp({\alpha_x(v)})\, \textrm{for all } v\in V.$ The set of all symmetry operations of a framework $(G,\bp)$ in $\mathbb{R}^d$ forms a group under composition, called the \emph{point group} of $(G,\bp)$. Clearly, we may assume wlog that the point group of a framework is always a \emph{symmetry group}, i.e., a subgroup of the orthogonal group $O(\mathbb{R}^{d})$.


 We use the Schoenflies notation for the symmetry operations and symmetry
groups in dimensions 2 and 3 considered in this paper, as this is one of the standard notations in
the literature about symmetric structures (see \cite{bishop,FGsymmax,BS1,BSWWorbit}, for example).  The relevant groups in this paper are $\mathcal{C}_s$, $\mathcal{C}_n$ and $\mathcal{C}_{nv}$. $\mathcal{C}_s$ is a group of order $2$ generated by a single reflection, $\mathcal{C}_n$, $n\geq 1$, is a cyclic group
generated by a rotation $\mathcal{C}_n$ about the origin (in the plane) or an axis through the origin  (in $3$-space) by an angle of $\frac{2\pi}{n}$, and $\mathcal{C}_{nv}$ is a dihedral group that is generated by a rotation $\mathcal{C}_n$ and
a reflection (whose reflectional plane contains the rotational axis of $\mathcal{C}_n$ in $3$-space).

Let $\mathcal{S}$ be an abstract group, and   $G=(V,E)$ be an  $\mathcal{S}$-symmetric graph with respect to an action $\theta:\mathcal{S}\rightarrow \textrm{Aut}(G)$.
Suppose also that $\mathcal{S}$ acts on $\mathbb{R}^d$ via the homomorphism $\tau:\mathcal{S}\rightarrow O(\mathbb{R}^d)$.
Then we say that a framework $(G,\bp)$ is \emph{$\mathcal{S}$-symmetric} (with respect to $\theta$ and $\tau$) if
\begin{equation}
\label{eq:symmetric_func}
\tau(x) (\bp(v))=\bp(\theta(x) v) \qquad \text{for all } x\in \mathcal{S} \text{ and all } v\in V.
\end{equation}
Note that if $(G,\bp)$ is $\mathcal{S}$-symmetric, then the point group of $(G,\bp)$ is either equal to $\tau(\mathcal{S})$ or contains $\tau(\mathcal{S})$ as a subgroup \cite{BS1}.

Let $H=(\tilde{V},\tilde{E})$ be the quotient graph of $G$ with the covering map $c:G\rightarrow H$. Then it is convenient to fix a representative vertex $v$ of each vertex orbit $\mathcal{S} v=\{xv\colon x\in \mathcal{S}\}$,
and define the \emph{quotient} of $\bp$ to be $\tilde{\bp}:\tilde{H}\rightarrow \mathbb{R}^d$,
so that there is a one-to-one correspondence between $\bp$ and $\tilde{\bp}$ given by
$\bp(v)=\tilde{\bp}(c(v))$ for each representative vertex $v$.

For a point group $\mathcal{S}$ in $O(\mathbb{R}^{d})$, let $\mathbb{Q}_{\mathcal{S}}$ be the field
 generated by $\mathbb{Q}$ and the entries of the matrices in $\mathcal{S}$.
We say that $\bp$ (or $\tilde{\bp}$) is \emph{\SG-generic}
if the set of coordinates of the image of $\tilde{\bp}$ is algebraically independent over $\mathbb{Q}_{\mathcal{S}}$.
Note that this definition does not depend on the choice of representative vertices. An $\mathcal{S}$-symmetric framework $(G,\bp)$ is called
\emph{$\mathcal{S}$-generic} if $\bp$ is $\mathcal{S}$-generic (see also \cite{BS1}).

An infinitesimal motion $\bu:V\rightarrow \mathbb{R}^d$ of an $\mathcal{S}$-symmetric framework $(G,\bp)$ (with respect to $\theta$ and $\tau$) is called \emph{fully $\mathcal{S}$-symmetric} if
\begin{equation}
\tau(x)\bu(v)=\bu(\theta(x)v) \qquad \textrm{ for all } v\in V \textrm{ and } x\in\mathcal{S},
\end{equation}
i.e., if the velocity vectors of $\bu$ satisfy the same symmetry constraints as the joints of $(G,\bp)$ (see also Figure~\ref{fulsym}). Similarly, a self-stress $\mathbf{\omega}$ of $(G,\bp)$ is called \emph{fully $\mathcal{S}$-symmetric} if $\omega_e=\omega_f$ for all edges $e$ and $f$ belonging to the same edge orbit
under the action of $\theta$ (see also \cite{BSWWorbit,schtan}, for example).

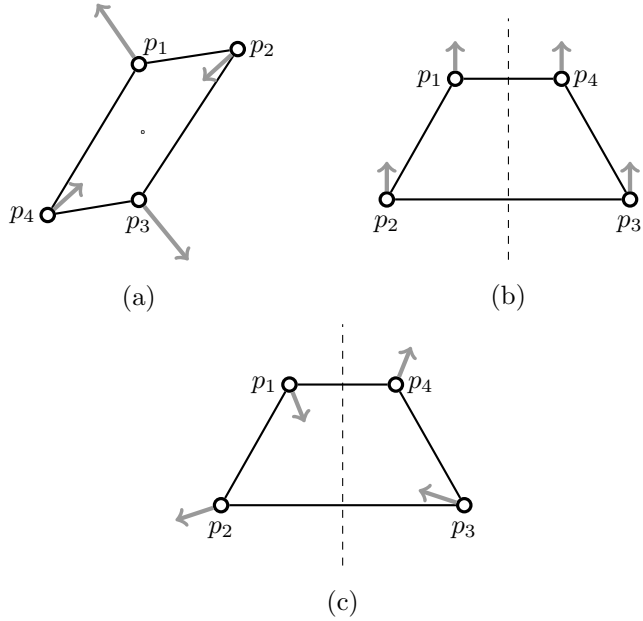
\begin{figure}[htp]
\begin{center}
\begin{tikzpicture}[rotate=90, very thick,scale=1]
\tikzstyle{every node}=[circle, draw=black, fill=white, inner sep=0pt, minimum width=5pt];
    \path (0.1,1.2) node (p1) [label = left: $p_{4}$] {} ;
    \path (2.1,0) node (p4) [label = above right: $p_{1}$]{} ;
    \path (2.3,-1.3) node (p3) [label = right: $p_{2}$] {} ;
     \path (0.3,0) node (p2) [label = below: $p_{3}$] {} ;
        \draw[thick] (p1) -- (p4);
      \draw[thick] (p3) -- (p4);
     \draw[thick] (p2) -- (p3);
      \draw[thick] (p2) -- (p1);
            \draw [ultra thick, ->, black!40!white](p1) -- (0.52,0.74);
      \draw [ultra thick, ->, black!40!white](p3) -- (1.88,-0.84);
      \draw [ultra thick, ->, black!40!white](p2) -- (-0.5,-0.65);
      \draw [ultra thick, ->, black!40!white](p4) -- (2.9,0.55);
      \filldraw[fill=black, draw=black]
    (1.2,-0.05) circle (0.004cm);
     \node [draw=white, fill=white] (b) at (-1,-0) {(a)};
              \end{tikzpicture}
 \hspace{1cm}
            \begin{tikzpicture}[very thick,scale=1]
\tikzstyle{every node}=[circle, draw=black, fill=white, inner sep=0pt, minimum width=5pt];
    \path (-0.7,0.8) node (p1) [label = left: $p_{1}$] {} ;
    \path (0.7,0.8) node (p4) [label = right: $p_{4}$]{} ;
    \path (-1.6,-0.8) node (p2) [label = below: $p_{2}$] {} ;
     \path (1.6,-0.8) node (p3) [label = below: $p_{3}$] {} ;
      \draw[thick] (p1) -- (p4);
    \draw[thick] (p1) -- (p2);
    \draw[thick] (p3) -- (p4);
    \draw[thick] (p2) -- (p3);
     \draw [dashed, thin] (0,-1.6) -- (0,1.6);
     \draw [ultra thick, ->, black!40!white] (p1) -- (-0.7,1.3);
      \draw [ultra thick, ->, black!40!white] (p4) -- (0.7,1.3);
      \draw [ultra thick, ->, black!40!white] (p2) -- (-1.6,-0.3);
      \draw [ultra thick, ->, black!40!white] (p3) -- (1.6,-0.3);
      \node [draw=white, fill=white] (b) at (0,-2.1) {(b)};
        \end{tikzpicture}
                \hspace{1cm}
                \begin{tikzpicture}[very thick,scale=1]
\tikzstyle{every node}=[circle, draw=black, fill=white, inner sep=0pt, minimum width=5pt];
    \path (-0.7,0.8) node (p1) [label = left: $p_{1}$]  {} ;
    \path (0.7,0.8) node (p4)[label = right: $p_{4}$] {} ;
    \path (-1.6,-0.8) node [label = below: $p_{2}$](p2)  {} ;
     \path (1.6,-0.8) node [label = below: $p_{3}$](p3)  {} ;
      \draw[thick] (p1) -- (p4);
    \draw[thick] (p1) -- (p2);
    \draw[thick] (p3) -- (p4);
    \draw[thick] (p2) -- (p3);
     \draw [dashed, thin] (0,-1.6) -- (0,1.6);
     \draw [ultra thick, ->, black!40!white] (p1) -- (-0.5,0.3);
      \draw [ultra thick, ->, black!40!white] (p4) -- (0.9,1.3);
      \draw [ultra thick, ->, black!40!white] (p2) -- (-2.2,-1);
      \draw [ultra thick, ->, black!40!white] (p3) -- (1,-0.6);
      \node [draw=white, fill=white] (b) at (0,-2.1) {(c)};
        \end{tikzpicture}
       \end{center}
\vspace{-0.3cm}
\caption{Infinitesimal motions of frameworks in the plane: (a) a fully $\mathcal{C}_2$-symmetric non-trivial infinitesimal motion; (b) a fully $\mathcal{C}_s$-symmetric trivial infinitesimal motion; (c) a non-trivial infinitesimal motion which is not  fully $\mathcal{C}_s$-symmetric, but `anti-symmetric' wtih respect to $\mathcal{C}_s$.}
\label{fulsym}
\end{figure}

We say that $(G,\bp)$ is \emph{\SG-symmetric (infinitesimally) rigid} if every
fully $\mathcal{S}$-symmetric infinitesimal motion of $(G,\bp)$ is trivial (i.e. if it corresponds to a translation or rotation (or a combination of those) of $(G,\bp)$) \cite{BSWWorbit,schtan}. Moreover, the framework $(G,\bp)$ is called \emph{\SG-isostatic} if it is minimally \SG-symmetric infinitesimally rigid, i.e., if  $(G,\bp)$ is \SG-symmetric infinitesimally rigid and has no non-zero \SG-symmetric self-stress.
To simplify the detection of fully $\mathcal{S}$-symmetric motions and self-stresses of $(G,\bp)$,
the orbit rigidity matrix  of  $(G,\bp)$ was introduced in \cite{BSWWorbit}.

\begin{Definition} [Schulze and Whiteley~\cite{BSWWorbit}]\label{orbitmatrixdef}
Let $(G,\bp)$ be an $\mathcal{S}$-symmetric framework with respect to $\theta$ and $\tau$, where $G=(V,E)$ and $\theta$ acts freely on $V$. Further, let $(H,\psi)$ be the quotient $\mathcal{S}$-gain graph of $G$, where $H=(\tilde{V},\tilde{E})$. For each edge $\tilde{e}=(\tilde{u},\tilde{v})\in \tilde{V}$, the \emph{orbit (rigidity) matrix} $\mathcal{O}(H,\psi,\tilde{\bp})$ of $(G,\bp)$
 has the following corresponding ($d|\tilde{V}|$-dimensional) row vector:
\begin{equation}\renewcommand{\arraystretch}{0.8} \label{orbitmatrixrow}
        \bordermatrix{& & \tilde{u} & &  \tilde{v} &  \cr & 0\dots0 & \big(\tilde{\bp}(\tilde{u})-\tau(\psi_{\tilde{e}})\tilde{\bp}(\tilde{v})\big) & 0\dots 0 & \big(\tilde{\bp}(\tilde{v})-\tau(\psi_{\tilde{e}})^{-1}\tilde{\bp}(\tilde{u})\big) & 0\dots 0}\textrm{,}
    \end{equation}
where each vector is assumed to be transposed. If $\tilde{e}$ is a loop at $\tilde{u}$, then $\mathcal{O}(H,\psi,\tilde{\bp})$  has the following corresponding ($d|\tilde{V}|$-dimensional) row vector:
\begin{equation}\renewcommand{\arraystretch}{0.8} \label{orbitmatrixrow}
        \bordermatrix{& & \tilde{u} & &  \tilde{v} &  \cr & 0\dots0 & \big(2\tilde{\bp}(\tilde{u})-\tau(\psi_{\tilde{e}})\tilde{\bp}(\tilde{v})- \tau(\psi_{\tilde{e}})^{-1}\tilde{\bp}(\tilde{u})  \big) & 0\dots 0 & 0& 0\dots 0}\textrm{,}
    \end{equation}
\end{Definition}

See also Example~\ref{ex:orbitmatrix} in Section~\ref{sec:pinsymfw} for an example of an  orbit matrix (for a pinned framework).

\begin{Remark} \label{rem:fixorb}
When the action $\theta$ is not free on the vertices of $G$, the number of columns in  $\mathcal{O}(H,\psi,\tilde{\bp})$, corresponding to vertices that are fixed by non-trivial symmetry operations, are reduced accordingly. Specifically, let $\bp(u)$ be a joint of the  $\mathcal{S}$-symmetric framework $(G,\bp)$ (with respect to $\theta$ and $\tau$), let $x$ be a symmetry operation in $\mathcal{S}$ which
fixes $\bp(u)$ (i.e., $ \tau(x)(\bp(u))=\bp(u)$), and
let  $F_x$ be the linear subspace of
$\mathbb{R}^d$ which consists of all points $a\in \mathbb{R}^d$ with $\tau(x)(a) = a$. (In the applied sciences, the space $F_x$ is sometimes also referred to as the \emph{symmetry element of $x$} \cite{bishop}.) Then the joint $\bp(u)$ of any
$\mathcal{S}$-symmetric framework $(G,\bp)$ (with respect to $\theta$ and $\tau$) must lie in the linear subspace $$U_{\bp(u)}=\bigcap_{x\in \mathcal{S}: \, \tau(x)(\bp(u))=\bp(u)} F_x.$$
Therefore, the number of columns corresponding to $\tilde{u}$ in the orbit matrix is reduced from $d$ to $\text{dim} U_{\bp(u)}$. This is achieved by multiplying the corresponding $d$-dimensional row vectors in $\mathcal{O}(H,\psi,\tilde{\bp})$ with the $d\times \text{dim} U_{\bp(u)}$ matrix whose columns are the coordinates of the basis vectors of $U_{\bp(u)}$ relative to the canonical basis (see also \cite{BSWWorbit} for details).

For example, if a joint $\bp(u)$ of a $\mathcal{C}_s$-symmetric framework $(G,\bp)$ in $3$-space is fixed by the reflection $s$ in $\mathcal{C}_s$, and the symmetry element (reflection plane) $F_s$ of $s$ is the $xy$-plane, then the orbit matrix has only two columns corresponding to $\bp(u)$ (since $F_{id}\cap F_s=\mathbb{R}^3\cap F_s=F_s$) and the two entries of each row are obtained by deleting the third coordinate of the corresponding row vectors $\tilde{\bp}(\tilde{u})-\tau(\psi_{\tilde{e}})\tilde{\bp}(\tilde{v})$ in (\ref{orbitmatrixrow}).
\end{Remark}

The following result summarizes the key properties of the orbit matrix.

\begin{Theorem}[\cite{BSWWorbit,schtan}] \label{thm:orbitmatrxprop}
Let $(G,\bp)$ be an $\mathcal{S}$-symmetric framework (with respect to $\theta$ and $\tau$) and let $(H,\psi)$ be the quotient \SG-gain graph of $G$. Then  the solutions to $\mathcal{O}(H,\psi,\tilde{\bp})u=0$ are isomorphic to the space of fully  $\mathcal{S}$-symmetric infinitesimal motions of $(G,\bp)$. Moreover, the solutions to $\omega^T \mathcal{O}(H,\psi,\tilde{\bp}) = 0$ are isomorphic to the space of fully  $\mathcal{S}$-symmetric self-stresses of  $(G,\bp)$.
\end{Theorem}

We say that $(G,\bp)$ is \emph{$\mathcal{S}$-regular} if the orbit matrix  $\mathcal{O}(H,\psi,\tilde{\bp})$ has
maximal rank  among all $\mathcal{S}$-symmetric realisations of $G$ (see also \cite{BSWWorbit,schtan}). Clearly, all \SG-regular realisations of a given graph share the same \SG-symmetric infinitesimal rigidity properties. Note that if a framework is $\mathcal{S}$-generic, then it is clearly also $\mathcal{S}$-regular.
Moreover, it was shown in \cite{BS6,KG1} that for $\mathcal{S}$-regular frameworks, the existence of a non-trivial fully $\mathcal{S}$-symmetric infinitesimal motion guarantees the existence of a symmetry-preserving mechanism. Thus, for $\mathcal{S}$-regular frameworks, \SG-symmetric infinitesimal
rigidity and \SG-symmetric (finite) rigidity are equivalent, and hence a graph $G$ is called \emph{\SG-rigid (\SG-isostatic)} if there exists a \SG-regular realisation of $G$ which is \SG-symmetric (infinitesimally) rigid (\SG-isostatic).

For \SG-rigidity in the plane, there are Laman type theorems for all groups except the `even order' dihedral groups of the form $\mathcal{C}_{2nv}$, $n\geq 1$ \cite{M&T,jkt}.
To state these theorems we need the following definitions (see also \cite{jkt,schtan}).

Let $(H,\psi)$ be an $\mathcal{S}$-gain graph with $H=(\tilde{V},\tilde{E})$.
A cycle in $H$ is called \textrm{balanced} if the product of its edge gains is equal to the identity.
(If $\mathcal{S}$ is an additive group, we take the sum instead of the product.)
More precisely, a cycle of the form $\tilde{v}_1,\tilde{e}_1,\tilde{v}_2,\tilde{e}_2,\tilde{v}_3,\dots,\tilde{v}_k,\tilde{e}_k,\tilde{v}_1$, is balanced if $\Pi_{i=1}^k \psi(\tilde{e}_i)^{\textrm{sign}(\tilde{e}_i)}=id$,
where $\textrm{sign}(\tilde{e}_i)=1$ if $\tilde{e}_i$ is directed from $\tilde{v}_i$ to $\tilde{v}_{i+1}$, and $\textrm{sign}(\tilde{e}_i)=-1$  otherwise.

We say that an edge subset $F\subseteq \tilde{E}$ is \emph{balanced} if all cycles in $F$ are balanced;
otherwise it is called \emph{unbalanced}.

\begin{Definition}\label{def:gainsparse}
Let $(H,\psi)$ be an $\mathcal{S}$-gain graph with  $H=(\tilde{V},\tilde{E})$ and let  $k, \ell, m$ be nonnegative integers with $m\leq \ell$.
$(H,\psi)$ is called $(k,\ell,m)$-gain-sparse if
\begin{itemize}
\item $|F|\leq k|V(F)|-\ell$ for any nonempty balanced $F\subseteq \tilde{E}$;
\item $|F|\leq k|V(F)|-m$ for any nonempty $F\subseteq \tilde{E}$.
\end{itemize}
 A  $(k,\ell,m)$-gain-sparse graph $(H,\psi)$ which satisfies $|\tilde E|=k|\tilde V|-m$ is called \emph{$(k,\ell,m)$-gain-tight}.
\end{Definition}

\begin{Example} Consider, for example, the gain-graph $H$ shown in Figure \ref{c2gaingraphs} (b).   $H=(\tilde{V},\tilde{E})$ is $(2,3,1)$-gain tight, since  $|\tilde E|=5=2|\tilde V|-1$ and all the conditions in Definition~\ref{def:gainsparse}
are satisfied. However, $H$ is not $(2,3,2)$-gain sparse, for example, since there exists a loop at vertex $1$ with gain $C_2$, and for this subgraph, we have  $|F|=1>0= 2|V(F)|-2$.
\end{Example}

For $\mathcal{S}$-symmetric frameworks in the plane, where the action $\theta$ is free on the vertex set, we have the following elegant characterizations of \SG-symmetric rigid graphs for the groups $\mathcal{C}_s$ and $\mathcal{C}_n$.

\begin{Theorem}[Malestein and Theran~\cite{M&T}, Jord{\'a}n et al. \cite{jkt}]
\label{thm:symmetry_reflection}
Let $\mathcal{S}$ be $\mathcal{C}_s$ or $\mathcal{C}_n$ for some $n\geq 2$, $\tau:\mathbb{Z}/n\mathbb{Z} \rightarrow \mathcal{S}$ be a homomorphism,
$G=(V,E)$ be an $\mathcal{S}$-symmetric graph with $\theta:\mathcal{S}\rightarrow \textrm{Aut}(G)$, where $\theta$ acts freely on $V$,
and $(G,\bp)$ be a $2$-dimensional $\mathcal{S}$-regular  framework with respect to $\theta$ and $\tau$.
Then $(G,\bp)$ is $\mathcal{S}$-isostatic if and only if
the quotient $\mathcal{S}$-gain graph $(H,\psi)$ is  $(2,3,1)$-gain-tight.
\end{Theorem}

Note that the condition $|F|\leq 2|V(F)|-1$ for any nonempty $F\subseteq \tilde{E}$ in the quotient graph  $H=(\tilde{V},\tilde{E})$ in Theorem~\ref{thm:symmetry_reflection} reflects the fact that there is only a $1$-dimensional space of \SG-symmetric trivial infinitesimal motions for the groups $\mathcal{C}_s$ and $\mathcal{C}_n$ in dimension 2.
In general, for any point group $\mathcal{S}$ in dimension 2 or 3, the dimension of the space of trivial \SG-symmetric ininitesimal motions can easily be read off from the character table of $\mathcal{S}$ \cite{bishop}.

A similar, but slightly more complicated characterization of \SG-rigid graphs for dihedral groups $\mathcal{S}$  of the form  $\mathcal{C}_{(2n+1)v}$, $n\geq 1$, was also established in \cite{jkt}. However,
 for dihedral groups of the form $\mathcal{C}_{2nv}$, $n\geq 1$, a combinatorial
characterization of \SG-rigid graphs in the plane is not known. For example, it was shown in \cite{jkt} that Bottema's mechanism (a realisation of the complete bipartite
graph $K_{4,4}$ with $\mathcal{C}_{2v}$ symmetry in the plane) is falsely predicted to be $\mathcal{C}_{2v}$-symmetric rigid by the matroidal counts for the orbit matrix.

\begin{Remark}
For \SG-symmetric graphs, where the action $\theta$ is not free  on the vertex set, no combinatorial characterizations for \SG-rigidity have been derived yet. This is because for such graphs the structure of the orbit matrix and the corresponding combinatorial counts
 become significantly more messy (recall Remark~\ref{rem:fixorb}). However, in principle we do not expect any major new difficulties to arise when making the extension of Theorem~\ref{thm:symmetry_reflection} to symmetric graphs with non-free group actions.

 Clearly, in $3$-space, no combinatorial characterizations for \SG-rigidity are known, since
the problem of finding a combinatorial charaterization of rigid graphs (without symmetry) in dimensions $d\geq 3$ remains a long-standing open problem in discrete geometry \cite{W1}.
\end{Remark}

\section{Symmetric Pinned Frameworks} \label{sec:pinsymfw}

Based on the discussion in the previous section, we are now ready to introduce the key concepts of pinned \SG-isostatic graphs and \SG-Assur graphs.

\subsection{Basic definitions}
\label{subsec:basicdef}

Let $\hat G=(I,P;E)$ be a \SG-symmetric pinned graph (with respect to the action $\theta: \mathcal{S}\to \textrm{Aut}(G)$) and let $\hat H=(\hat{I},\hat{P};\hat{E})$ be the associated quotient \SG-gain graph.
Given a pinned \SG-symmetric  realisation of  $\hat G$ (with respect to the action $\theta$ and the homomorphism $\tau:\mathcal{S}\rightarrow \mathcal{O}(\mathbb{R}^d)$), we define the \emph{(pinned) orbit matrix} $\mathcal{O}_{pin}(\hat H,\psi,\tilde{\bp})$ of $(\hat G,\bp)$ as follows. There are $d$ columns for every inner vertex $\tilde{v}\in \hat{I}$ and no columns for any pinned vertex. For an edge $(\tilde{u},\tilde{v})$ with $\tilde{u},\tilde{v} \in \hat{I}$ there is a row in the pinned orbit matrix exactly like in the orbit matrix (see Definition \ref{orbitmatrixdef}). Moreover if an edge $(\tilde{u},\tilde{v})$ has $\tilde{u}\in \hat{I}$ and  $\tilde{v}\in \hat{P}$ then the corresponding row in $\mathcal{O}_{pin}(\hat H,\psi, \tilde{\bp})$ is:
\begin{displaymath}\renewcommand{\arraystretch}{0.8}
        \bordermatrix{  & &  \tilde{u} &  & \tilde{v} &  \cr & 0 \ldots 0 &  \big( \tilde{\bp}(\tilde{u})-\tau(\psi_{\tilde{e}})\tilde{\bp}(\tilde{v})\big) & 0  \ldots  0 & 0 & 0  \ldots  0}\textrm{.}
    \end{displaymath}

All solutions $U$ to $\mathcal{O}_{pin}(\hat H,\psi, \tilde{\bp})\times U^{T}=0$ are called \emph{pinned fully \SG-symmetric infinitesimal motions} of $(\hat G,\bp)$. If the only such motion is the zero motion then  $(\hat G,\bp)$  is said to be \emph{pinned \SG-symmetric infinitesimally rigid}. Equivalently,
$(\hat G,\bp)$ is pinned \SG-symmetric infinitesimally rigid if $\textrm{rank }\mathcal{O}_{pin}(\hat H,\psi, \tilde{\bp})=d|\hat{I}|$.

  A  \emph{pinned fully $\mathcal{S}$-symmetric self-stress} of $(\hat G,\bp)$ is a pinned self-stress $\mathbf{\omega}$ of  $(\hat G,\bp)$ with the property that  $\mathbf{\omega}_e=\mathbf{\omega}_f$ whenever $e$ and $f$ belong to the same edge orbit under the action of $\theta$. Note that it follows immediately from \cite[Theorem 8.3]{BSWWorbit} that there exists a one-to-one correspondence between the pinned fully $\mathcal{S}$-symmetric self-stresses of $(\hat G,\bp)$ and the row dependencies of  $\mathcal{O}_{pin}(\hat H,\psi, \tilde{\bp})$, i.e., the    solutions $\mathbf{\omega}$ to  $\mathbf{\omega}^T\mathcal{O}_{pin}(\hat H,\psi, \tilde{\bp})=0$.
The framework  $(\hat G,\bp)$ is called \emph{pinned \SG-independent} if the rows of  $\mathcal{O}_{pin}(\hat H,\psi, \tilde{\bp})$ are linearly independent and $(\hat G,\bp)$ is \emph{pinned \SG-isostatic} if  $(\hat G,\bp)$ is both pinned \SG-independent and pinned \SG-symmetric infinitesimally rigid.

Finally,   $(\hat G,\bp)$ is  called \emph{pinned $\mathcal{S}$-regular} if  $\mathcal{O}_{pin}(\hat H,\psi, \tilde{\bp})$ has maximal rank among all pinned \SG-symmetric realisations of $\hat G$, and  $(\hat G,\bp)$ is called \emph{pinned \SG-generic} if the set of the coordinates of the image of $\tilde\bp$ is algebraically independent over $\mathbb{Q}_{\mathcal{S}}$.

The following lemma summarises the definitions above for pinned \SG-regular realisations.

\begin{Lemma}
\label{lem:SGisostatic}
Let $\hat G$ be a $\mathcal{S}$-symmetric pinned graph (with respect to $\theta$ and $\tau$) and let $\hat H = (\hat I,\hat P;\hat E)$ be the corresponding quotient \SG-gain graph. Further, suppose the action $\theta$ is free on the vertices of $\hat G$. Then the following are equivalent:
\begin{enumerate}
\item There exists a pinned \SG-isostatic realisation of $\hat G$ in $d$-space;
\item Every \SG-regular pinned realisation of $\hat G$ in $d$-space is \SG-isostatic;
\item pinned \SG-regular  realisations of $\hat G$ are pinned \SG-rigid and $|\hat E|=d|\hat I|$;
\item pinned \SG-regular  realisations  of $\hat G$ are pinned \SG-independent and $|\hat E|=d|\hat I|$.
\end{enumerate}
\end{Lemma}

In particular, for any pinned \SG-isostatic framework $(\hat G,\bp)$, the orbit matrix $\mathcal{O}_{pin}(\hat H,\psi,\tilde{\bp})$ is square and invertible. (Note that this remains true even if $\theta$ does not act freely on the vertices of $\hat G$.)

Any \SG-symmetric graph $\hat G$  (or equivalently its quotient \SG-gain graph)  satisfying the equivalent conditions in Lemma~\ref{lem:SGisostatic} is \emph{pinned \SG-isostatic}. Analogous to the non-symmetric situation, a \emph{\SG-Assur graph} is a minimal pinned \SG-isostatic graph. Further,  a \SG-Assur graph is  \emph{strongly \SG-Assur} if the removal of any edge of its quotient \SG-gain graph (i.e., the removal of an edge orbit in the covering graph) puts all inner vertices of the covering graph into (a symmetry-preserving) motion.

\begin{Example} \label{ex:orbitmatrix} Consider the $2$-dimensional $\mathcal{C}_3$-symmetric (with respect to $\theta$ and $\tau$) pinned framework $(\hat G,\bp)$ and the corresponding quotient  $\mathcal{C}_3$-gain graph of $\hat G$ depicted in Figure~\ref{fig:desargues} (and Figure~\ref{fig:desargues2}). Let $\tau:\mathcal{C}_3\to O(\mathbb{R}^2)$ be the homomorphism defined by $\tau(C_3)=\left(\begin{array} {cc} -\frac{1}{2} & -\frac{\sqrt{3}}{2}\\ \frac{\sqrt{3}}{2} &  -\frac{1}{2} \end{array}\right)$. Suppose $p_w=(-2,3)$,  $p_u=(-1,2)$ and $p_v=(-\frac{\sqrt{3}}{4},\frac{1}{4})$.  Then the pinned orbit matrix $\mathcal{O}_{pin}(\hat H,\psi,\tilde{\bp})$ of  $(\hat G,\bp)$ is
\begin{displaymath}\bordermatrix{
                &\tilde u&\tilde w \cr
 (\tilde u,\tilde v) &(p_u-p_v)  & 0 \, 0 \cr
 (\tilde u,C_3^{-1}\tilde v) &(p_u-\tau(C_3^{-1})(p_v)) &0 \, 0 \cr
                (\tilde w,\tilde u)&(p_u-p_w) &  (p_w-p_u)\cr
                 (\tilde w,C_3\tilde w)& 0 \, 0 &(2p_w-\tau(C_3)(p_w)-\tau(C_3^{-1})(p_w) ) \cr
             }
\end{displaymath}
\begin{displaymath}
=\bordermatrix{
                 &&&& \cr
 &  -1+\frac{\sqrt{3}}{4} & \frac{7}{4} & 0 & 0\cr
                 &  -1-\frac{\sqrt{3}}{4} & \frac{7}{4}  & 0 & 0\cr
                 &1 & -1 & -1 & 1\cr
                 & 0 & 0 & -6 & 9\cr
                            }.
\end{displaymath}
$\mathcal{O}_{pin}(\hat H,\psi,\tilde{\bp})$ is a square matrix of full rank, and hence $\hat G$ is pinned $\mathcal{C}_3$-isostatic. However,  $\hat G$ is not $\mathcal{C}_3$-Assur.  The lower triangular block-decomposition of $\mathcal{O}_{pin}(\hat H,\psi,\tilde{\bp})$ (into two blocks)
corresponds to the $\mathcal{S}$-Assur decomposition of $\hat G$ (i.e., the decomposition of the graph into two components which are both $\mathcal{S}$-Assur) shown in Figure~\ref{fig:desargues} (c) (and Figure~\ref{fig:desargues2} (c)).
\end{Example}


\subsection{Counting conditions for pinned \SG-isostatic graphs}
\label{subsec:counting}

In this subsection we prove an analogue of \cite[Theorem 3.6]{3directed} giving necessary counting conditions for a pinned \SG-symmetric graph, under a group action that is free on the vertices, to be pinned \SG-isostatic.
We then consider, for plane symmetry groups, when these counts are sufficient; that is, for the groups in Theorem \ref{thm:symmetry_reflection} we prove analogues of \cite[Theorem 4]{SSW1}.

We begin with an observation which follows immediately from Lemma~\ref{lem:SGisostatic} and Theorem~\ref{thm:orbitmatrxprop}.

\begin{Theorem}\label{thm:Scounts3D}
Let \SG~be a symmetry group in dimension $d$, and let $\hat G$ be a pinned \SG-symmetric graph (with respect to an action $\theta$) with quotient \SG-gain graph  $\hat H = (\hat I,\hat P;\hat E)$. Further, let $triv_{\mathcal{S}}$ denote the dimension of the space of fully \SG-symmetric trivial infinitesimal motions. Suppose $\theta$ acts freely on the vertices of $\hat G$. Then, if $\hat G$ is pinned \SG-isostatic, the following hold:
\begin{itemize}
\item $|\hat E|=d |\hat I|$;
\item every subgraph  $H'=(I',P';E')$ of $\hat H$ satisfies  $| E'| \leq d |I'|$.
\item every subgraph of $\hat H$ with no pinned vertices is $(d,\binom{d+1}{2},triv_{\mathcal{S}})$-gain sparse.
\end{itemize}
\end{Theorem}

\begin{proof}
By the definition of pinned \SG-isostatic, we have $|\hat E|=d|\hat I|$.
If either of the other two conditions fails, then any $\mathcal{S}$-symmetric realisation of  $\hat G$ has a pinned fully $\mathcal{S}$-symmetric self-stress.
\end{proof}

For example, if \SG~is a symmetry group $\mathcal{C}_s$ in dimension $3$, then  $triv_{\mathcal{C}_s}=3$ and every subgraph of $\hat H$ with no pinned vertices must be $(3,6,3)$-gain sparse. Similarly, if \SG~is a symmetry group $\mathcal{C}_n$, $n\geq 2$, in dimension $3$, then  $triv_{\mathcal{C}_s}=2$ and every subgraph of $\hat H$ with no pinned vertices must be $(3,6,2)$-gain sparse (see also \cite{BSWWorbit,gsw, bishop}).

For the plane symmetry groups  in Theorem \ref{thm:symmetry_reflection} we have both necessary and sufficient conditions for a pinned symmetric graph to be \SG-isostatic.

\begin{Theorem}\label{thm:Scounts}
Let \SG~be a symmetry group $\mathcal{C}_s$ or $\mathcal{C}_n$, $n\geq 2$, in dimension $2$, and let $\hat G$ be a pinned \SG-symmetric graph (with respect to an action $\theta$) with quotient \SG-gain graph  $\hat H = (\hat I,\hat P;\hat E)$. Further, suppose $\theta$ acts freely on the vertices of $\hat G$. Then $\hat G$ is pinned \SG-isostatic if and only if the following hold:
\begin{itemize}
\item $|\hat E|=2 |\hat I|$;
\item every subgraph of $\hat H$ with no pinned vertices is $(2,3,1)$-gain sparse;
\item every subgraph $H'=(I',P';E')$ of $\hat H$  with $P'\neq \emptyset$ satisfies  $| E'| \leq 2 |I'|$.
\end{itemize}
\end{Theorem}

\begin{proof}
Replace the pinned vertices $\hat P$ of  $\hat H$ with the quotient \SG-gain graph of a non-pinned $\mathcal{S}$-isostatic graph with vertex set  $\hat P$ and edge set $F$. (For example, if  $\hat P$ contains only a single vertex $\tilde v$, then we attach a loop with a non-trivial gain  to $\tilde v$.) Let $H^*$ be the (non-pinned)  quotient \SG-gain graph with vertex set $\hat I \cup \hat P$ and edge set $\hat E \cup F$.  Then, by Theorem~\ref{thm:symmetry_reflection},
the covering graph of $H^*$ is \SG-isostatic if and only if the conditions in Theorem~\ref{thm:Scounts} are satisfied. This gives the result.
\end{proof}

Any pinned graph $\hat G$ satisfying the counts in the first part of Theorem \ref{thm:Scounts} is said to be \emph{pinned $(2,3,1)$-gain-tight}.

Note that we can easily obtain analogous necessary counts for  pinned \SG-symmetric graphs to be pinned \SG-isostatic in the case where the action $\theta$ is not free on the vertices. For example,  for a symmetry group $\mathcal{C}_n$ which acts freely on the inner vertices of a pinned  $\mathcal{C}_n$-symmetric graph $\hat G$, but not freely on the pinned vertices of $\hat G$  (i.e., any $\mathcal{C}_n$-symmetric framework  has a pinned vertex at the origin), we obtain the necessary condition that every subgraph $H'=(I',P';E')$ of the quotient $\mathcal{C}_n$-gain graph $\hat H$ which contains the pinned vertex that is fixed by $\mathcal{C}_n$ must  satisfy  $| E'| \leq 2 |I'|-1$. This is because rotations about the origin are the only fully $\mathcal{C}_n$-symmetric infinitesimal motions (i.e., $triv_{\mathcal{C}_n}=1$) and the fixed pinned vertex does not prevent such a trivial motion.

Similarly, the conditions in Theorem~\ref{thm:Scounts} remain unchanged for a group $\mathcal{C}_s$ in the plane which acts freely on the inner, but not freely on the pinned vertices of a $\mathcal{C}_s$-symmetric graph, since a pinned vertex removes all fully $\mathcal{C}_s$-symmetric infinitesimal motions (translations along the mirror), regardless of whether it is fixed by the reflection or not.


\section{Decomposing pinned \SG-isostatic  graphs}\label{sec:sassur}
\label{sec:decomp}

In subsection \ref{subsec:Assur} we recalled a description of the $d$-Assur graph decomposition of the pinned $d$-isostatic graphs that appeared in \cite{3directed}. In this section we extend the key results from $d$-Assur to \SG-Assur graphs and introduce the \SG-Assur decomposition via similar techniques. For any \SG-isostatic graph $\hat G$, the pinned orbit matrix $\mathcal{O}_{pin}(\hat H,\psi,\tilde{\bp})$ is square and invertible, and we are assured to get an appropriate set of directions on the associated gain graph (out-degree of a vertex is equal to the number of columns of that vertex in the pinned orbit matrix). Any orientation of the gain graph gives a unique strongly-connected directed graph decomposition, whose components with their outgoing edges are the \SG-Assur graphs. We will state this precisely, and also introduce strongly \SG-Assur graphs.

Before we state the main results, we outline the \SG-Assur decomposition and illustrate it on some examples, following the approach of $d$-Assur decompositions.

Let $\hat G$ be a pinned \SG-isostatic graph  with a free group action on the vertices and let $\hat H$ be  its quotient pinned \SG-gain graph. We first seek a minimal pinned \SG-isostatic subgraph (i.e. a \SG-Assur graph), where the ground is the bottom layer. This will be above the ground component and is then collapsed into the ground. We then seek another minimal pinned \SG-isostatic subgraph, collapsing it into the ground. This is then repeated unitl each vertex orbit belongs to some minimal pinned \SG-isostatic subgraph. This is the \SG-Assur decomposition. The \SG-Assur decomposition will be carried out via the decomposition into strongly connected components of the directed gain graph.

An \emph{\SG-directed orientation} of a pinned \SG-gain graph $\hat H$ is an assignment of directions to the edges of $\hat H$ such that every inner vertex has out-degree exactly dim$U_{\bp(u)}$ (recall Remark~\ref{rem:fixorb}) and every pinned vertex has out-degree exactly 0 (see Proposition~\ref{Proposition:dorientation}). Recall that when \SG~is free, then dim$U_{\bp(u)}$ is exactly $d$. The strongly connected components in $\hat H$, with its outgoing edges becoming pinned (i.e. the extended components) will correspond to the \SG-Assur decomposition of the \emph{\SG-directed orientation} of a pinned \SG-gain graph $\hat H$ (Theorem~\ref{thm:decomp}).

Recall that there is a one-one correspondence between the covering graph and the \SG-gain graph. This implies that there is a one-one correspondence between a \SG-Assur component of $\hat G$ and a \SG-Assur component of $\hat H$. In $\hat H$ such components are connected, while in $\hat G$ one component may consist of $|\SG|$ disconnected subgraphs. This bijection justifies the terminology `component' for such disconnected subgraphs.

\begin{Example} Consider the pinned framework shown in Figure~\ref{fig:desargues2} with $\mathcal{C}_3$-symmetry in the plane. Note that this is the same framework as in Figure~\ref{fig:desargues} (a). This framework is pinned isostatic in the plane as well as pinned $\mathcal{C}_3$-isostatic. The $\mathcal{C}_3$-Assur decomposition of the covering graph is shown in Figure~\ref{fig:desargues2} (b) and the corresponding decomposition of the quotient $\mathcal{C}_3$-gain graph is shown in Figure~\ref{fig:desargues2} (d). \end{Example}

\begin{figure}[ht]
    \begin{center}
  \subfigure[] {\includegraphics [width=.36\textwidth]{C3Ext}} \quad\quad\quad
    \subfigure[] {\includegraphics [width=.36\textwidth]{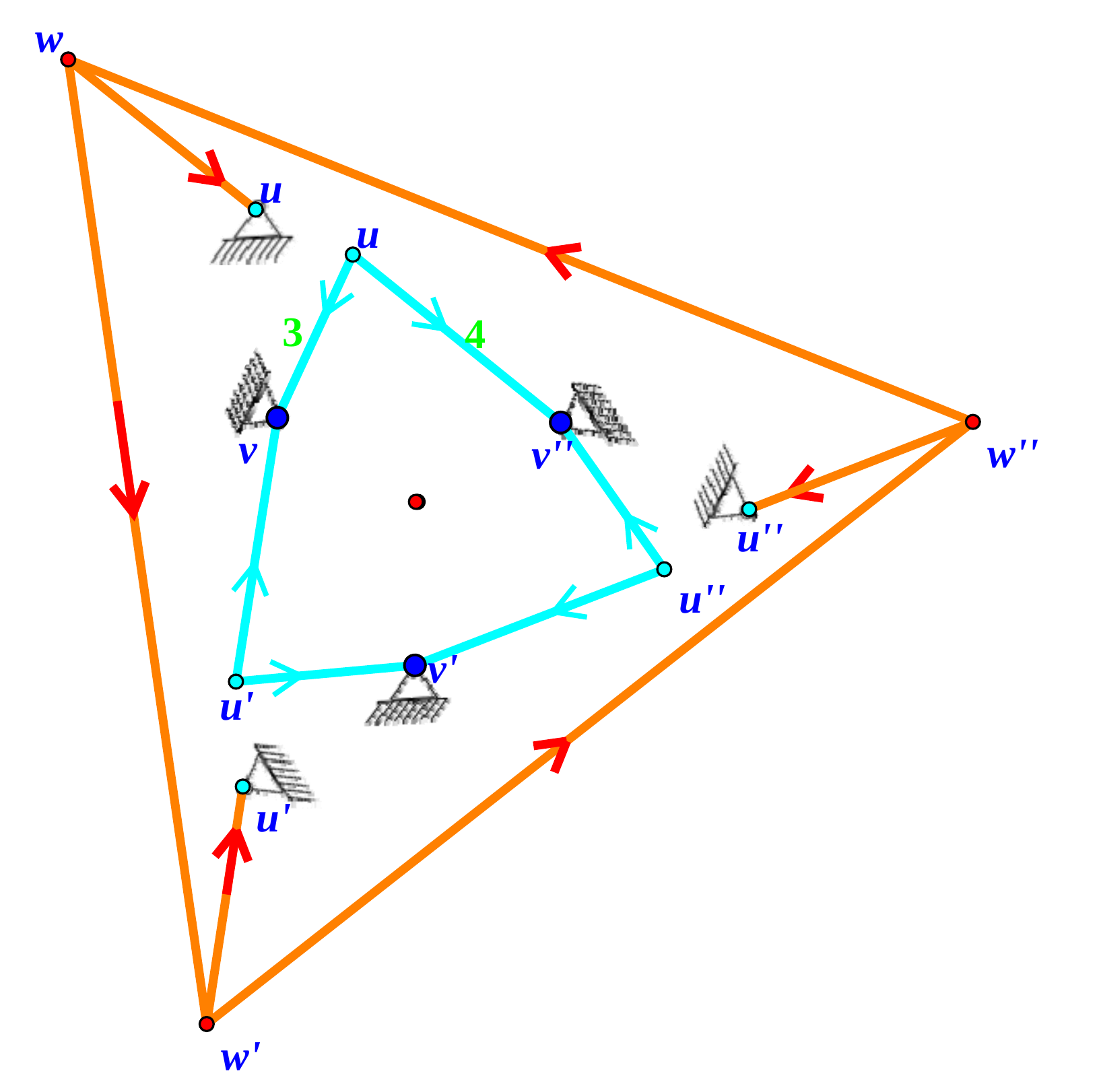}} \quad\quad
    \subfigure[] {\includegraphics [width=.16\textwidth]{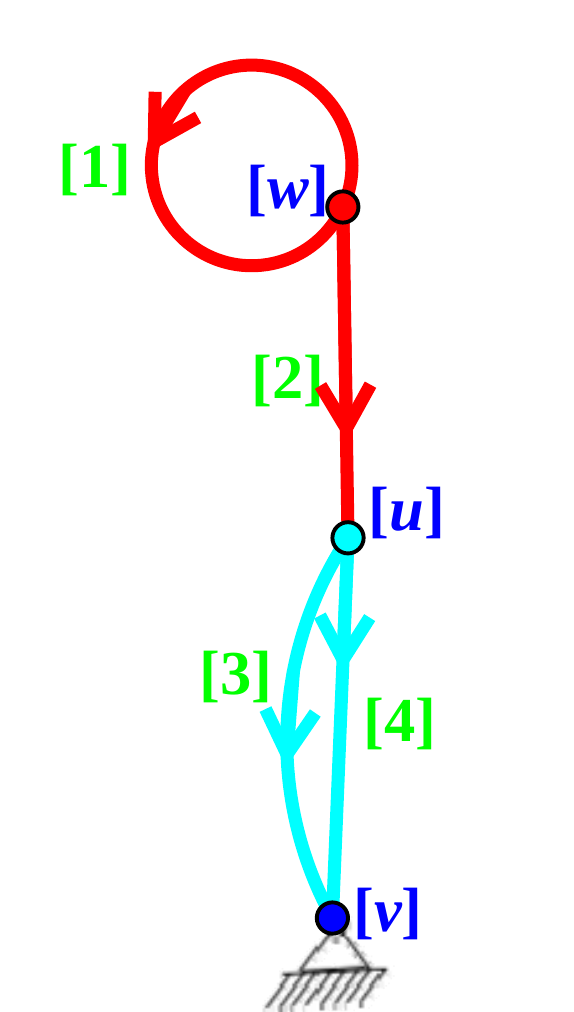}} \quad
    \subfigure[] {\includegraphics [width=.16\textwidth]{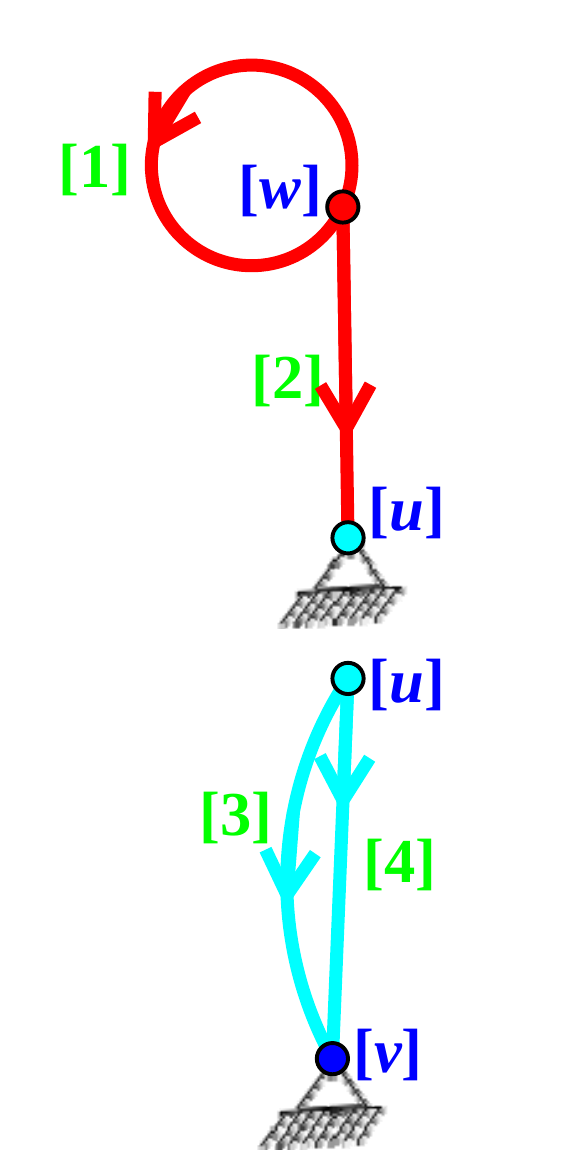}} \quad
    \subfigure[] {\includegraphics [width=.12\textwidth]{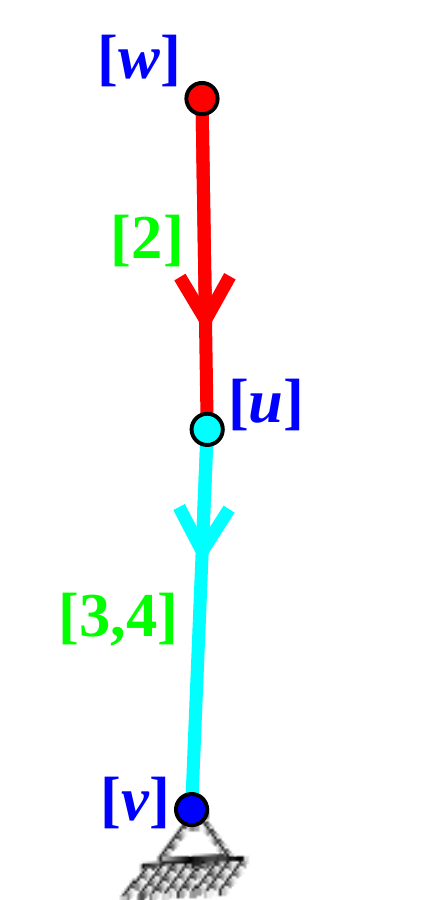}}
       \end{center}
    \caption{A $\mathcal{C}_3$-symmetric pinned isostatic framework  (a) with symmetric directions, the symmetry-adapted $\mathcal{C}_3$-Assur decomposition (b), the $2$-directed quotient $\mathcal{C}_3$-gain graph and its $\mathcal{C}_3$-Assur decomposition (c, d) with $\mathcal{C}_3$-Assur block graph (e). The square brackets in the gain graph correspond to the vertex (and edge) orbits.}
    \label{fig:desargues2}
    \end{figure}


\begin{Example}
Consider the pinned $2$-dimensional framework with $\mathcal{C}_s$ symmetry shown in Figure~\ref{fig:PlaneExample} (a). Unlike the previous example, the $\mathcal{C}_s$-Assur decomposition (shown in  Figure~\ref{fig:PlaneExample} (c)) is not a tree in this case. Note that the framework in Figure~\ref{fig:PlaneExample} (a) has the correct overall pinned count consisting of 28 edges and 14 inner vertices; however, this graph is not pinned $2$-isostatic (and not pinned $2$-rigid), as it has a generically dependent component (the subgraph induced by $A$, $B$, $C$, $A'$, $B'$, $C'$ is overbraced). Alternatively, if we ignore the vertices above this component, we can observe that this component is only attached to the ground via two edges (coloured in orange). Thus, without imposed symmetry this graph does not have a 2-Assur decomposition.
\end{Example}

In standard Assur decompositions, to apply the decomposition and to find the Assur components, the graph $\hat G$ must be pinned isostatic. However, for symmetric frameworks, even if the original underlying graph $\hat G$ is flexible or has redundancy we may still find the symmetry adapted Assur decomposition.



\begin{figure}[ht]
    \begin{center}
  \subfigure[] {\includegraphics [width=.25\textwidth]{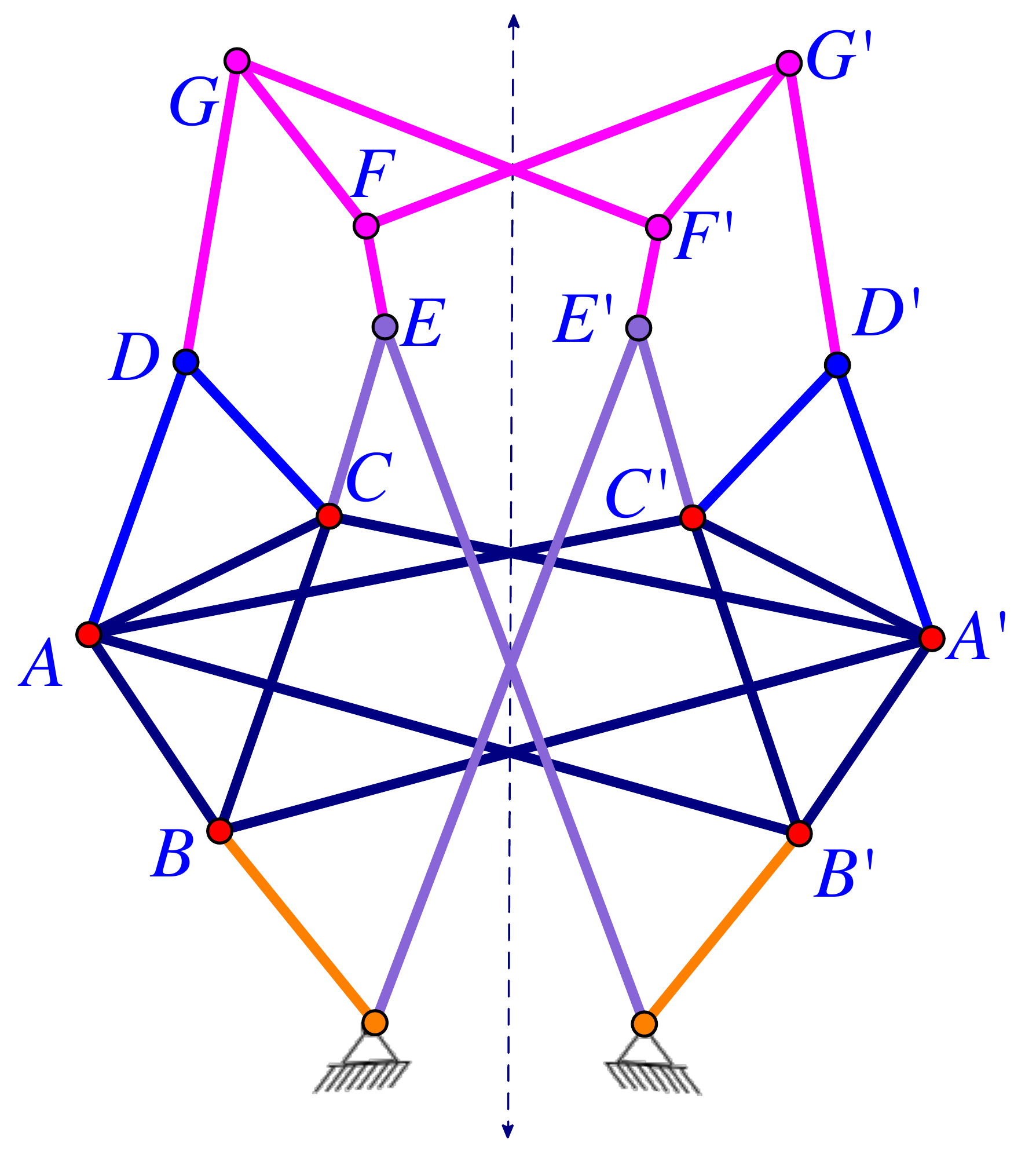}} \quad\quad
\subfigure[] {\includegraphics [width=.15\textwidth]{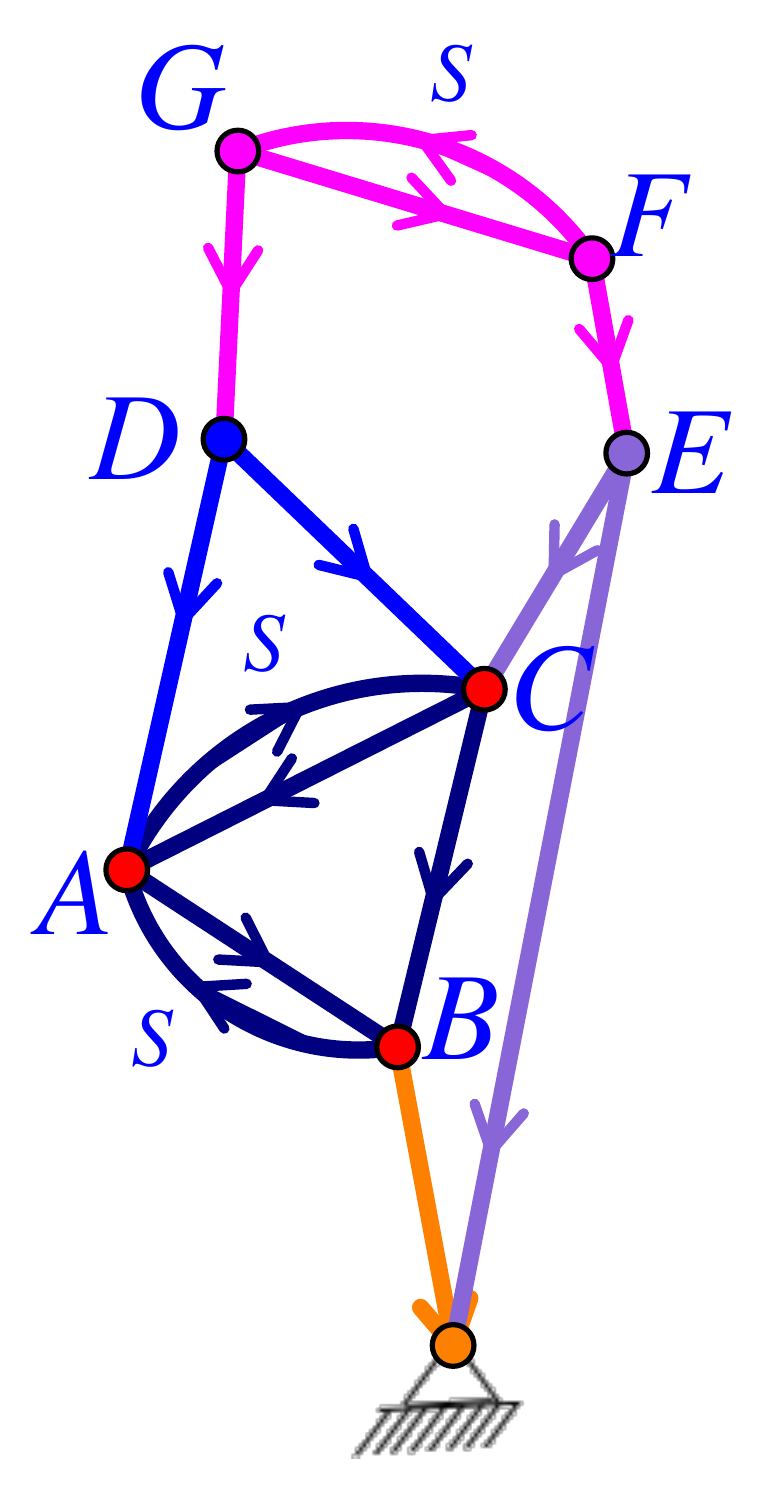}}\quad\quad
  \subfigure[] {\includegraphics [width=.15\textwidth]{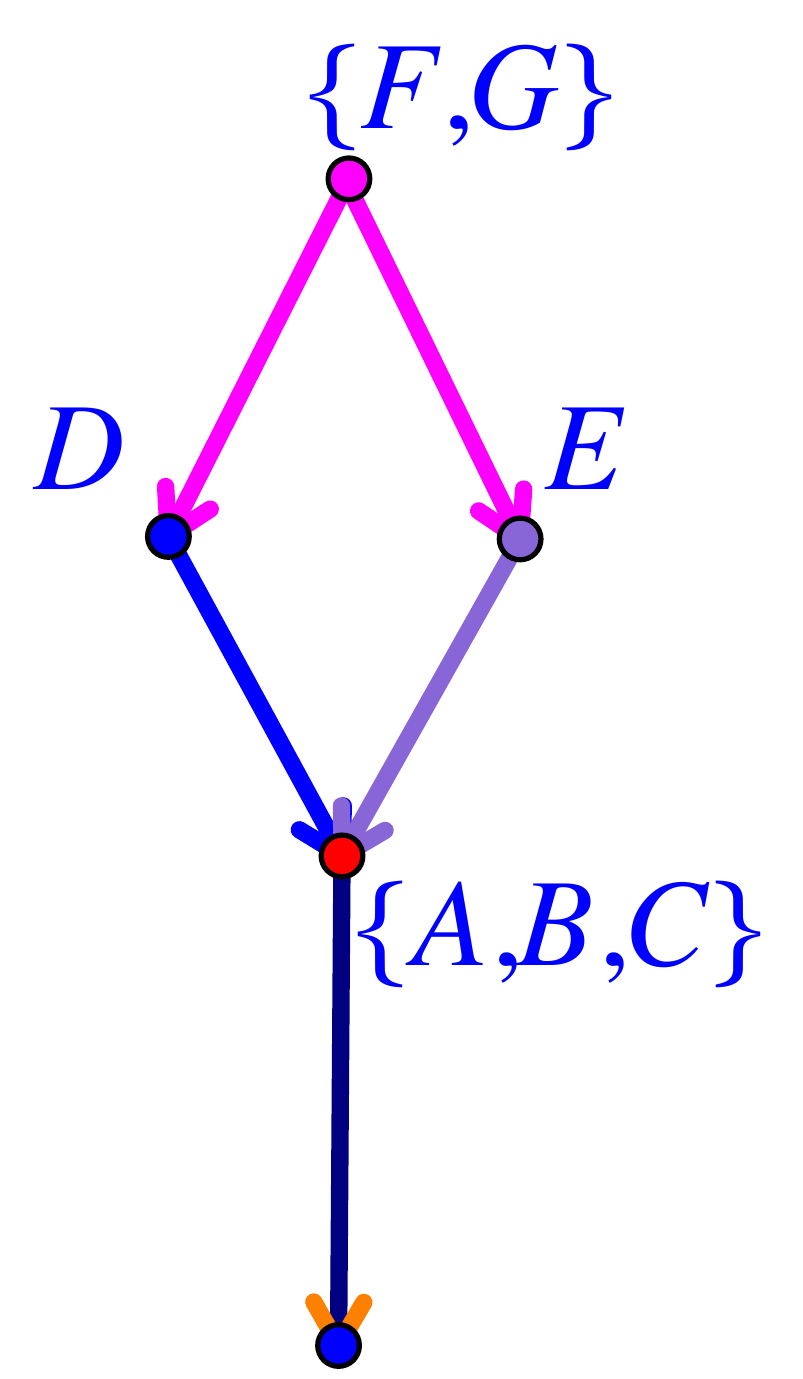}}
  \subfigure[] {\includegraphics [width=.25\textwidth]{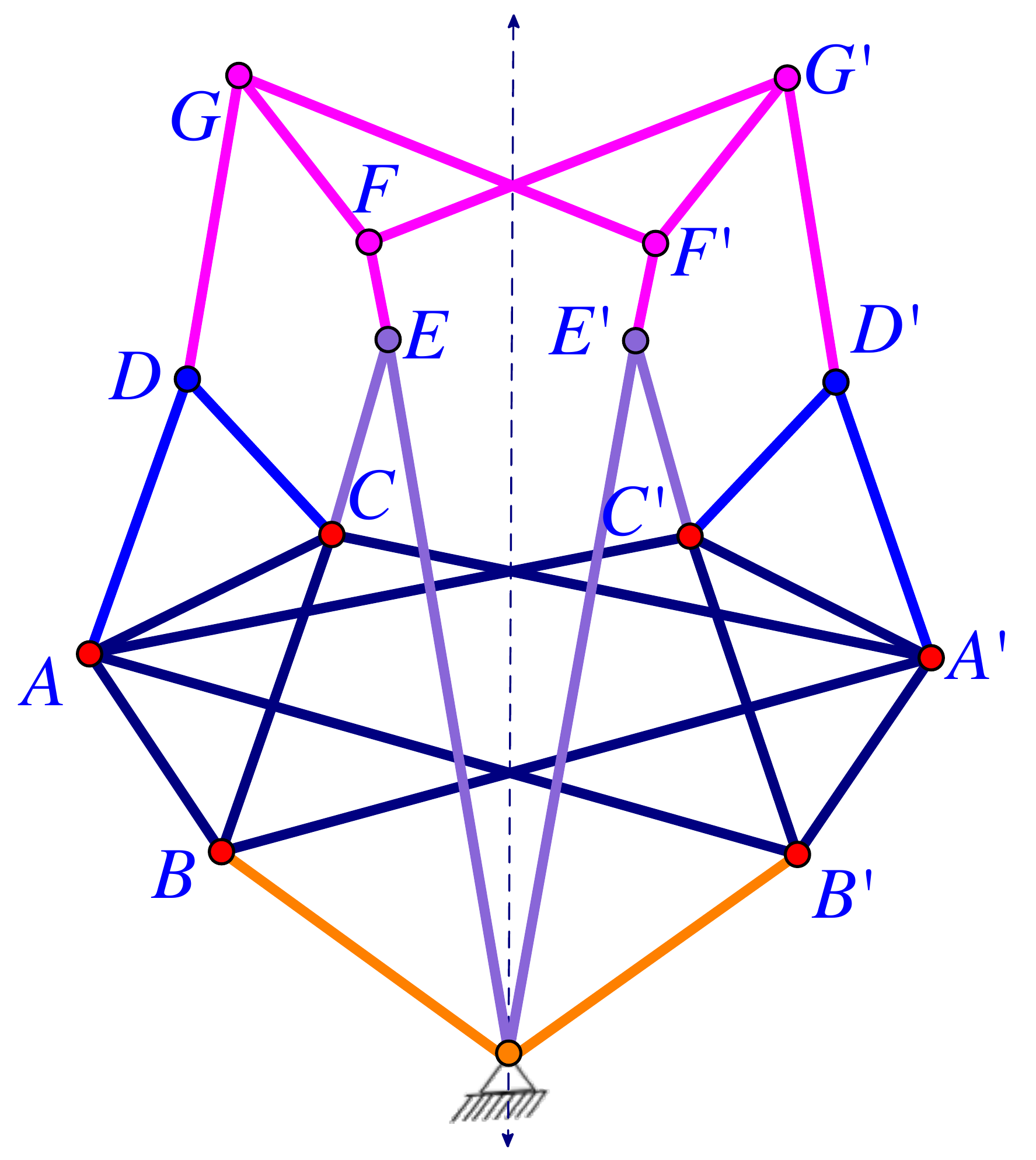}}
      \end{center}
    \caption{A $\mathcal{C}_s$-symmetric pinned framework in the plane (a), with a directed quotient $\mathcal{C}_s$-gain graph (b). The  $\mathcal{C}_s$-Assur block graph is shown in (c).  Figure (d) gives a variant of this with the pins identified on the mirror - but with essentially the same $\mathcal{C}_s$-Assur block graph.}
    \label{fig:PlaneExample}
    \end{figure}


\subsection{Pinned symmetric frameworks with no fixed inner vertices}
\label{subsec:freedecomp}

Analogous to work in \cite{3directed} we use the pinned orbit matrix for any pinned \SG-isostatic graph $\hat G$ to generate the directed quotient \SG-gain graph. Then we relate the directed gain graph strongly connected component decomposition to a block
decomposition of the pinned orbit matrix.

\begin{Proposition} Let $\hat G$  be a pinned \SG-isostatic graph  in dimension $d$, where the group action is free on the vertices. Then there is a \SG-directed orientation of the associated pinned \SG-gain graph $\hat H = (\hat I,\hat P;\hat E)$.
\label{Proposition:dorientation}
\end{Proposition}

\begin{proof}
By Theorem \ref{thm:Scounts3D} we know that $|\hat E|=d|\hat I|$ and $|\hat E'|\leq |\hat I'|$ for all subgraphs $\hat H'$ of $\hat H$. The proposition now follows quickly from standard results on oriented sparse graphs (e.g. \cite[Theorem 1]{FG} or \cite[Lemma 6]{HLST}).
\end{proof}

We remark also that the proof of \cite[Theorem 1]{FG} reveals an efficient polynomial-time algorithm for obtaining the \SG-directed orientation. The algorithm, being based around reversing directed paths, is similar to the well known pebble game algorithm \cite{H&J, adnanthesis}.
Clearly this algorithm is independent of the dimension, group or action.

Given a maximal lower triangular block decomposition of $\mathcal{O}_{pin}(\hat H, \psi, \tilde{\bp})$, the \emph{induced \SG-Assur block graph} has one vertex per block plus a vertex $Z$ for the ground.  There is a directed block graph edge if there is a directed edge $(A,B)$ that goes from the block $A$ to a block $B$ which is upper left from it, i.e., if there is an edge with start vertex in $A$ and end vertex in $B$.  There is a directed edge $(A,Z)$ to the ground if there is an edge in block $A$ which goes to a pinned vertex. Recall Figure \ref{fig:desargues2} and Example \ref{ex:orbitmatrix}.

We now state, without proof, the obvious analogue of Proposition \ref{Proposition:inducedblockgraph} for pinned \SG-isostatic graphs, see also Remark \ref{rem:inducedblockgraph}.

\begin{Proposition} Given a pinned \SG-isostatic graph $\hat G$ with a free group action on the vertices and a maximal lower triangular block decompositon of $\mathcal{O}_{pin}(\hat H, \psi, \tilde{\bp})$, the induced block graph is an acyclic directed graph, with the ground $Z$ on the bottom. Therefore it forms a partial order.
\label{Proposition:SGinducedblockgraph}
\end{Proposition}

Next we prove one of our main results, giving an extension of  \cite[Theorem 3.4]{3directed} to the symmetric setting.

\begin{Theorem}\label{thm:decomp}
Let $\hat G$ be  a pinned \SG-isostatic graph, where the group action is free on the vertices. For any \SG-directed orientation of the quotient $\mathcal{S}$-gain graph $\hat H$, the following decompositions are equivalent:
\begin{enumerate}
\item the \SG-Assur decomposition of $\hat G$;
\item the decomposition into strongly connected components of the $d$-directed orientation of $\hat H$;
\item the induced block graph  from a  maximal block-triangular decomposition of $\mathcal{O}_{pin}(\hat H, \psi, \tilde{\bp})$.
\end{enumerate}
\end{Theorem}

\begin{proof}
$(1) \Rightarrow (2)$. Let $\hat G_1$ be the \SG-Assur component of $\hat G$ containing the ground and let $\hat H_1$ be the corresponding component in $\hat H$. By Proposition \ref{Proposition:dorientation} there is a \SG-orientation of $\hat H_1$. Choose any such orientation and suppose that $\hat H_1$ is not strongly connected. (Remember that Corollary \ref{cor:strong} implies that our choice of \SG-orientation is not important.) Now consider a strongly connected component $\hat H_1^-$ in $\hat H_1$ containing the ground and its covering graph $\hat G_1^-$. Since $\hat H_1$ is pinned \SG-isostatic we have $|E(\hat H_i^-)|\leq d|V(\hat H_i^-)|$ by Theorem \ref{thm:Scounts3D}. Moreover if this inequality were strict then the  strongly connected component above $\hat H_1^-$ , i.e. the graph $\hat H_1 - \hat H_1^-$ together with its edges to $\hat H_1^-$, would have too many edges contradicting the fact that $\hat H_1$ is \SG-isostatic. Thus $|E(\hat H_i^-)|=d|V(\hat H_i^-)|$. Now, since $\hat G$ is \SG-isostatic we know that $\hat G_1^-$ is \SG-isostatic. This contradicts the minimality of $\hat G_1$ since $|V(\hat G_1^-)|< |V(\hat G_1)|$. Repeating this argument for subsequent components completes the argument.

$(2)\Rightarrow (3)$. If there are two or more strongly connected components, then take the bottom component with its edges to the ground. With a permutation of rows and a permutation of column vertices, we place its vertices and edges at the top left of the pinned orbit matrix $\mathcal{O}_{pin}(\hat H, \psi, \tilde{\bp})$. The remaining rows and columns form a second block (which could have several strongly connected components). Continuing this process for each of the blocks up the acyclic strongly connected component decomposition, we find a diagonal matrix block for each component of the decomposition, which gives the matrix the desired lower block triangular form.

$(3)\Rightarrow (1)$. We consider minimal components.
Assume $\hat G$ is not minimal pinned \SG-isostatic, and there is a proper pinned \SG-isostatic subgraph $G^{*}$. We will show the pinned orbit matrix $\mathcal{O}_{pin}(\hat H, \psi, \tilde{\bp})$ decomposes.
If we permute the inner vertices and all the edges associated with $G^{*}$ to the upper left corner of $\mathcal{O}_{pin}(\hat H, \psi, \tilde{\bp})$, the rest of these rows from $G^{*}$ are 0 and the remaining columns and rows form a second block. This gives a block triangular decomposition of the pinned orbit matrix $\mathcal{O}_{pin}(\hat H, \psi, \tilde{\bp})$. The contrapositive says that if $\mathcal{O}_{pin}(\hat H, \psi, \tilde{\bp})$ does not have a proper block triangular decomposition, then the pinned \SG-isostatic graph is minimal.
\end{proof}

As a corollary we observe that the equivalence is true for individual components in the decomposition.

\begin{Corollary}
\label{cor:decomp} Let $\mathcal{S}$ be a symmetry group in dimension $d$, and let  $\hat G$ be  a pinned \SG-isostatic graph, where the group action is free on the vertices.
Further, let $\hat H$ be the quotient $\mathcal{S}$-gain graph of $\hat G$. Then the following are equivalent:
\begin{enumerate}
  \item $\hat G$  contains no proper pinned \SG-isostatic subgraphs;
  \item $\hat H$ is indecomposable for some (any) \SG-directed orientation;
  \item the pinned orbit matrix $\mathcal{O}_{pin}(\hat H, \psi, \tilde{\bp})$ has no proper block triangular decomposition.
\end{enumerate}\label{thm:2blockdecomp}
\end{Corollary}

We have shown that the \SG-Assur block graph encodes all the information about the \SG-Assur decomposition. This bijection allows us to be slightly terse in subsequent examples and figures, referring only to the partial order (the \SG-Assur block graph) rather than the \SG-Assur decomposition.

\subsection{Pinned symmetric frameworks with fixed inner vertices }\label{subsec:fixed}

In the previous subsection we assumed that the group action is free on vertices. 
 Here we discuss how the results in Section 3.1 easily generalize to non-free actions.

Recall that if some of the inner vertices are fixed by a non-trivial symmetry operation (i.e. the underlying group action is not free on the vertex set), then not all vertices in the pinned orbit matrix will have the same number of columns (see Remark~\ref{rem:fixorb} above). Since the number of assigned outgoing edges for each vertex is equal to the number of columns under the vertex in the pinned orbit matrix, the directed gain graph will no longer be $d$-directed.

Note that in the case where the group action is free on the vertices, we have dim$U_{\bp(u)}=d$. So this is exactly a $d$-directed orientation of the gain graph and hence a \SG-orientation, as defined for the free case.



The results in Subsection \ref{subsec:freedecomp} will also apply to \SG-directed orientations, where the column size (vertex out-degree) of the gain graph is not uniform among all vertices. To make this evident, we recall from \cite{3directed} that the choices in orientations of edges which conserve a fixed out-degree of each of the vertices do not alter the decomposition. For any two equivalent directed gain grain orientations (i.e. corresponding vertices have same out-degree), the two orientations only differ by reversals on a set of directed cycles. Moreover, since cycle reversals do not change the strongly connected components, the decompositions for two orientations are the same (see Corollary \ref{cor:strong}). We can now generalize the results above without any assumption on the group action. That is, Theorem \ref{thm:decomp} and Corollary \ref{cor:decomp} apply without the assumption that \SG~acts freely on $\hat G$.

\begin{Corollary}\label{thm:decompfixed}
Let $\hat G$ be  a pinned \SG-isostatic graph, where the group action is not free on the vertices. For any \SG-directed orientation of the quotient $\mathcal{S}$-gain graph $\hat H$, the following decompositions are equivalent:
\begin{enumerate}
\item the \SG-Assur decomposition of $\hat G$;
\item the decomposition into strongly connected components of the \SG-directed orientation of $\hat H$;
\item the induced block graph  from a  maximal block-triangular decomposition of $\mathcal{O}_{pin}(\hat H, \psi, \tilde{\bp})$.
\end{enumerate}
\end{Corollary}


\subsection{\SG-Drivers and strongly \SG-Assur graphs}

 Assume we have a \SG-Assur decomposition of a pinned graph $\hat{G}$.   If we replace an orbit of edges in this graph by a set of symmetric drivers (i.e., a set of edges simultaneously change their lengths in a coordinated, symmetric fashion), then the pinned framework will have a symmetric motion which fixes the ground (and the ground includes an orbit of end-vertices of the edges that were converted to drivers).

\begin{Proposition} Let $\hat G$ be a \SG-Assur graph with quotient \SG-gain graph $\hat H$.
\begin{enumerate}
\item For a \SG-regular realisation, removing any edge from $\hat H$
generates a fully \SG-symmetric finite motion in which some non-pinned vertices (of $\hat G$) are in motion relative to the ground.

\item Given a \SG-Assur decomposition of $\hat G$, if an edge $\tilde e$ of $\hat H$ is removed,
then no vertices in components below or incomparable to the component containing $\tilde e$ move.
\end{enumerate}
\end{Proposition}

\begin{proof} The proof is similar to the proof in \cite{3directed}.

1. We know that the orbit matrix $\mathcal{O}_{pin}(\hat H, \psi, \tilde \bp)$ is square and invertible.
Suppose $\tilde e$ corresponds to row $k$. Let $e_k$ denote the $k$-th standard basis vector.
Then
$$
\mathcal{O}_{pin}(\hat H, \psi, \tilde \bp) \times U\ = \ e_k  \qquad \Rightarrow   \qquad U\ =\mathcal{O}_{pin}(\hat H, \psi, \tilde \bp)^{-1} \times\ e_k
$$
for some non-zero $U$.  $U$ is then the required set of velocities for the inner vertices.

2. $\hat{G}$ is \SG-Assur, guaranteeing some inner vertices have non-zero velocities (after the edge orbit is removed).  In particular, to obtain $e_k$ from $\mathcal{O}_{pin}(\hat H, \psi, \tilde \bp) \times U$,
 at least one of the vertices for this edge must have a non-zero velocity.  This in turn converts to a fully \SG-symmetric motion of the original covering graph.

If we consider a block lower-triangular decomposition of the matrix, it  is clear that vertices in blocks above the row $k$ continue to have a full invertible submatrix and all entries above the $k$th entry in $e_k$ are also $0$.  Therefore, the entries in $U$ for orbits of vertices in these components (components below the component of row $k$, or incomparable to the component of row $k$) must be $0$.
\end{proof}

In \cite{3directed} it was shown that not all $d$-Assur graphs (when $d>2$) are strongly $d$-Assur. This is connected to obstacles to a good combinatorial (counting) characterisation for $3$- and higher-dimensional  bar and joint frameworks to be rigid.
The example in Figure \ref{fig:Weakly} is a \SG-Assur graph but it is not {\em strongly \SG-Assur}. Therefore, as for  graphs without imposed symmetry, being strongly \SG-Assur is a stronger property than being \SG-Assur.

\begin{figure}[ht]
    \begin{center}
  \subfigure[] {\includegraphics [width=.38\textwidth]{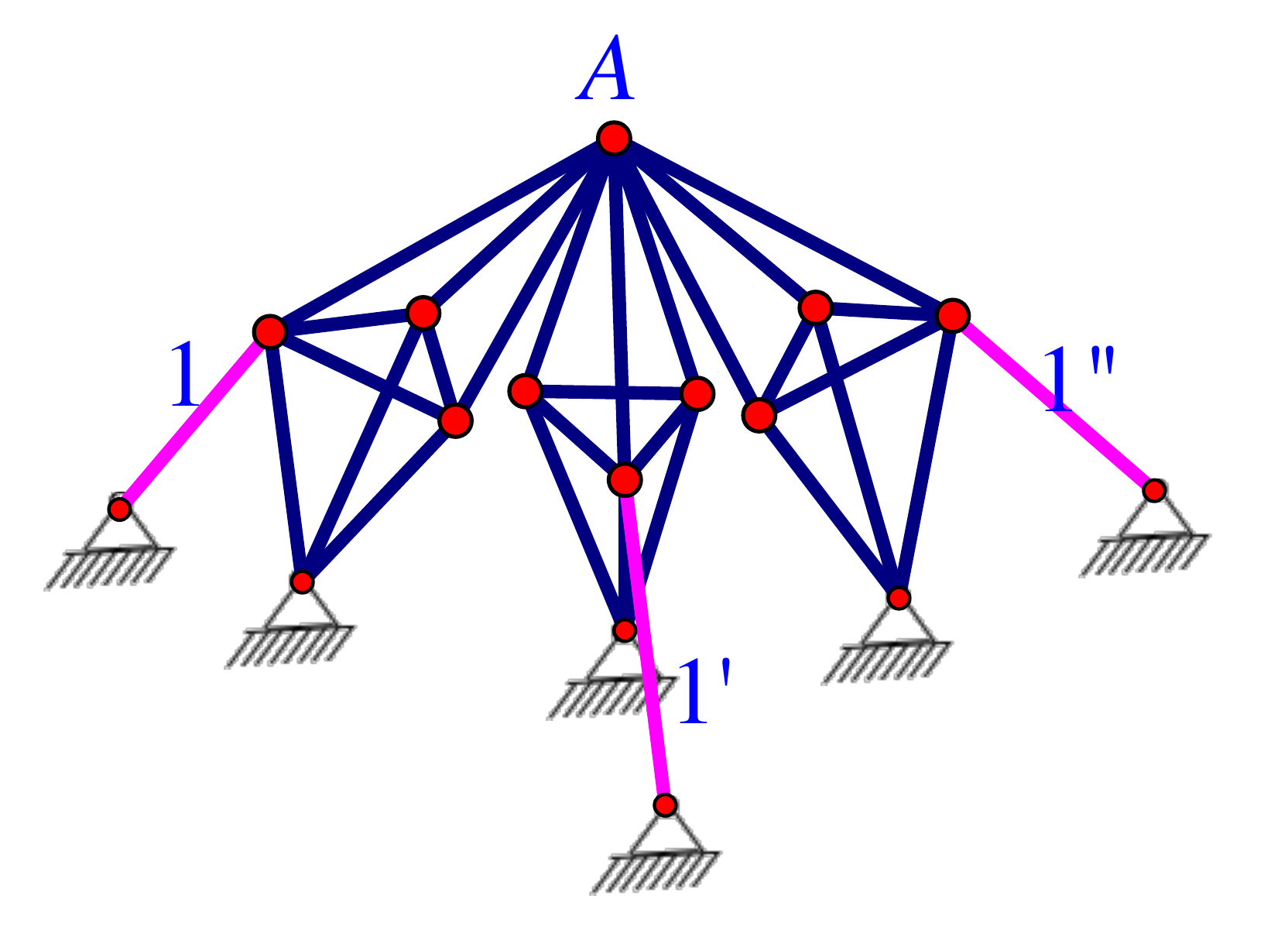}} \quad\quad
    \subfigure[] {\includegraphics [width=.38\textwidth]{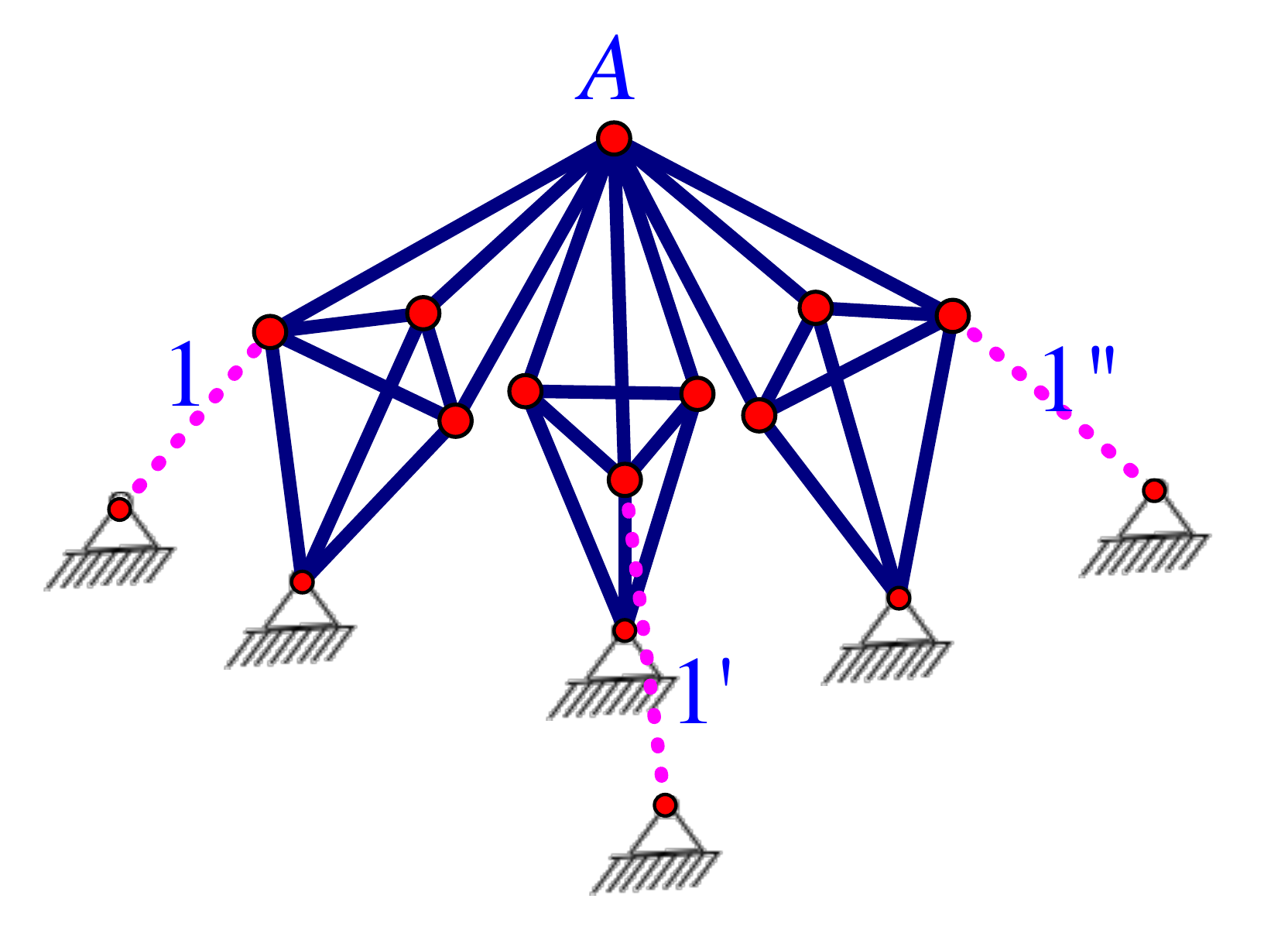}}
       \end{center}
    \caption{A pinned $\mathcal{C}_3$-isostatic graph in 3-space  (a) which is  $\mathcal{C}_3$-Assur but not {\em strongly  $\mathcal{C}_3$-Assur}. If we turn the purple edges into drivers (i.e. remove these edges) (b) vertex A will not move.}
    \label{fig:Weakly}
    \end{figure}

\begin{Corollary}
Let $\hat G$ be pinned \SG-isostatic and let $\hat H$ be the quotient  \SG-gain graph of $\hat G$. If an edge $\tilde e$ is removed from some component of the strongly \SG-Assur decomposition of ($\hat G$ and hence) $\hat H$, then all vertices in this component $\hat H^*$ or in components above the component containing $\tilde e$ move.
\end{Corollary}

\begin{proof} The definition of a strongly \SG-Assur graph $\hat H^*$  guarantees that all inner vertices of $\hat H^*$  are in motion.    If there is a higher strongly \SG-component in which some inner vertices are not in motion, then none of them are in motion, including the vertices at the heads of the outgoing edges of the component, since the \SG-Assur component is \SG-isostatic.
This in turn means that some vertices of the components just below this are also not moving.  Going down the \SG-Assur decomposition, we conclude that the vertices of the strongly \SG-component from which the orbit was removed are also not in motion. This is a contradiction.
\end{proof}

For certain groups we can show that there is no distinction between \SG-Assur graphs and strongly \SG-Assur graphs.

\begin{Proposition}
\label{prop:strong}
Let \SG~be the symmetry group $\mathcal{C}_s$ or $\mathcal{C}_n$ in the plane, where $n\geq 2$, and let \SG~act freely on the vertices of a pinned graph $\hat G$. Then $\hat G$ is \SG-Assur if and only if it is strongly \SG-Assur.
\end{Proposition}

The proof is similar to \cite[Proposition 4.2]{3directed}.

\begin{proof}
By definition, strongly \SG-Assur graphs are \SG-Assur.
For the converse, assume $\hat G$ is \SG-Assur and delete an edge $\tilde u \tilde v$ from the corresponding \SG-gain graph $\hat H$. Theorem \ref{thm:Scounts3D} implies that $\hat H-\tu\tv$ satisfies $|\tilde E|=2|\tilde I|-1$. Hence the orbit matrix $\mathcal{O}_{pin}(\hat H-\tu\tv, \psi, \tilde \bp)$ admits a non-trivial solution $U$ to the equation $\mathcal{O}_{pin}(\hat H-\tu\tv, \psi, \tilde \bp) \times U=0$. Let $U=(U_1,U_2,\dots,U_{|\tilde E|})^T$. If $U_i=0$ for some $i$ then that inner vertex is still \SG-rigidly connected to the ground. Therefore Theorem \ref{thm:Scounts} implies vertex $i$ must be contained in a pinned subgraph $\hat H'$ with $|\tilde E'|=2|\tilde I'|$. The corresponding covering graph $\hat G'$ would then be pinned \SG-isostatic but $\hat H'$ contains at most one of $\tu$ and $\tv$ implying that $|\tilde I'| < |\tilde I|$. This contradicts the minimality of $\hat G$.
\end{proof}

For more exotic symmetry groups, we do not know if the proposition holds.
We do however note that in higher dimensions the proposition fails, and hence, see the difficulty for even order dihedral groups in the plane (see the comments following Theorem \ref{thm:symmetry_reflection}) as a warning that there is the potential for the proposition above to break down in that case.

\section{\SG-Assur graphs with generically pinned-isostatic graphs}\label{sec:isostatic}

In this section we will analyse $\mathcal{S}$-Assur decompositions of $\mathcal{S}$-symmetric pinned graphs which are both pinned $\mathcal{S}$-isostatic and pinned isostatic (i.e. there are also no symmetry-breaking motions or non-symmetric self-stresses).
Examples of such graphs are shown in Figures~\ref{fig:desargues}, \ref{fig:C3Space} and \ref{fig:C4}. Given such a graph (with a free group action), we will show in  Theorem~\ref{thm:map} that there exists a very strong connection between the Assur and the \SG-Assur decompositions of the graph - essentially a bijection between them (see Figures~\ref{fig:desargues}, \ref{fig:desargues2}).

\begin{figure}[ht]
    \begin{center}
  \subfigure[] {\includegraphics [width=.48\textwidth]{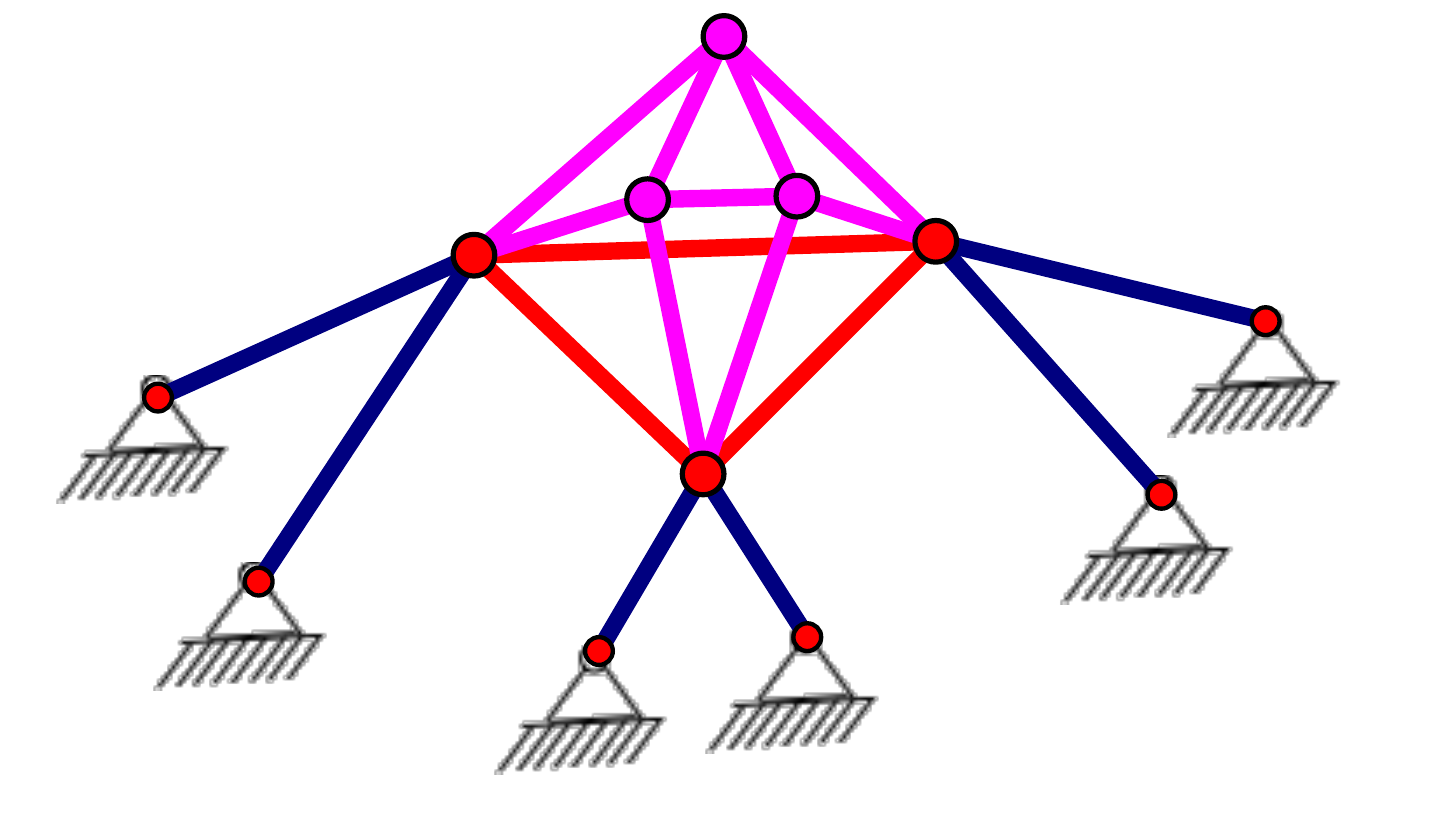}} \quad\quad
    \subfigure[] {\includegraphics [width=.14\textwidth]{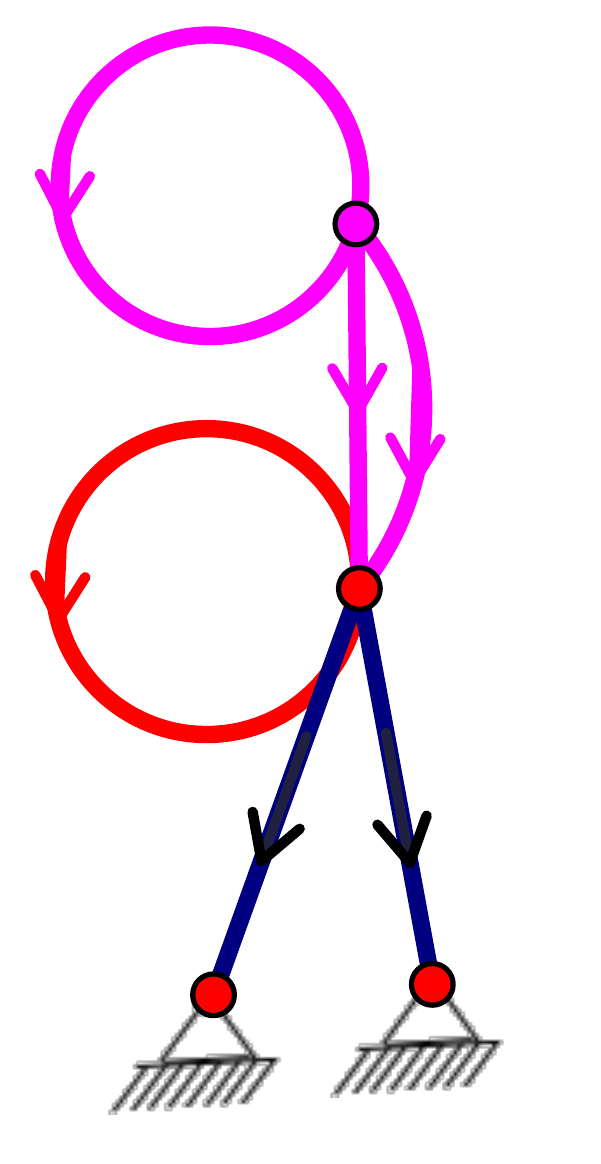}} \quad\quad
\subfigure[] {\includegraphics [width=.09\textwidth]{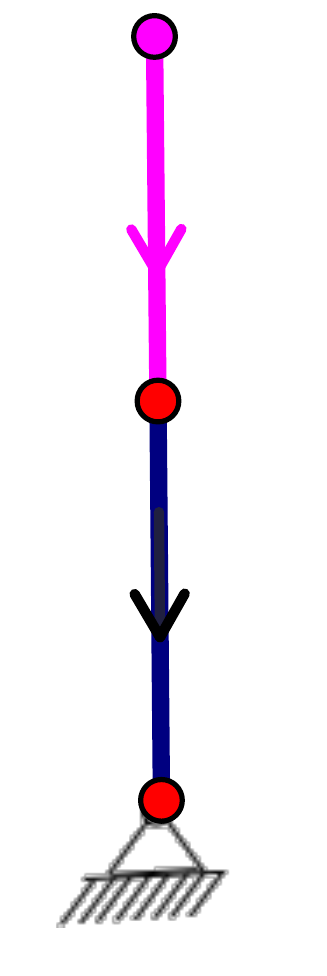}}
       \end{center}
    \caption{A $\mathcal{C}_3$-symmetric isostatic framework in $3$-space (a), its   quotient  $\mathcal{C}_3$-gain graph (b), and its $\mathcal{C}_3$-Assur block graph (c).}
    \label{fig:C3Space}
    \end{figure}

For the symmetry group $\mathcal{C}_3$ in the plane, a pinned $\mathcal{C}_3$-symmetric graph, where $\mathcal{C}_3$ acts freely on the vertices, is pinned $\mathcal{C}_3$-isostatic if and only if it is  pinned  isostatic. This  follows from Theorem~\ref{thm:Scounts} and the fact that for each of the three irreducible representations of $\mathcal{C}_3$,  the rank properties of the corresponding orbit matrix are described by exactly the same counting conditions given  in Theorem~\ref{thm:Scounts} (see \cite{schtan}).
More generally, we conjecture that for the rotational groups $\mathcal{C}_n$, $n>2$, in the plane, pinned $\mathcal{C}_n$-isostaticity is equivalent to pinned isostaticity, provided that the action of   $\mathcal{C}_n$ is free on the vertices and edges. In particular, this includes all odd-order rotational groups which act freely on the vertices.

 However, note that if a pinned graph $\hat G=(I,P;E)$ is $\mathcal{C}_n$-isostatic, $n\geq 2$, or  $\mathcal{C}_s$-isostatic in the plane, where the action is free on the vertices, but not on the edges, then  we must have $|E|<2|I|$, and hence $\hat G$ is not pinned isostatic, but flexible \cite{FGsymmax}. For  $\mathcal{C}_2$ or $\mathcal{C}_s$ in the plane, symmetry-breaking motions can also arise in $\mathcal{C}_2$- or $\mathcal{C}_s$-isostatic pinned graphs   if the action is free on both vertices and edges, due to a violation of the $(2,3,2)$-gain sparsity count (see Figure~\ref{fig:nonisostatic} and \cite{schtan}).





For $\mathcal{S}$-symmetric pinned graphs which are both pinned $\mathcal{S}$-isostatic and pinned isostatic, we {\it conjecture} that if we remove an orbit of edges and convert those edges  into synchronized drivers (actuators), then the resulting symmetry-preserving motion at a \SG-regular configuration is the only motion of the structure. In other words, 
no symmetry-breaking motions can occur.


\begin{figure}[ht]
    \begin{center}
  \subfigure[] {\includegraphics [width=.38\textwidth]{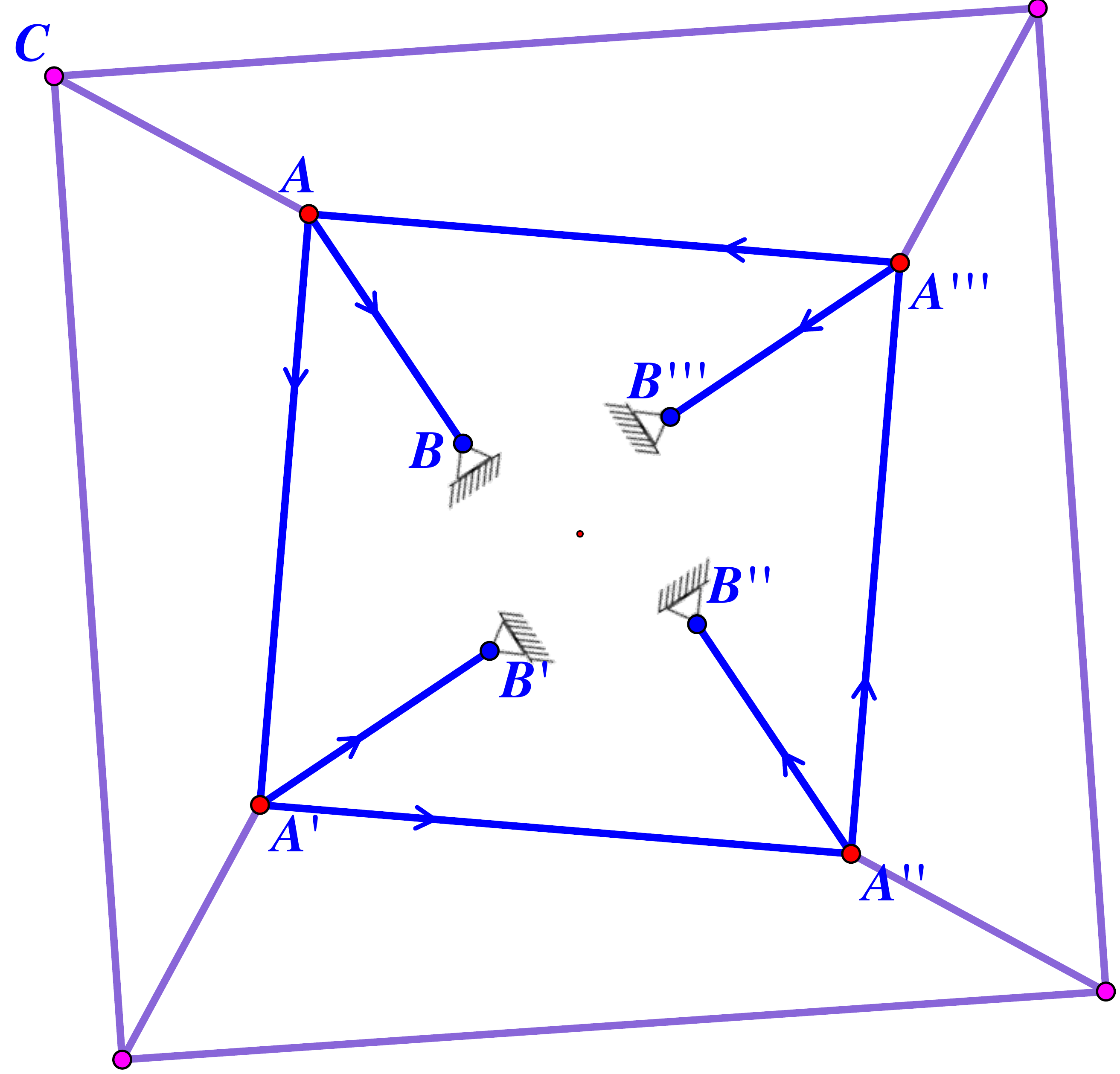}} \quad\quad \quad
  \subfigure[] {\includegraphics [width=.13\textwidth]{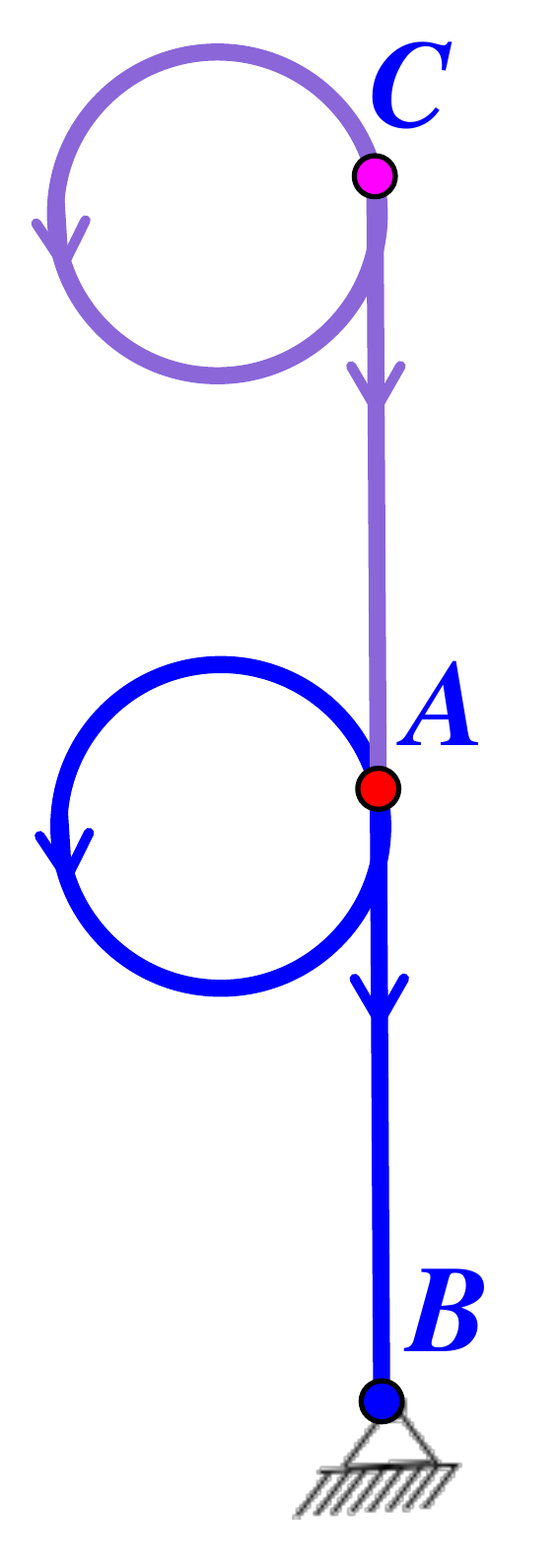}} \quad \quad
  \subfigure[] {\includegraphics [width=.08\textwidth]{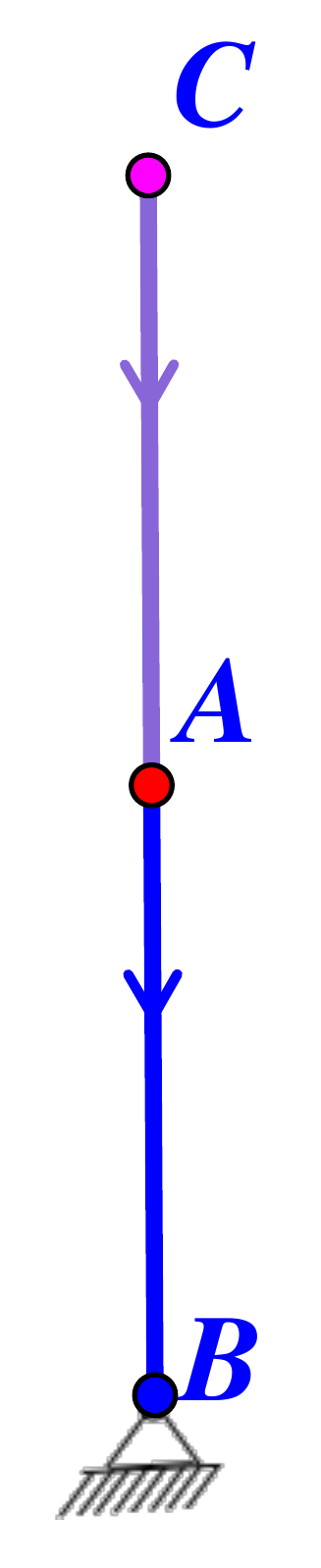}} 
   \end{center}
    \caption{A framework with $\mathcal{C}_4$-symmetry in the plane (a), its   quotient  $\mathcal{C}_4$-gain graph (b), and its $\mathcal{C}_4$-Assur block graph (c). }
    \label{fig:C4}
    \end{figure}

\begin{Example} Consider the example in Figure \ref{fig:C4}. A 2-direction on the gain graph (b) induces a symmetric set of 2-directions on the original graph.
If we shift any orbit of edges from the bottom component into a symmetric set of drivers, then this has a $4$-fold symmetric motion, which moves all vertices.  If we shift an orbit of edges from the top component into a symmetric set of drivers, then there is a $\mathcal{C}_4$-symmetric motion which moves only the vertices in the top component.

If these drivers encounter a `dead end position' in which the drivers no longer can move, then the rank of the orbit matrix has dropped at this singular position and there will be a symmetric self-stress.  We {\it conjecture}  this symmetric self-stress must be non-zero on all edges in this graph. Whether this extends to all \SG-Assur components is a question worth further investigation.
\end{Example}

\subsection{Mapping between gain graphs and covering graphs}

We now formalise the relationship between components in the covering graph and the gain graph.

\begin{Lemma}\label{lem:gaintocover}
Let $\hat G$ be a pinned graph and let the action of \SG~be free on vertices and edges. Then
\begin{enumerate}
\item any \SG-directed orientation of $\hat H$ lifts to a d-directed orientation of $\hat G$ and any d-directed orientation of $\hat G$ projects to a \SG-directed orientation of $\hat H$.
\item any strongly connected component of $\hat H$ lifts to a set of strongly connected components of $\hat G$ and any strongly connected component of $\hat G$ projects to a strongly connected component of $\hat H$.
\end{enumerate}
\end{Lemma}

\begin{proof}
For any vertex $\tilde v \in \tilde V$ and any vertex $v \in \hat V$ in the fiber over $\tilde v$ the out-degrees are equal by the definition of a gain graph and the assumption the action is free on the vertices and edges. This proves 1.

Let \SG~have elements $id,a_1,a_2,\dots,a_{m-1}$ (so \SG~has order $m$). The group operation will be written additively in this proof.
Suppose there is a directed cycle $\tilde C_n$ in $\hat H$ on vertices $\tilde v_1,\tilde v_2,\dots, \tilde v_n$ with gain $g_{i}$ on edge $\tilde v_i\tilde v_{i+1}$ for each $i$, and that the corresponding vertex orbits in $\hat G$ are $\{v_i^{id}=v_i,v_i^{a_1},\dots, v_i^{a_{m-1}}\}$ for $1\leq i \leq n$. Moreover, suppose $k$ is the order of the subgroup of \SG~induced by the labels on the edges of $\tilde C_n$. Each edge $\tilde v_i\tilde v_{i+1}$ lifts to an edge $v_i^rv_j^{r+g_i}$ where $r\in\mathcal{S}$ is the sum of the gains on the edges $\tilde v_1\tilde v_2,\dots, \tilde v_{i-1}\tilde v_i$. The path $v_1v_2^{g_1}, v_2^{g_1}v_3^{g_1+g_2}, \dots$ visits $v_1$ again on its $k$th visit to a vertex in the orbit $\tilde v_1$ (and also visits each other vertex $k$ times). For any vertex in the orbit of $\tilde v_1$ not visited by this cycle, repeat the path above.
It will take $m/k$ repetitions to ensure every vertex in the orbit of $\tilde v_1$ has been visited. Hence $\tilde C_n$ corresponds to $m/k$ directed cycles $\hat C_{kn}^1,\hat C_{kn}^2,\dots, \hat C_{kn}^{m/k}$ in $\hat G$. This proves the first part of 2, lifting from $\hat H$ to $\hat G$.

Conversely, take any two vertex orbits in $\hat H$  that project from the same strongly connected component in $\hat G$.  Since there is a directed cycle in $\hat G$ connecting elements of the two orbits, the projection of this cycle into $\hat H$ will be a cycle (not necessarily simple) connecting the two orbits.  The completes the proof of 2.
\end{proof}

We can now prove the second main result of the paper.

\begin{Theorem}
\label{thm:map}
Let $\hat G$ be a pinned \SG-isostatic  graph which is pinned isostatic at every pinned \SG-regular realisation and let $\hat H$ be the corresponding \SG-gain graph. Moreover suppose the action of \SG~is free on vertices and edges. Then the projection of a $d$-Assur graph is a \SG-Assur graph, and the lift of a \SG-Assur graph is a set of inner-vertex disjoint sets of $d$-Assur graphs (possibly only one).
\end{Theorem}

\begin{proof}
Theorem \ref{3DirectedAssur} implies the strongly connected components of $\hat G$ are the $d$-Assur components. Theorem \ref{thm:decomp} implies the strongly connected components of $\hat H$ are the \SG-Assur components. Thus Lemma \ref{lem:gaintocover} implies the result.
\end{proof}

We expect that the map between quotient and covering graphs contains enough information for the following stronger conclusion to be possible.

\begin{Conjecture}Let $\hat G$ be a pinned \SG-isostatic graph which is isostatic at every \SG-regular realisation and let the action of \SG~be free on the vertices and edges. Then a \SG-Assur component is strongly \SG-Assur if and only if the corresponding d-Assur components are strongly d-Assur.
\end{Conjecture}

Recall that for reflection and rotation symmetry in the plane pinned graphs are \SG-Assur if and only if they are strongly \SG-Assur (see Propposition \ref{prop:strong}). Hence the conjecture follows in those cases from Theorem \ref{thm:map}. However we still believe the conjecture in the cases where \SG-Assur is a strictly weaker property than strongly \SG-Assur.

\subsection{Subgroup Assur decompositions}\label{subsec:subgroup}

In Figure \ref{fig:C6},  we observe that there is a correspondence among the distinct \SG-block graphs induced by  various subgroups.  Specifically, when a symmetry group \SG$'$ is a subgroup of \SG, then  the gain graph of \SG is a projection of the gain graph of \SG$'$  and there is an induced projection from the \SG$'$-block graph to the \SG-block graph.  Conversely, there is a lifting of  the  gain graph of \SG to the gain graph of \SG$'$.   This illustrates a general principle.
If $\hat G$ is \SG-isostatic and \RG-isostatic for some subgroup \RG~of \SG~then the analysis in the previous subsection can be adapted to reveal the \RG-Assur decomposition.

\begin{figure}[ht]
    \begin{center}
  \subfigure[] {\includegraphics [width=.25\textwidth]{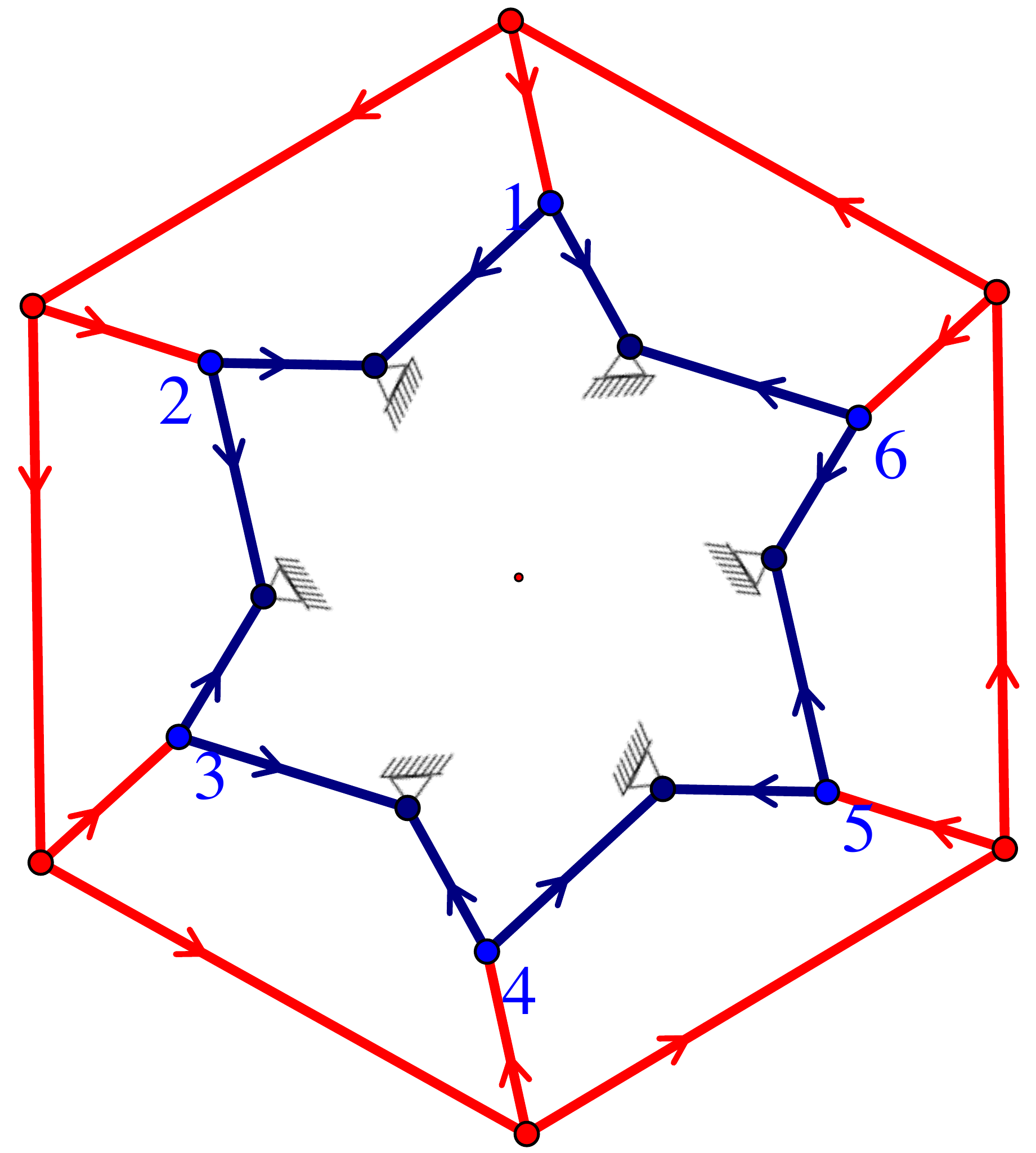}} \quad
  \subfigure[] {\includegraphics [width=.33\textwidth]{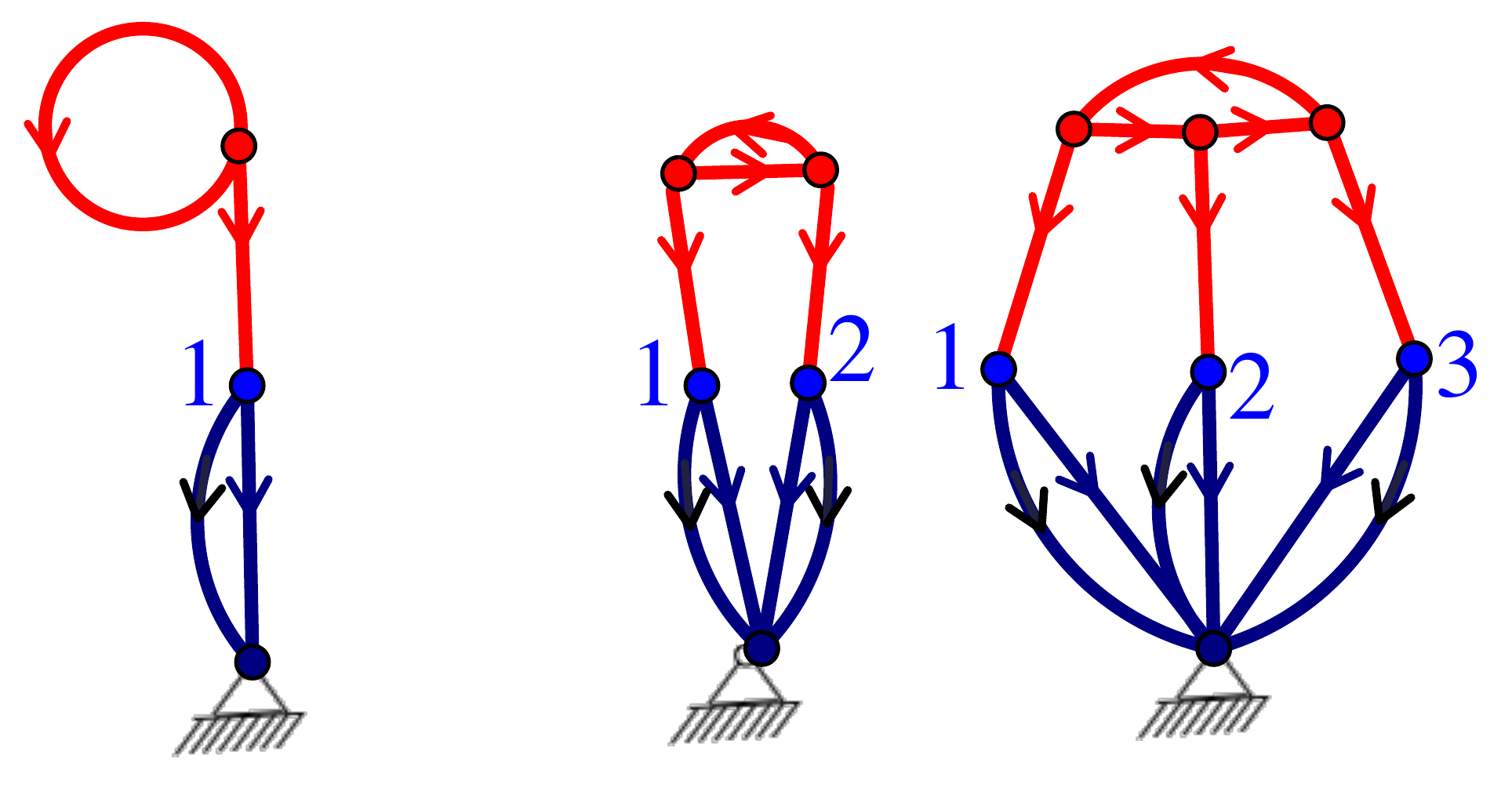}} \quad
  \subfigure[] {\includegraphics [width=.33\textwidth]{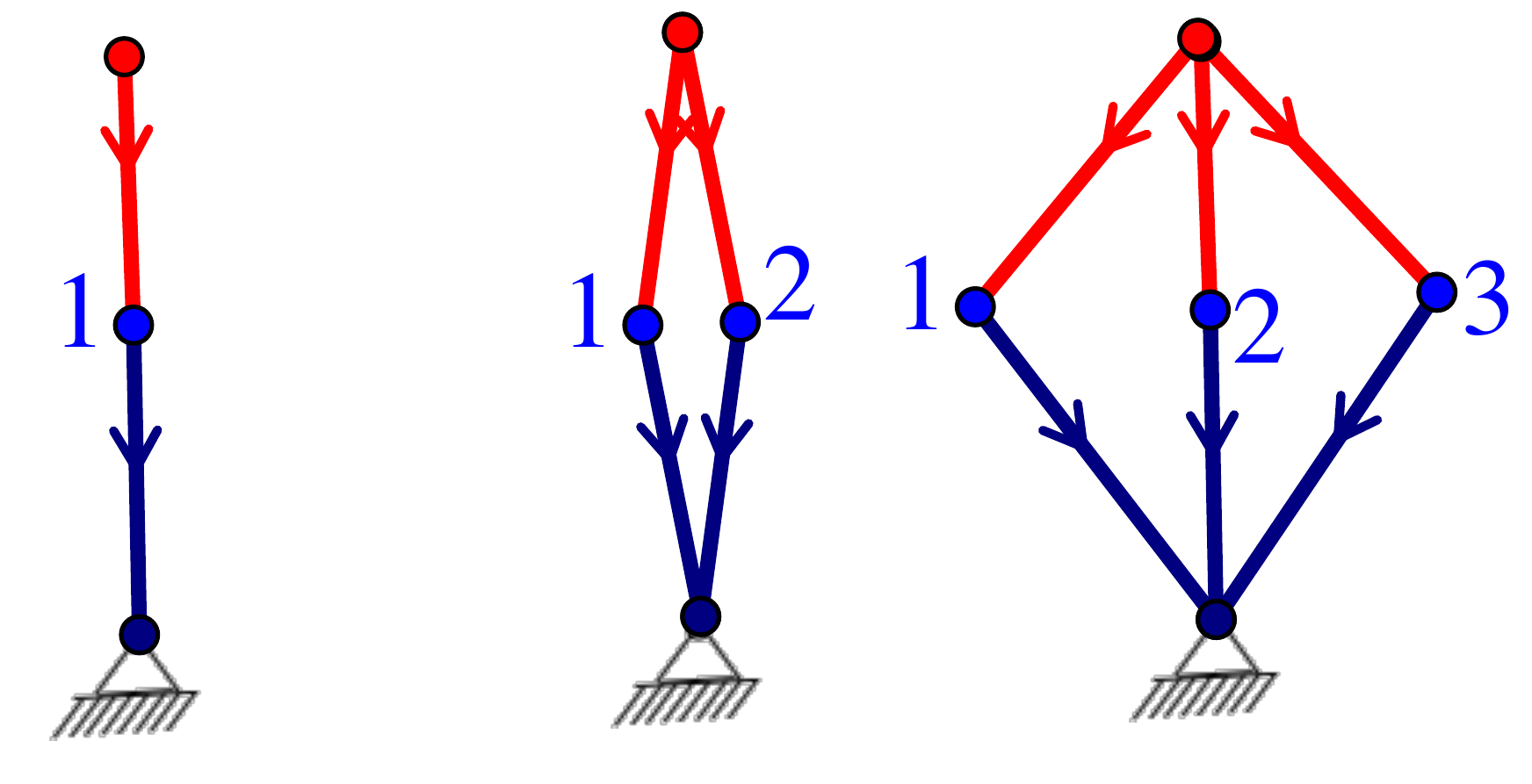}} 
   \end{center}
    \caption{(a) A pinned framework in the plane with $\mathcal{C}_6$ symmetry, (b) the \SG-gain graphs where \SG~corresponds to $\mathcal{C}_6, \mathcal{C}_3$ and $\mathcal{C}_2$-symmetry and (c) the corresponding \SG-Assur block graphs.}
    \label{fig:C6}
    \end{figure}

Let us mention one straightforward way in which this can be done. If $\hat G$ is pinned isostatic we can simply apply Theorem~\ref{thm:map} to \SG~and to \RG~separately. Here composing the \SG-quotient map with the \RG-lifting map gives the quotient map from \SG~to \RG.  Similarly composing the \RG-quotient map with the \SG-lifting map gives us the lifting map from \RG~to \SG.

Consider again Figure~\ref{fig:C6}, and how this might be applied. Since we now know that we can choose any subgroup \SG~of $\mathcal{C}_6$ to apply a \SG-Assur decomposition, we can make different choices of edge orbits to take as drivers depending on whether we want to see a $\mathcal{C}_2$, $\mathcal{C}_3$ or $\mathcal{C}_6$-symmetric motion.  We can pick a set of drivers of size 6, 3, 2 (or 1 for the identity subgroup).  Each choice will guarantee a motion with at least the symmetry of the subgroup we chose.   It is more subtle to realise in this example that the motion being driven will not have additional symmetry!

\section{\SG-Assur graphs which are not pinned isostatic at \SG-regular configurations}\label{sec:variants}

In Section~\ref{sec:isostatic} we focused on pinned \SG-isostatic graphs which were also pinned isostatic.  This assumption ensured that all \SG-Assur components were also pinned isostatic and $d$-Assur.  In that section, a key assumption was that the action of the group was free on the vertices and edges.
Some previous \SG-isostatic examples included components which were not isostatic without symmetry - and therefore the underlying graph does not have a $d$-Assur decomposition without symmetry (Figures~\ref{fig:Mechanisms}(d), \ref{fig:PlaneExample}(d)).  There are simple examples which illustrate the possible  impact of fixed vertices generating redundance (Figure~\ref{fig:nonisostatic}(a,b)) and the possible impact of fixed edges giving flexible frameworks (Figure~\ref{fig:nonisostatic}(c,d)).

\begin{figure}[ht]
    \begin{center}
\subfigure[] {\includegraphics [width=.32\textwidth]{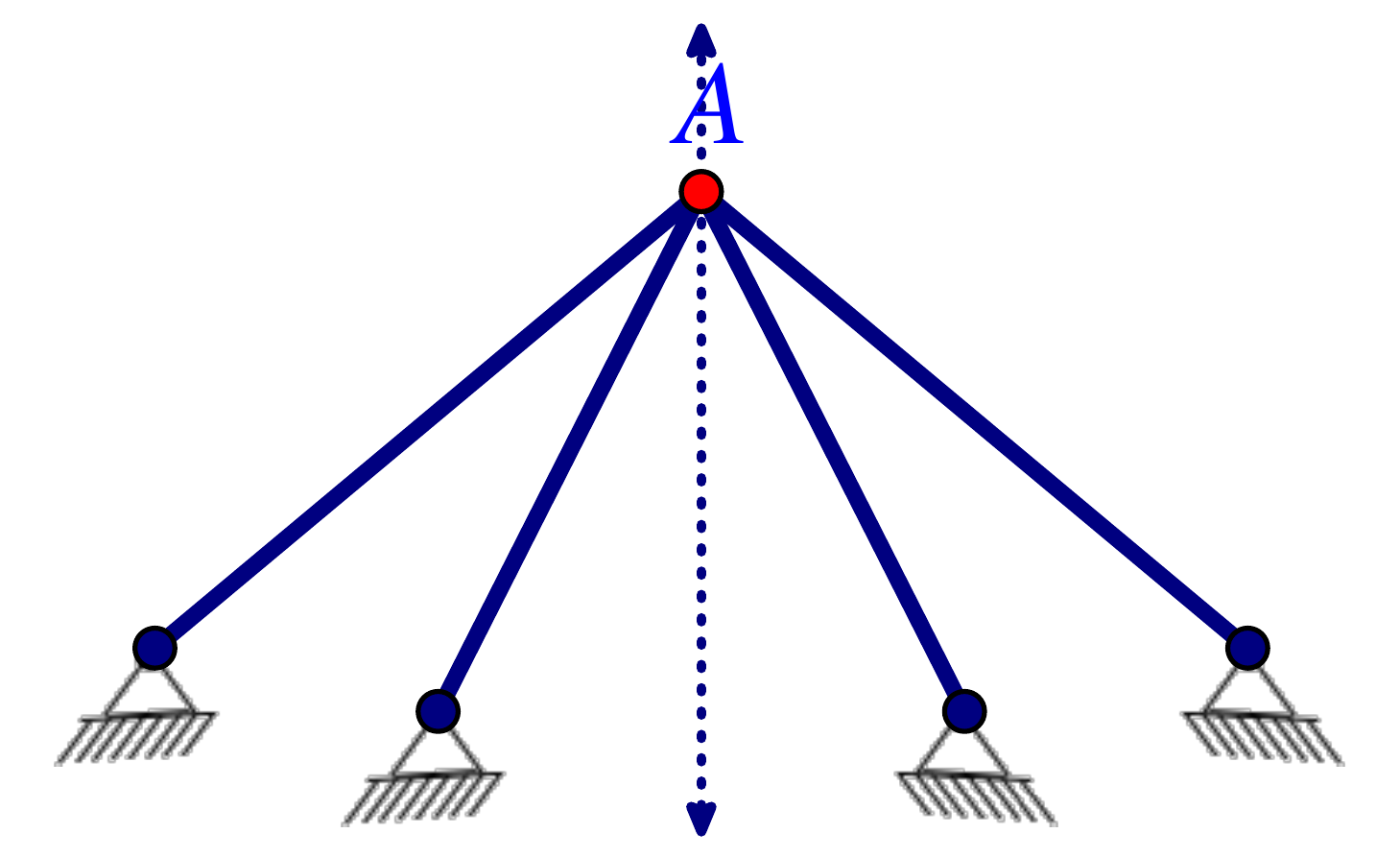}} \quad
  \subfigure[] {\includegraphics [width=.18\textwidth]{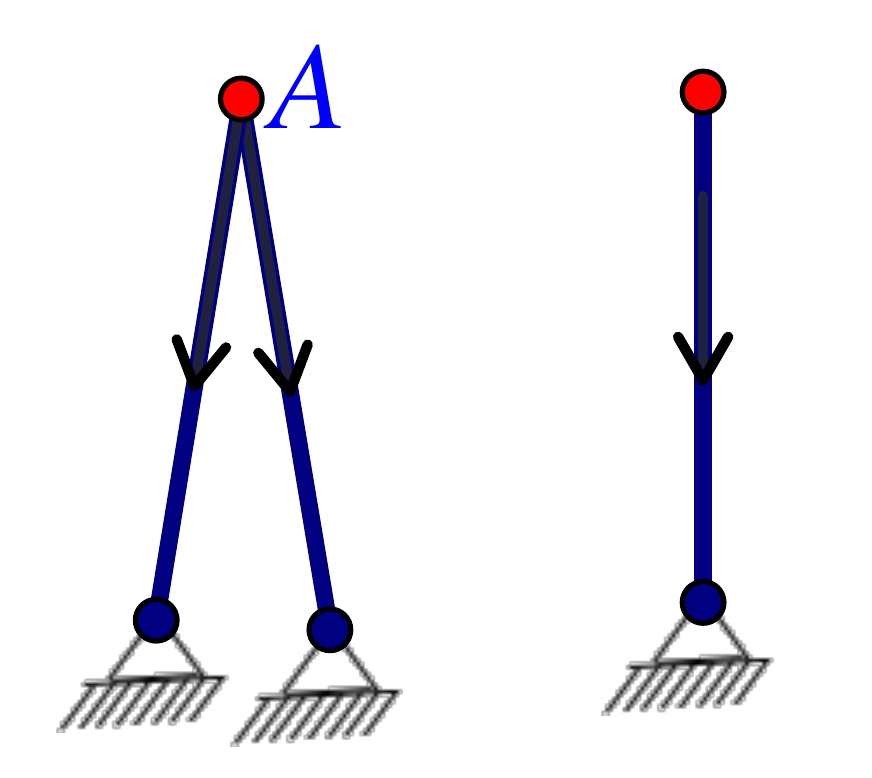}} \quad\quad
 \subfigure[] {\includegraphics [width=.20\textwidth]{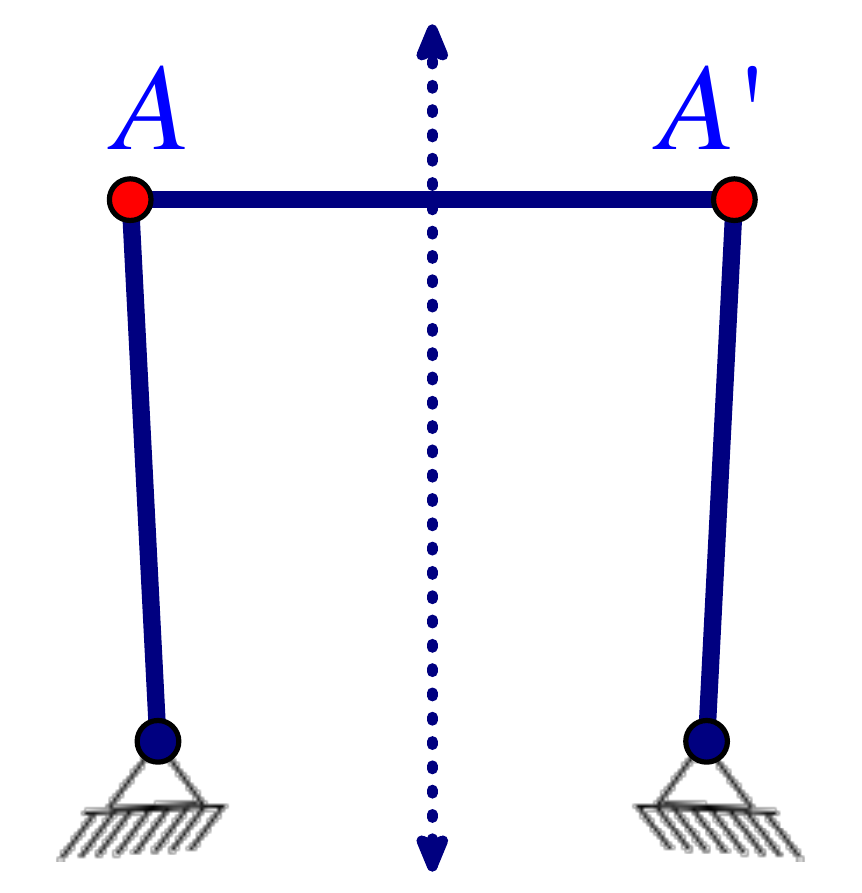}} \quad
  \subfigure[] {\includegraphics [width=.18\textwidth]{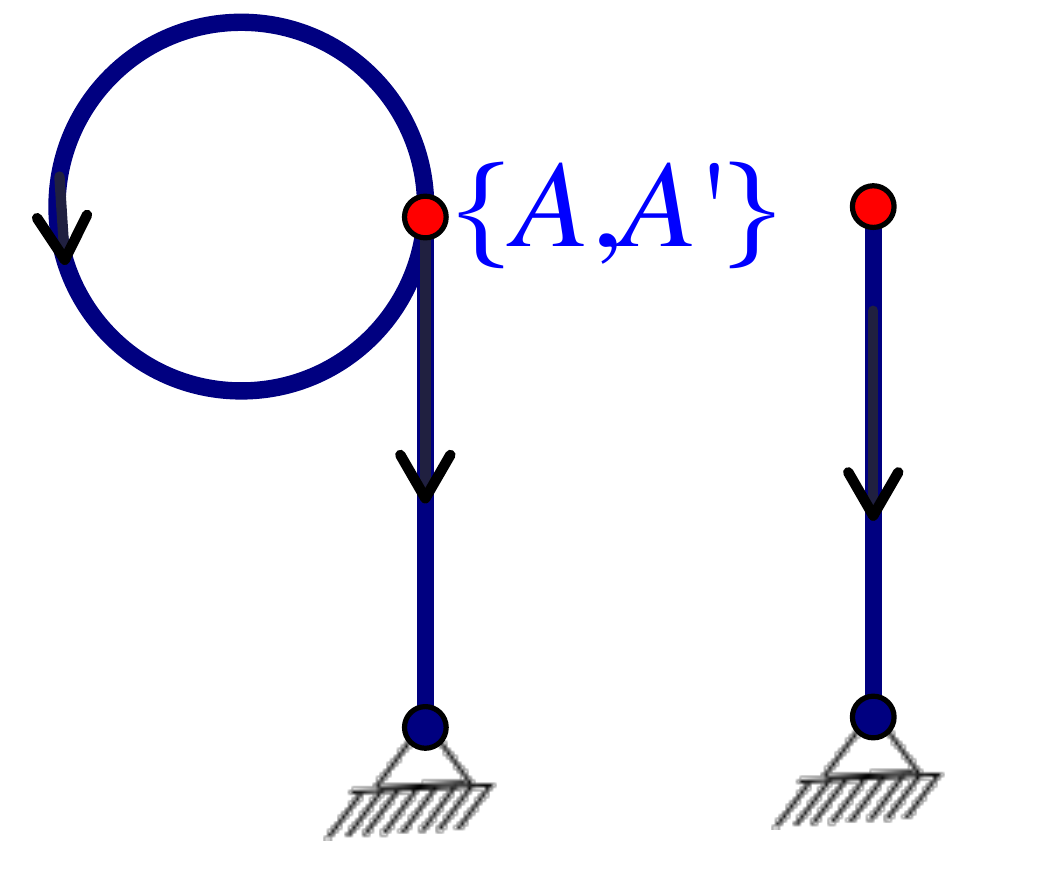}}
  \end{center}
    \caption{(a, b) A $\mathcal{C}_s$-Assur graph which is redundant without symmetry (in $3$-space); (c, d) A $\mathcal{C}_s$-Assur graph  with asymmetric motions (in the plane).}
    \label{fig:nonisostatic}
    \end{figure}

More generally, there are symmetry groups and associated group actions which fix vertices, edges, or both, where: (i) the \SG-Assur components may be pinned redundant, but rigid at \SG-regular configurations;   (ii) the \SG-Assur components may be pinned flexible (with symmetry breaking flexes) but independent at \SG-regular configurations;  or (iii)  a \SG-Assur component may be both redundant and flexible.  We may also have a \SG-Assur decomposition in which some components are isostatic, some are redundant, and some are flexible.  In all of these cases, the covering graph does not have a $d$-Assur decomposition, for full comparison of Assur decompositions of the type described in Subsection \ref{subsec:subgroup}.   However, through exploring subgroups, we will see that two \SG-Assur components for one group may now combine to a single \RG-Assur component for a subgroup \RG~of \SG.

While some of these examples are extreme, there are also several key examples, already mentioned in Figure~\ref{fig:Mechanisms}(d) and \ref{fig:PlaneExample}(d)), which are used, or are usable, for controlling stable symmetric motions in mechanical engineering.  We will return to the analysis of these examples below, as evidence that these additional types do contribute to the analysis of mechanical linkages and can contribute to the synthesis of new mechanical linkages.

\subsection{\SG-Assur graphs which are redundant and rigid at \SG-regular realisations}\label{sec:redundant}

Variants of the graph in Figure~\ref{fig:TwoMirror3D2} 
are generically redundant and rigid with different choices of drivers.  They also have different \SG-Assur decompositions depending on the symmetry group used for the analysis.

\begin{Example} Consider the $3$-dimensional pinned framework of Figures~\ref{fig:TwoMirror3D1}, \ref{fig:TwoMirror3D2}.  The $\mathcal{C}_{2v}$-symmetric graphs all have $|\hat{I}| = 2$ and the number of orbit matrix columns is equal to $4$. In (a) $|\hat{E}| = 3$ so there is a fully $\mathcal{C}_{2v}$-symmetric motion, while in (b), (c) $|\hat{E}| = 4$ and the graphs are \SG-isostatic.  However, they have generic counts such as $|E|=3|I|$ (Figure~\ref{fig:TwoMirror3D2} (a)),  as $|E|=3|I|+1$ (Figure~\ref{fig:TwoMirror3D1} (c)) and as $|E|=3|I|+4$ (Figure~\ref{fig:TwoMirror3D2} (c)).
\begin{figure}[ht]
    \begin{center}
  \subfigure[] {\includegraphics [width=.22\textwidth]{TwoMirror3DFlex}} \quad
\subfigure[] {\includegraphics [width=.22\textwidth]{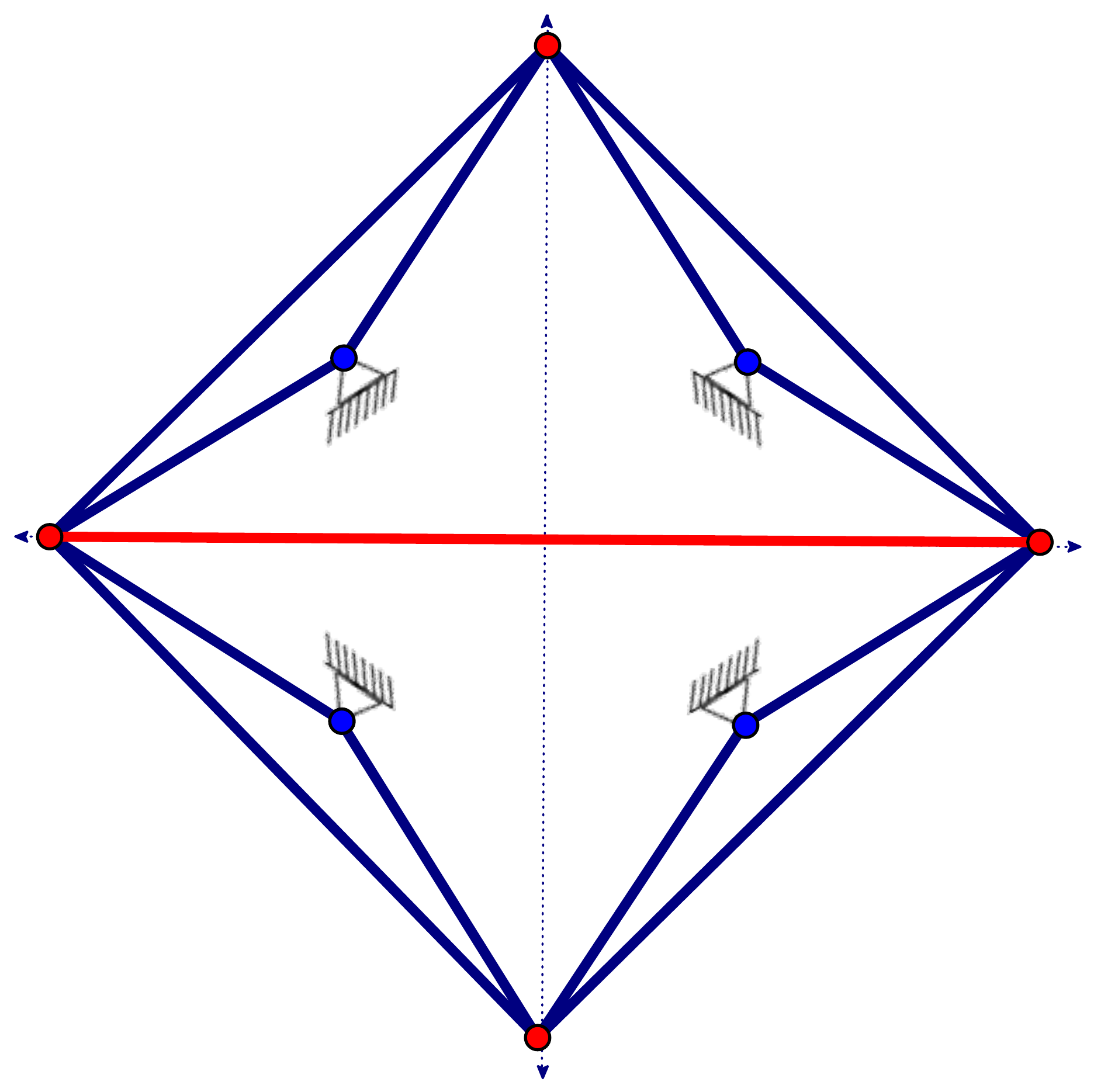}}\quad
\subfigure[] {\includegraphics [width=.22\textwidth]{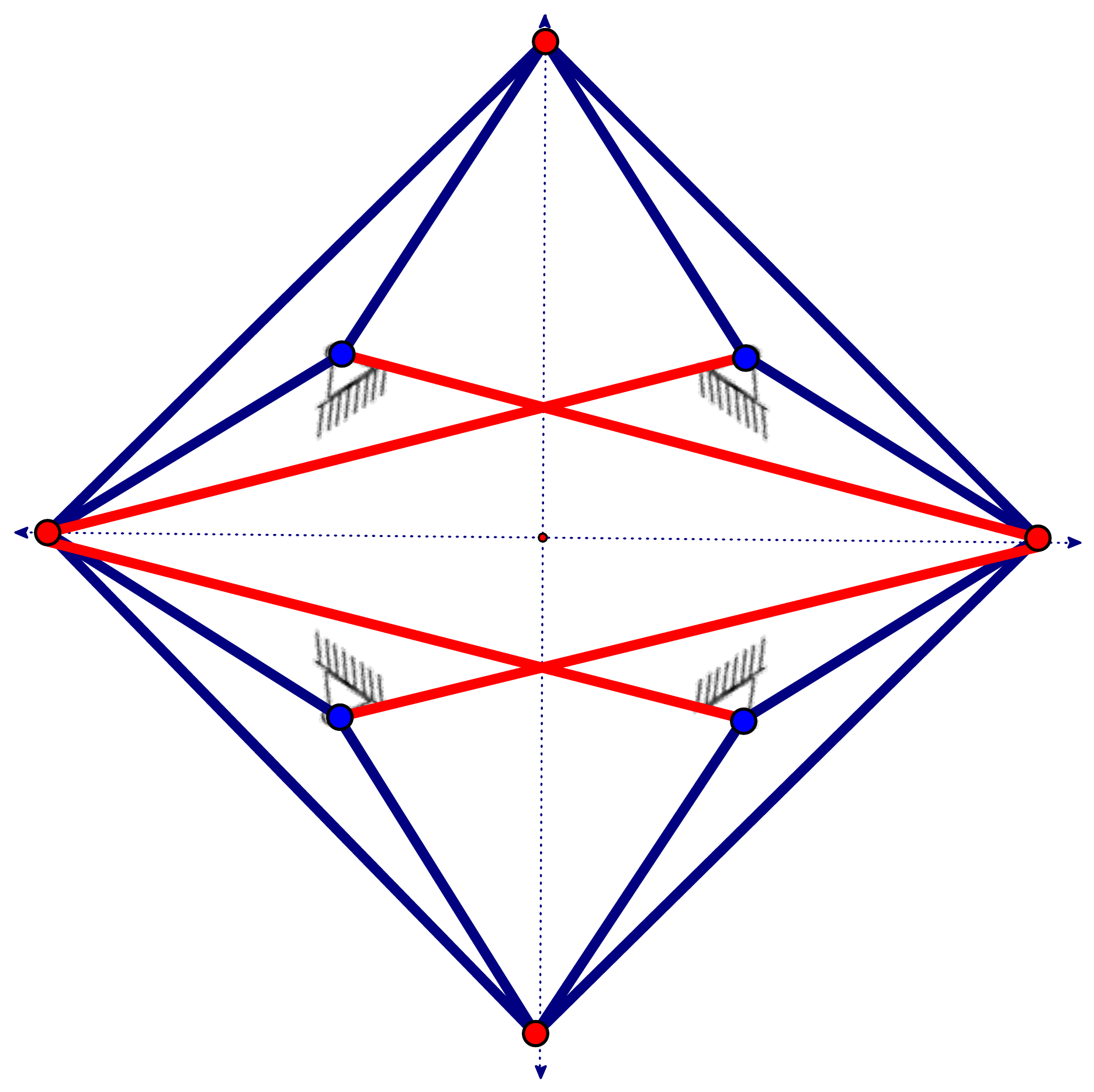}}\quad
\subfigure[] {\includegraphics [width=.22\textwidth]{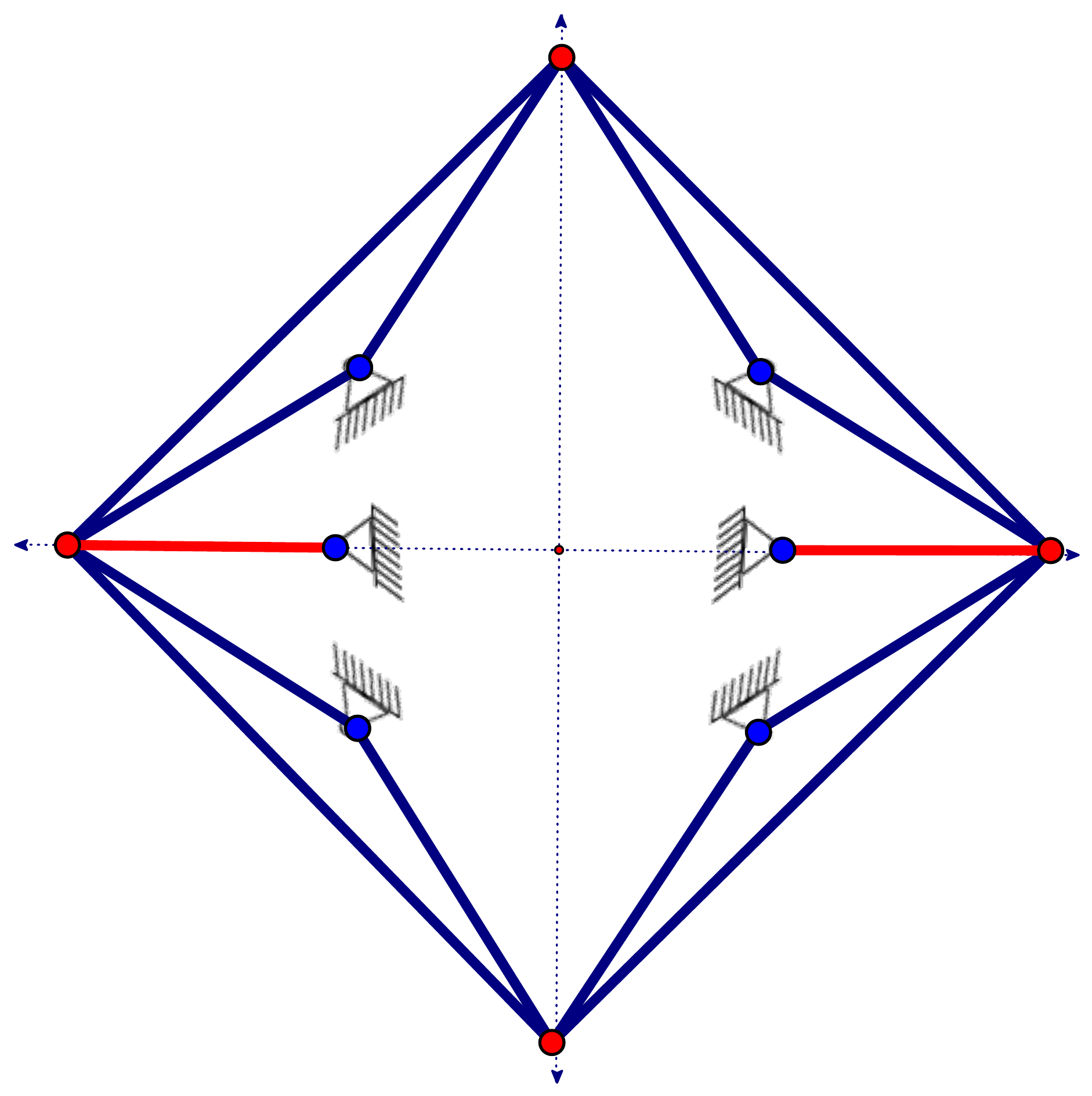}}
      \end{center}
    \caption{The framework in 3-space with dihedral symmetry $\mathcal{C}_{2v}$ (a) has a fully  $\mathcal{C}_{2v}$-symmetric motion.  Symmetric drivers (in red) can be  added in several ways (b) (c) creating distinct $\mathcal{C}_{2v}$-Assur graphs for controlling this symmetric motion.}
    \label{fig:TwoMirror3D2}
    \end{figure}

 These examples leave questions about their \SG-Assur decompositions, as well as what happens when we focus on one of the mirrors for the subgroup $\mathcal{C}_{s}$.  Figure~\ref{fig:TwoMirror3D2Orbit} illustrates that when we move down to a subgroup \RG~ of \SG, two \SG-Assur components can combine into a single \RG-Assur component.  This is quite different from  the behaviour guaranteed in Section~\ref{subsec:subgroup} when there is a free action.

 \begin{figure}[ht]
    \begin{center}
\subfigure[] {\includegraphics [width=.24\textwidth]{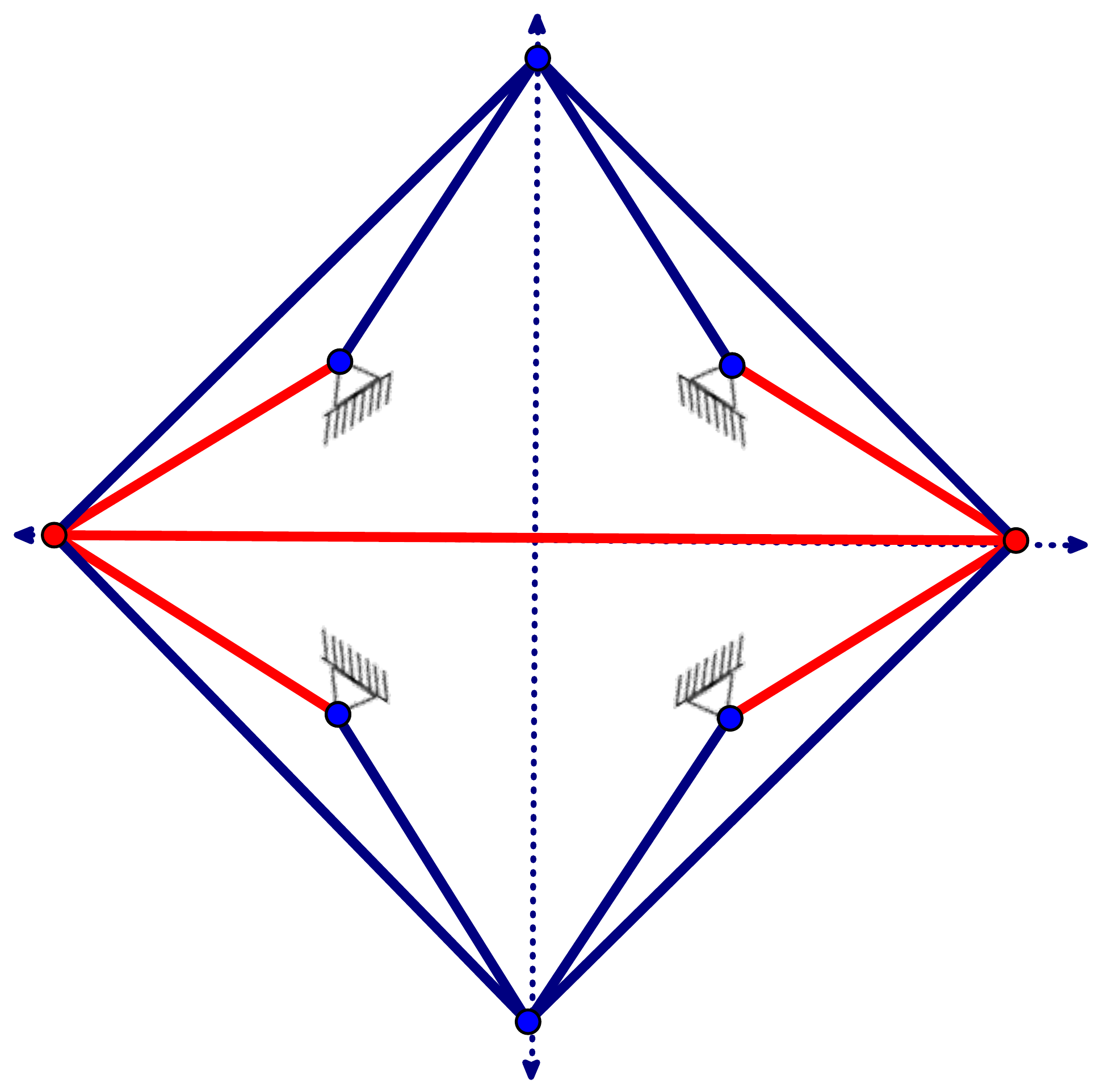}}
  \subfigure[] {\includegraphics [width=.20\textwidth]{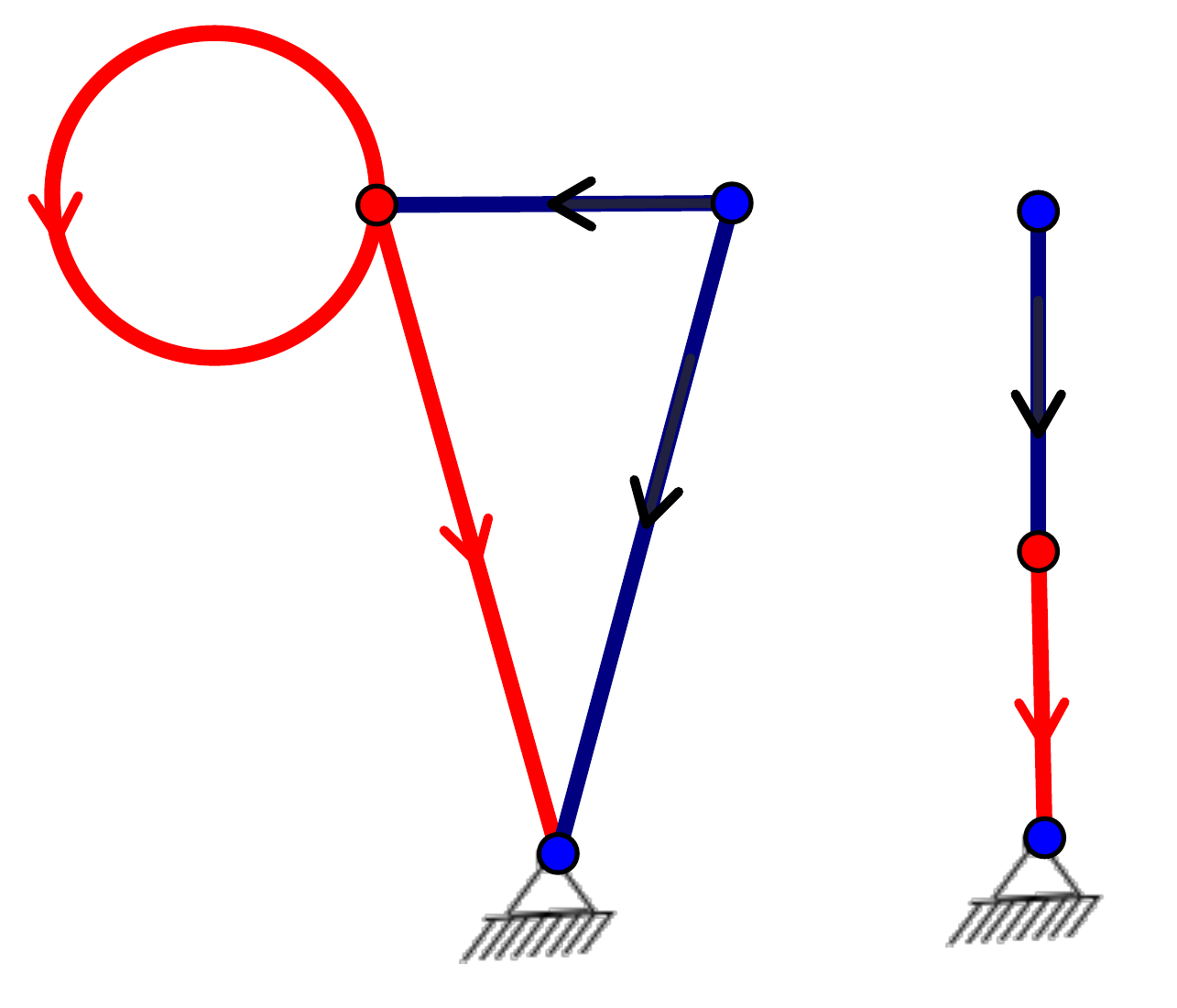}} \quad\quad
\subfigure[] {\includegraphics [width=.24\textwidth]{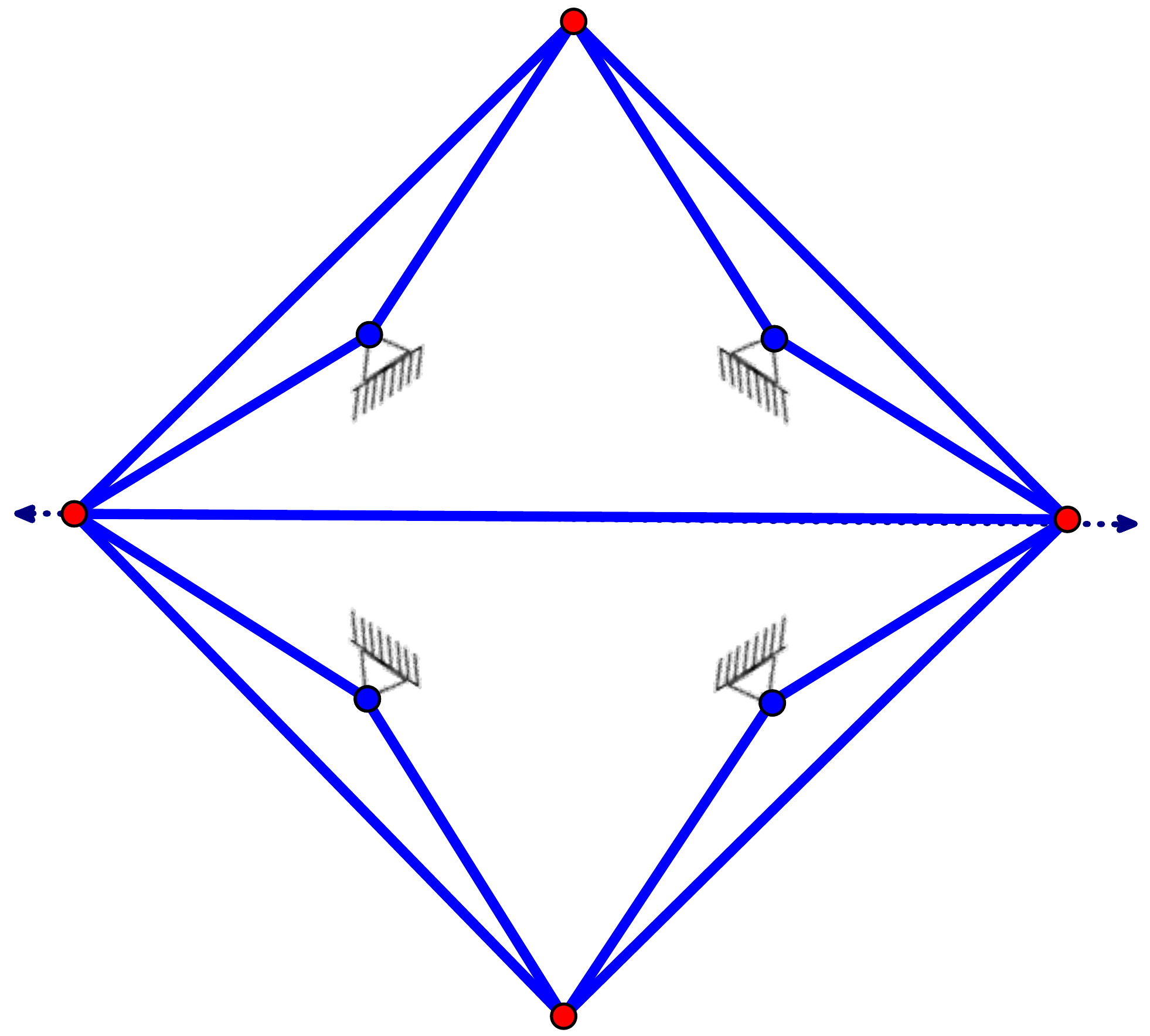}}\quad
\subfigure[] {\includegraphics [width=.20\textwidth]{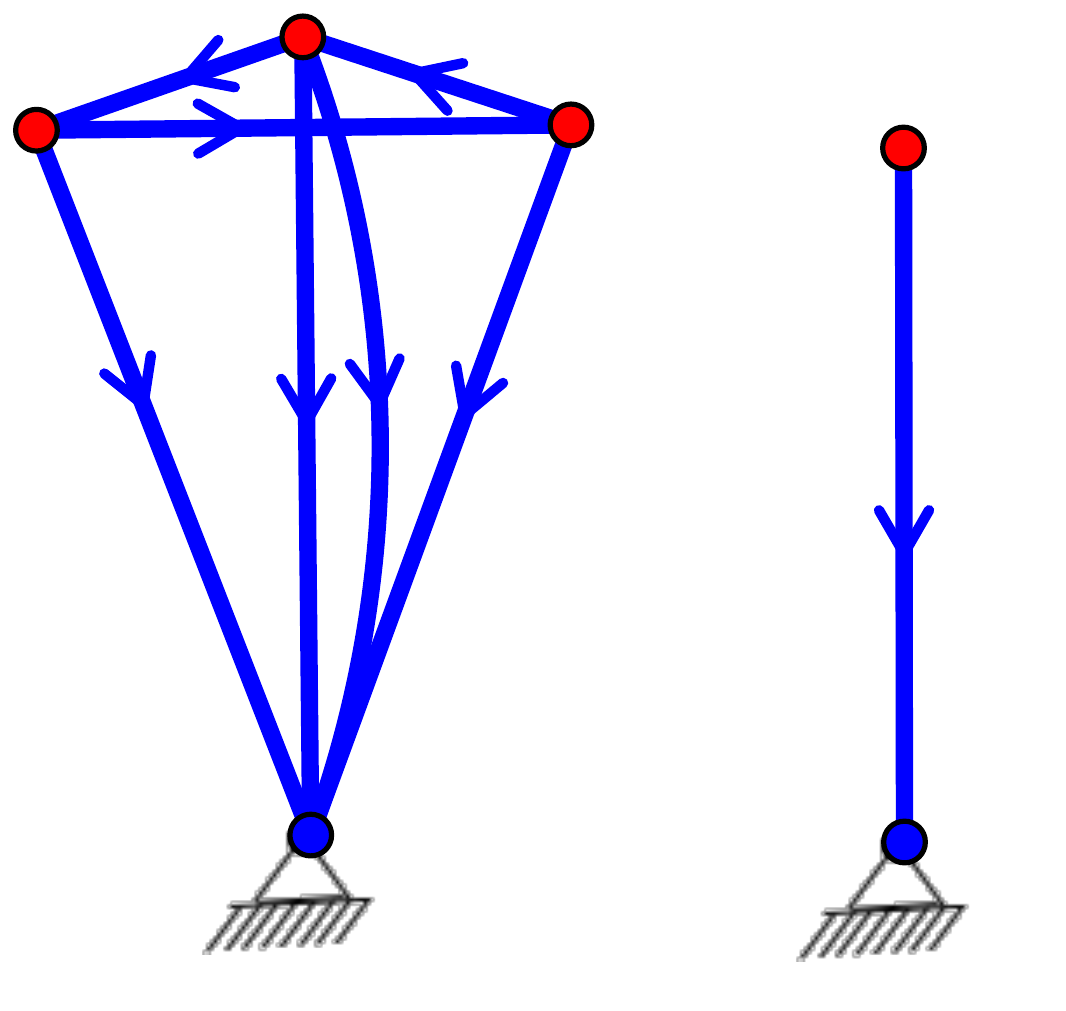}}\quad
      \end{center}
    \caption{The framework in 3-space with dihedral symmetry $\mathcal{C}_{2v}$ (a) has a $\mathcal{C}_{2v}$-Assur decomposition (b). The same framework in 3-space with a single horizontal mirror (c) is a $\mathcal{C}_{s}$-Assur graph (d). }
    \label{fig:TwoMirror3D2Orbit}
    \end{figure}

\end{Example}



\subsection{\SG-isostatic graphs with all \SG-regular realisations flexible without forced symmetry}\label{subsec:forcedflex}

We already saw that a \SG-isostatic graph may be flexible without forced symmetry, due to a fixed edge (Figure~\ref{fig:nonisostatic}~(c,d)).
There is a slightly different type of flexibility which happens when the framework, along with the ground, is generically isostatic, but the number of pins needed under symmetry is too small to eliminate all trivial motions.  The following example illustrates this.

\begin{Example}\label{ex:GrabBucket}

Consider a schematic of a Grab Bucket in Figure~\ref{fig:grab} which has a single mirror in the plane \cite[p~270]{kinematics}.
The quotient $\mathcal{C}_s$-gain graph in Figure~\ref{fig:grab} (b) confirms this has the correct count to be pinned $\mathcal{C}_s$-isostatic with $|\hat{E}|=7$,  $|\hat{I}|=4$, and, with the fixed vertex $u$, 7 columns in the orbit matrix.     The  $\mathcal{C}_s$-Assur block graph in (c) illustrates that the red edge is a driver that moves all the inner vertices.

\begin{figure}[ht]
    \begin{center}
 \subfigure[] {\includegraphics [width=.20\textwidth]{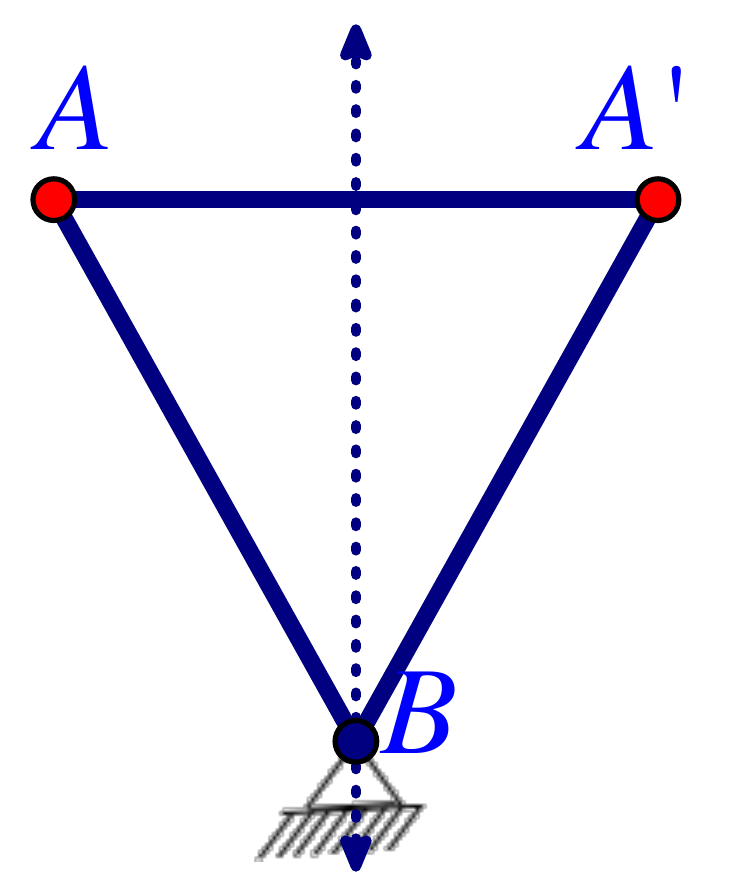}} \quad
  \subfigure[] {\includegraphics [width=.32\textwidth]{GrabBucketS}} \quad
  \subfigure[] {\includegraphics [width=.22\textwidth]{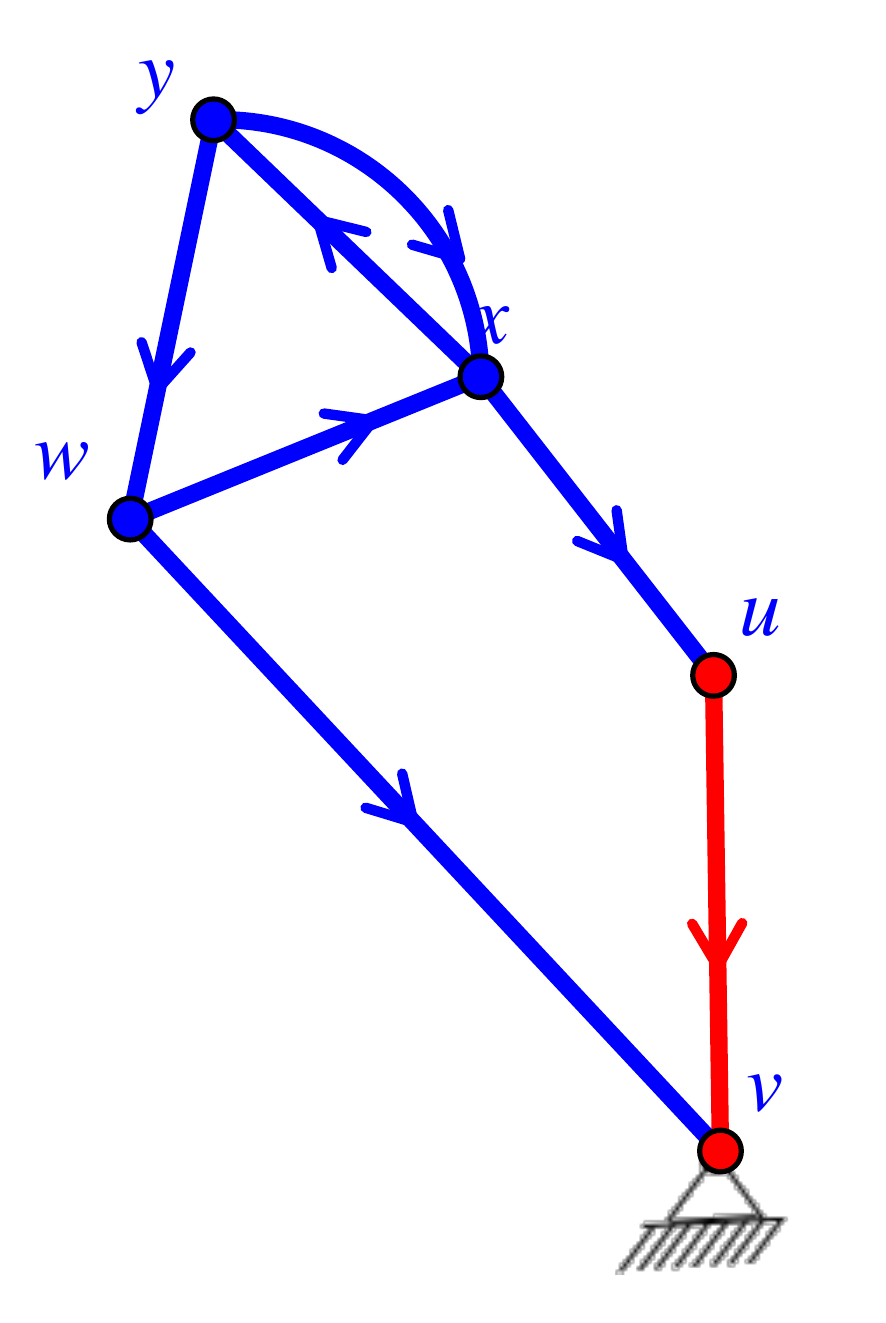}} \quad
  \subfigure[] {\includegraphics [width=.09\textwidth]{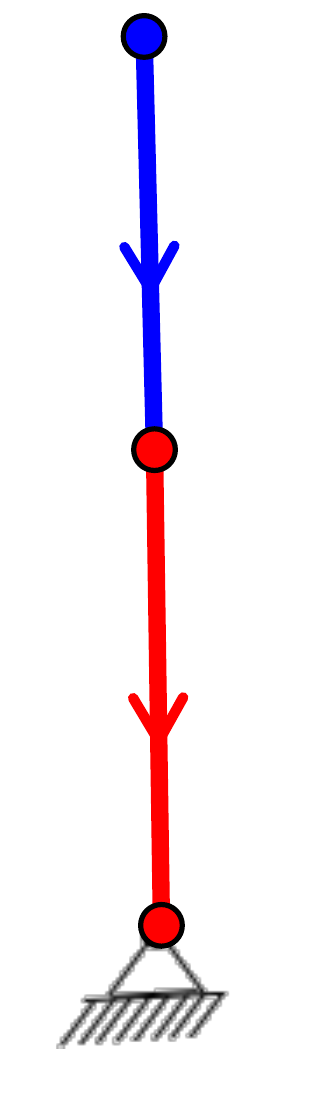}}
      \end{center}
    \caption{Figure (a) gives the simple case of a triangle with a fixed vertex as the single pin, which is not pinned isostatic.  The rest of the figure is a plane presentation of a side of the Grab Bucket which has a single mirror, and a `ground' vertex at the top (where it would be attached to a crane by a cable).  The red edge is used as a driver for a single mirror symmetric motion.}
    \label{fig:grab}
    \end{figure}

The covering graph, without forced symmetry is independent, but does not have sufficient pins to make it pinned 2-isostatic.  There will be a trivial `swing'   of the bucket around the top pin - something that is observed on any construction site.  This swing is controlled by gravity pulling down on the bucket - and the desired motion is the mirror symmetric operation of the bucket which is controlled by the driver.  Our analysis does contribute clarifying information about the symmetric motion.

If we modify the pinning and make both of the two vertices on the mirror into the ground, then what remains is a pinned $2$-isostatic graph ($|E|=12, \ |I|=6$), without forced symmetry.  What this analysis, not based on the symmetry, misses is the symmetric impact of changing the distance between these two pins!  From the point of view of the \SG-Assur components, it is the bottom component - a single bar - which is pinned \SG-isostatic, but not pinned isostatic.

In practice, this image is one side of a $3$-dimensional framework, which has an added mirror symmetry between the two sides.  In some images on the internet, the `ground' is expanded to two grounding sites, and the bucket retains a single swinging motion in addition to the driver-controlled motion which preserves both mirrors.  In others, the ground is pulled to a single pin, and the second vertex $u$ may be on the intersection of the two mirrors or doubled (symmetrically).  There is still a single fully-symmetric motion, and then other incidental motions due to having insufficient ground vertices.
\end{Example}

There is a variant of the Grab Bucket which is the bottom component of Figure~\ref{fig:PlaneExample}~(d).
\begin{Example}\label{ex:PlaneMirrorGrab}

Consider the graph in Figure \ref{fig:PlaneMirrorGrab} (a), with $\mathcal{C}_s$ symmetry and one fixed vertex on the mirror.
The direct analysis of the gain graph, and the orbit matrix, shows this is a $\mathcal{C}_s$-Assur graph.   However, a direct count of the covering graph shows $|{E}|=2|{I}|$, while there is a non-trivial symmetry breaking motion with just the one vertex pinned.  This guarantees an anti-symmetric self-stress.   With the red edges as coordinated drivers, we see an alternative mechanism  for a Grab Bucket.

\begin{figure}[ht]
    \begin{center}
  \subfigure[] {\includegraphics [width=.30\textwidth]{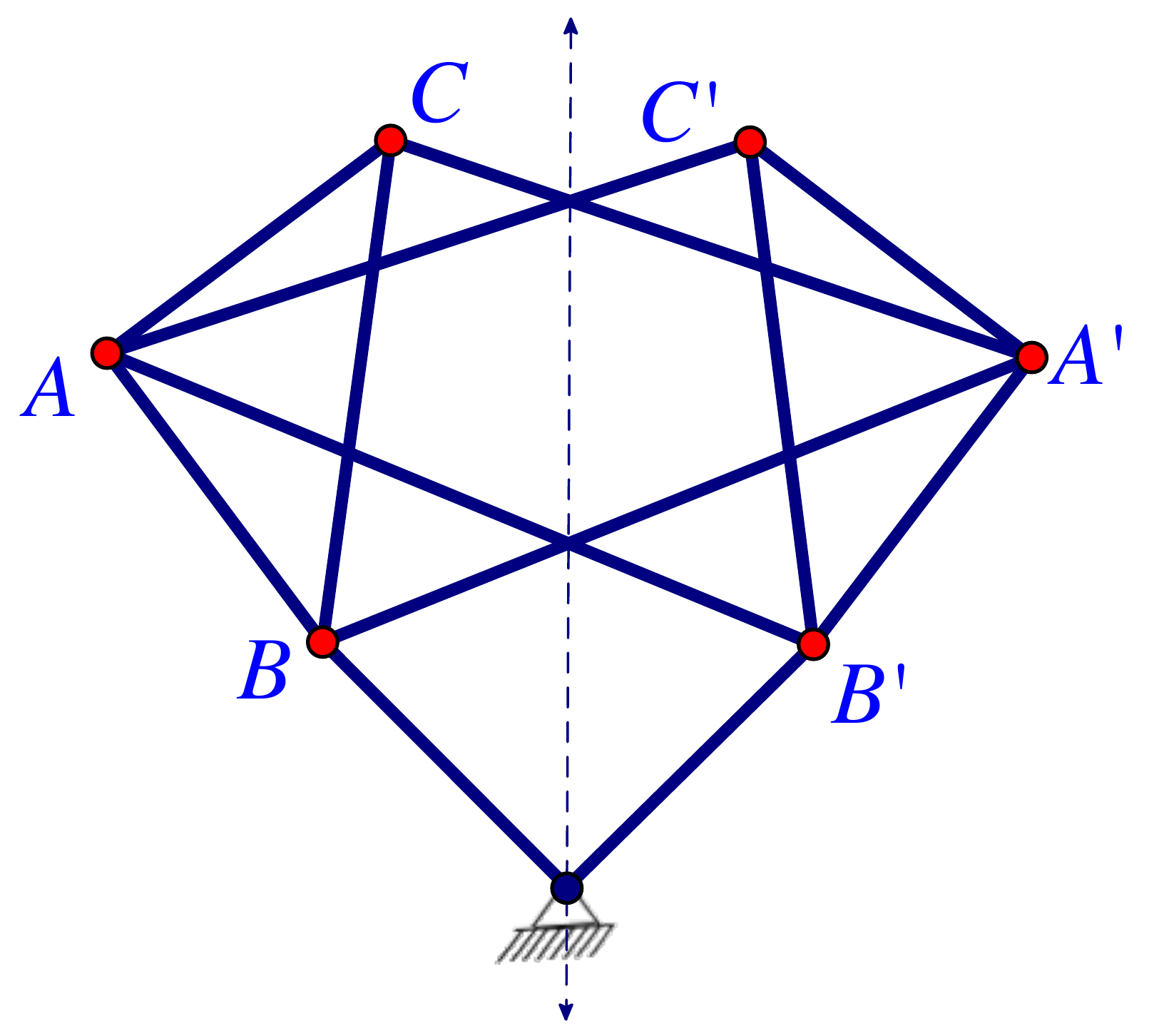}} \quad
  \subfigure[] {\includegraphics [width=.30\textwidth]{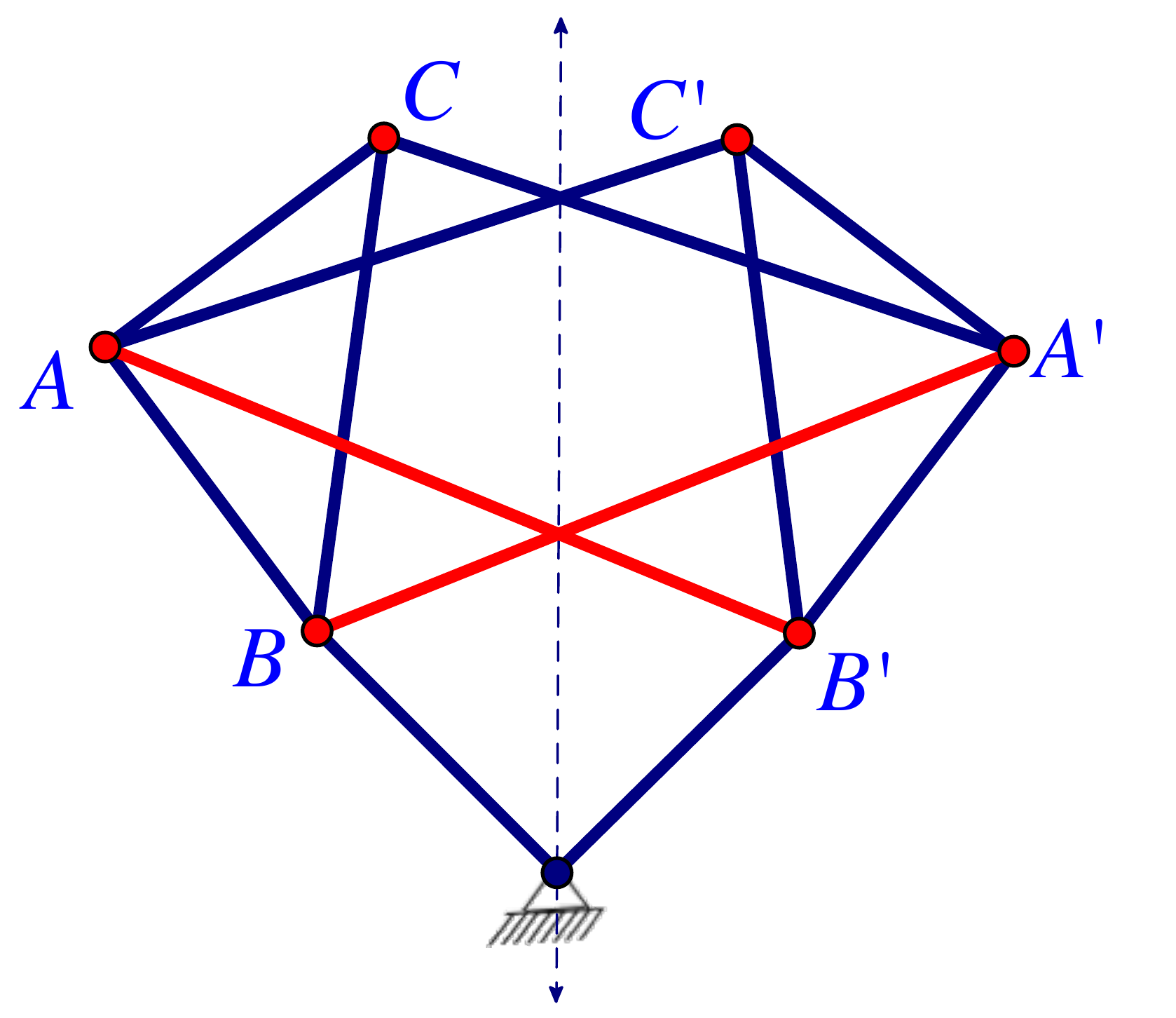}} \quad
      \end{center}
    \caption{The plane framework with $\mathcal{C}_s$ symmetry (a) is $\mathcal{C}_s$-Assur, but has a self-stress and is not pinned isostatic.  Figure (b) illustrates a possible choice of drivers in red for a $\mathcal{C}_s$-symmetric motion. }
    \label{fig:PlaneMirrorGrab}
    \end{figure}
\end{Example}

Finally, we present an extreme example of where our definitions of \SG-Assur can  take us.

\begin{Example}
The framework in Figure~\ref{fig:D_2}(a) is pinned $\mathcal{C}_{2v}$-isostatic - with no edge orbits to the ground.   This occurs because there are no fully $\mathcal{C}_{2v}$-symmetric trivial motions.
\begin{figure}[ht]
    \begin{center}
\subfigure[] {\includegraphics [width=.20\textwidth]{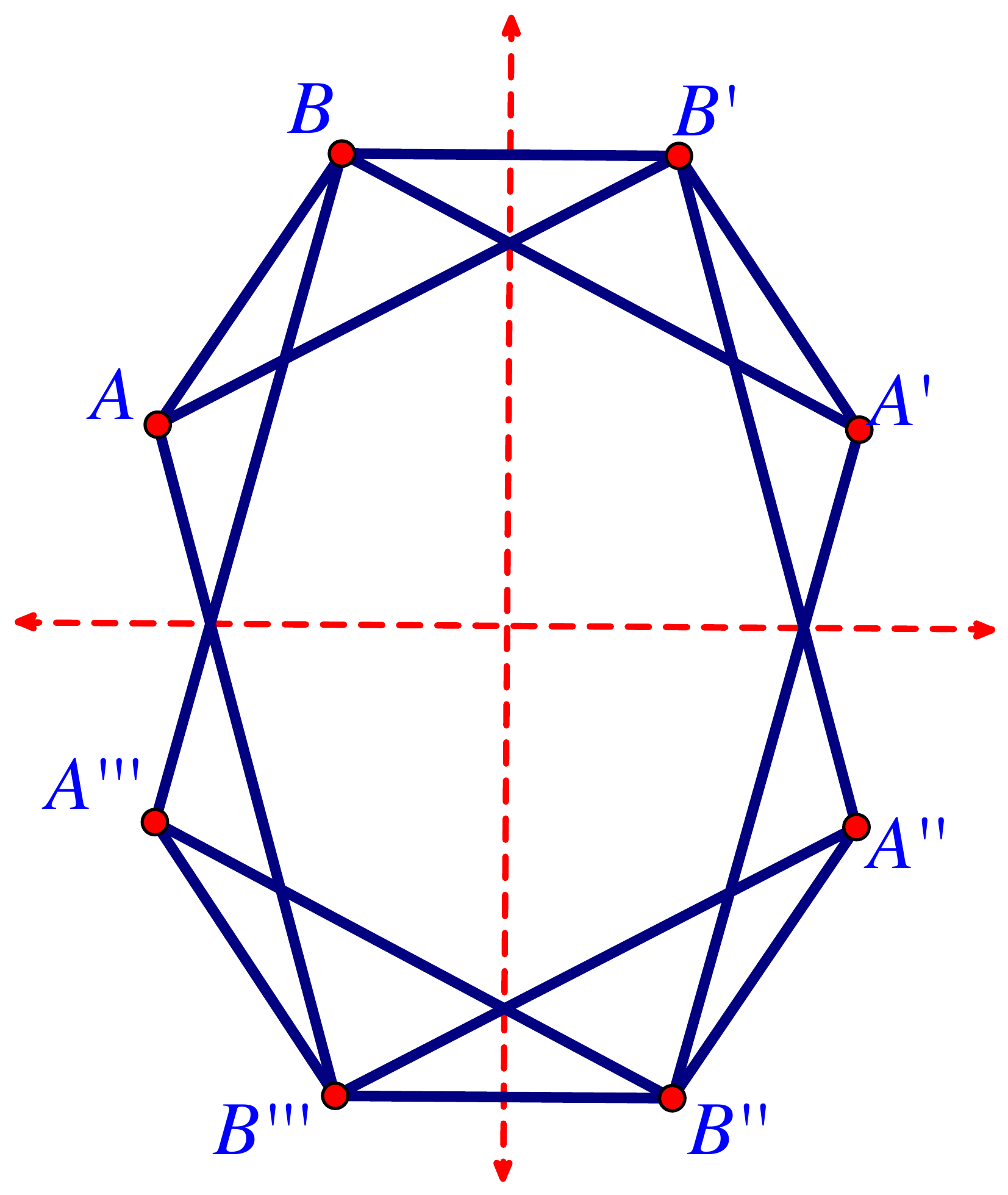}} \quad
 \subfigure[] {\includegraphics [width=.15\textwidth]{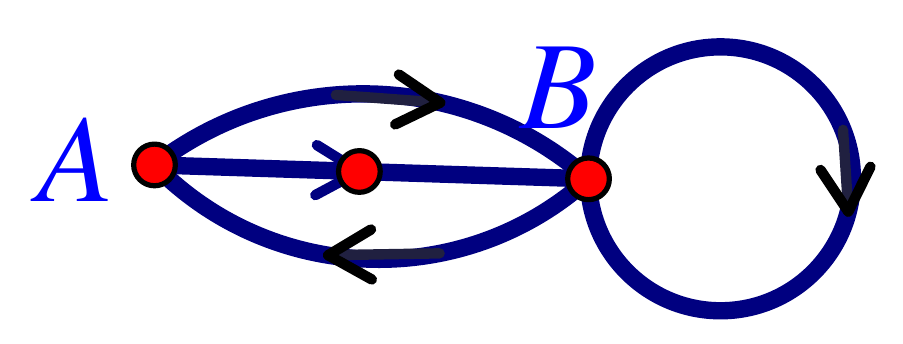}} \quad
  \subfigure[] {\includegraphics [width=.25\textwidth]{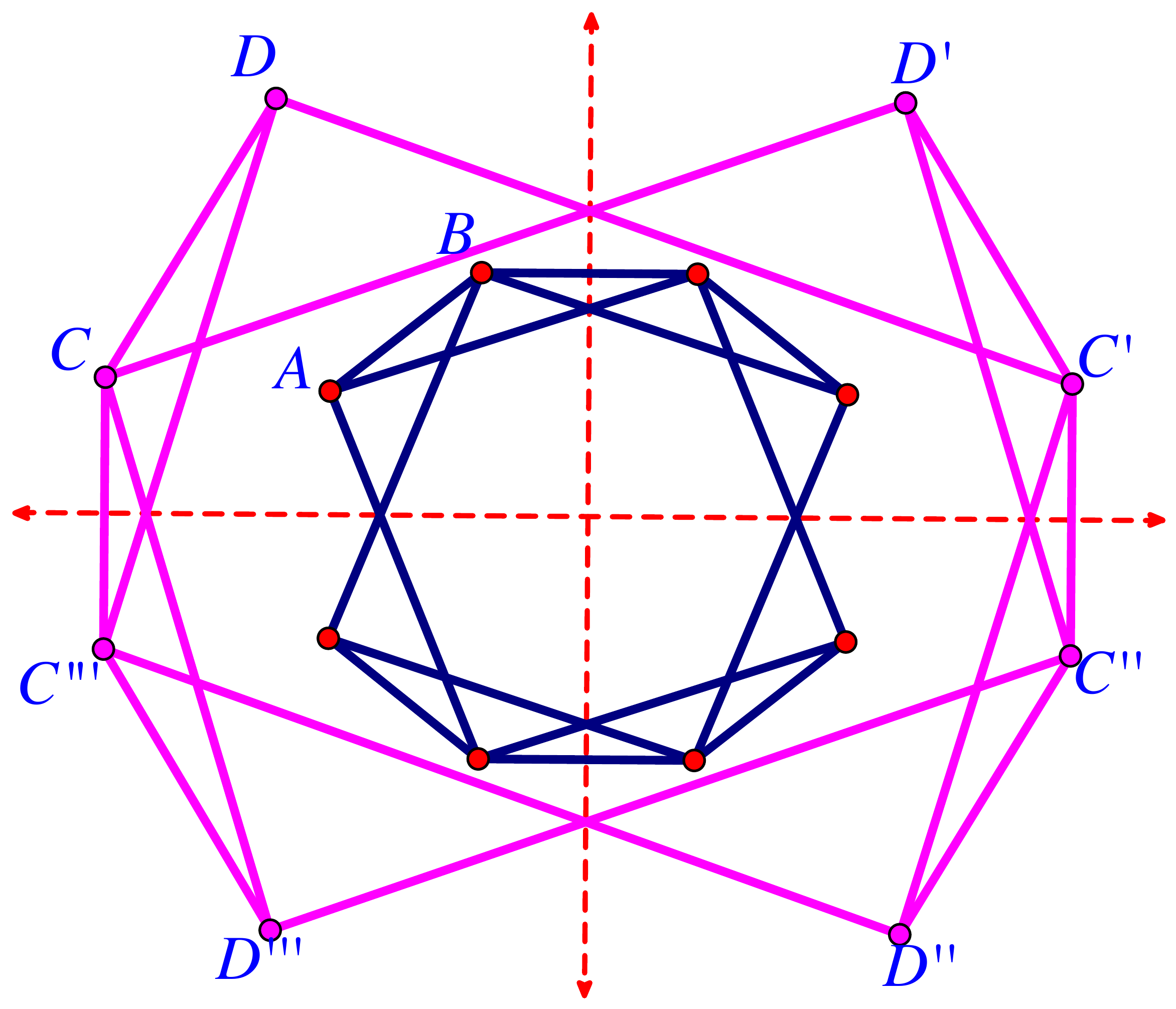}} \quad
  \subfigure[] {\includegraphics [width=.15\textwidth]{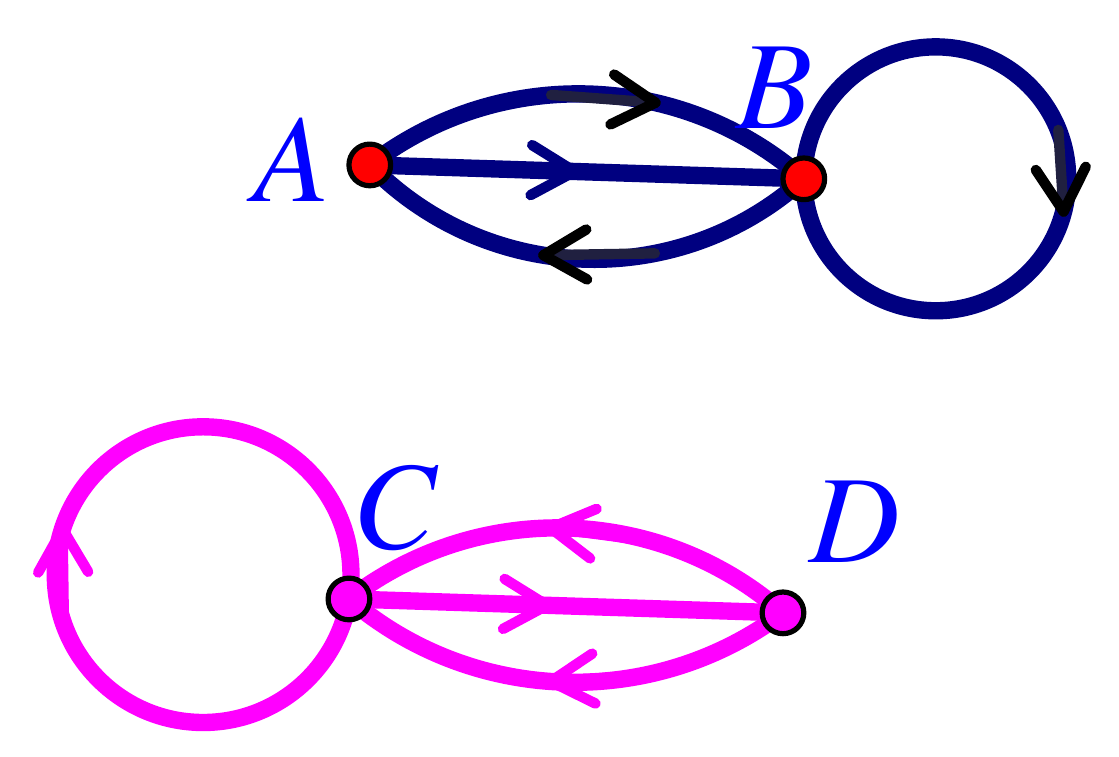}} \quad
 \subfigure[] {\includegraphics [width=.10\textwidth]{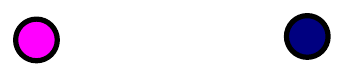}}
     \end{center}
    \caption{The framework (a) has $\mathcal{C}_{2v}$ symmetry, and the 2-directed quotient $\mathcal{C}_{2v}$-gain graph (b) which is independent, and maximal rank.  If we double the graph (c) with $\mathcal{C}_{2v}$ symmetry, it remains full rank in the orbit matrix, with the gain graph (d) and $\mathcal{C}_{2v}$-Assur block graph (e)}
    \label{fig:D_2}
    \end{figure}

However it is not pinned rigid, as it is unpinned when non-symmetric motions are allowed, with rotations around the origin, as well as translations.    With two copies  Figure~\ref{fig:D_2}(c), this is far from rigid once motions breaking the symmetry are permitted.

\end{Example}




\subsection{\SG-isostatic graphs with combined components}\label{subsec:comp}

We can combine components with different properties, some with isostatic graphs, some with redundant graphs, some with flexible graphs.

\begin{Example}  Consider the $2$-dimensional framework with mirror symmetry shown  in Figure~\ref{fig:Composite} (a), with its gain graph Figure~\ref{fig:Composite}(b) and  its $\mathcal{C}_s$-Assur block graph  in Figure~\ref{fig:Composite}(c). The underlying graph of the framework in Figure~\ref{fig:Composite}(a) has a generically dependent lower component,  a generically isostatic component (the middle) and a generically flexible upper component.  Even the lower component is also  not generically pinned rigid (it  does not have enough pins).   
\begin{figure}[ht]
    \begin{center}
  \subfigure[] {\includegraphics [width=.25\textwidth]{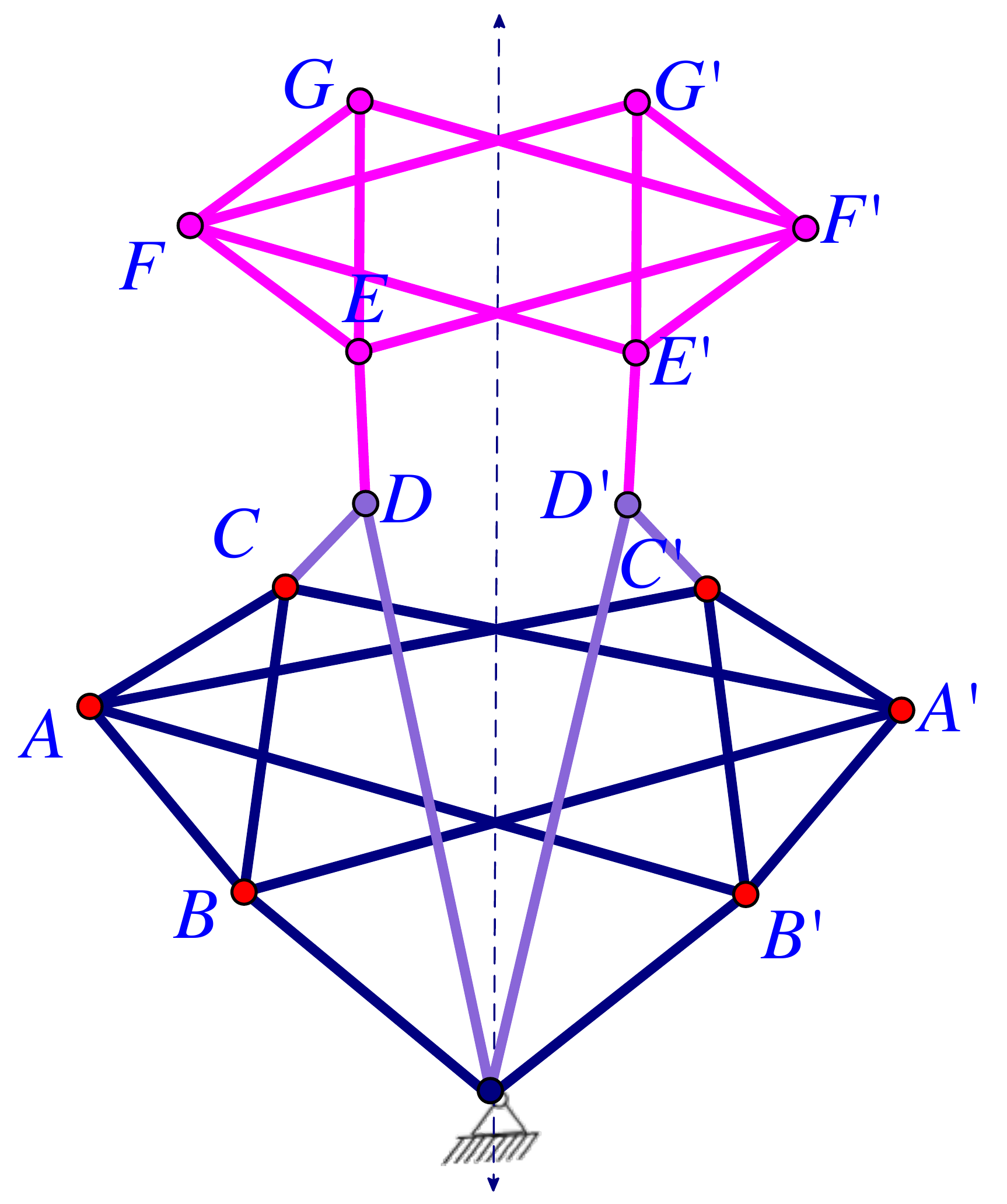}} \quad\quad
\subfigure[] {\includegraphics [width=.13\textwidth]{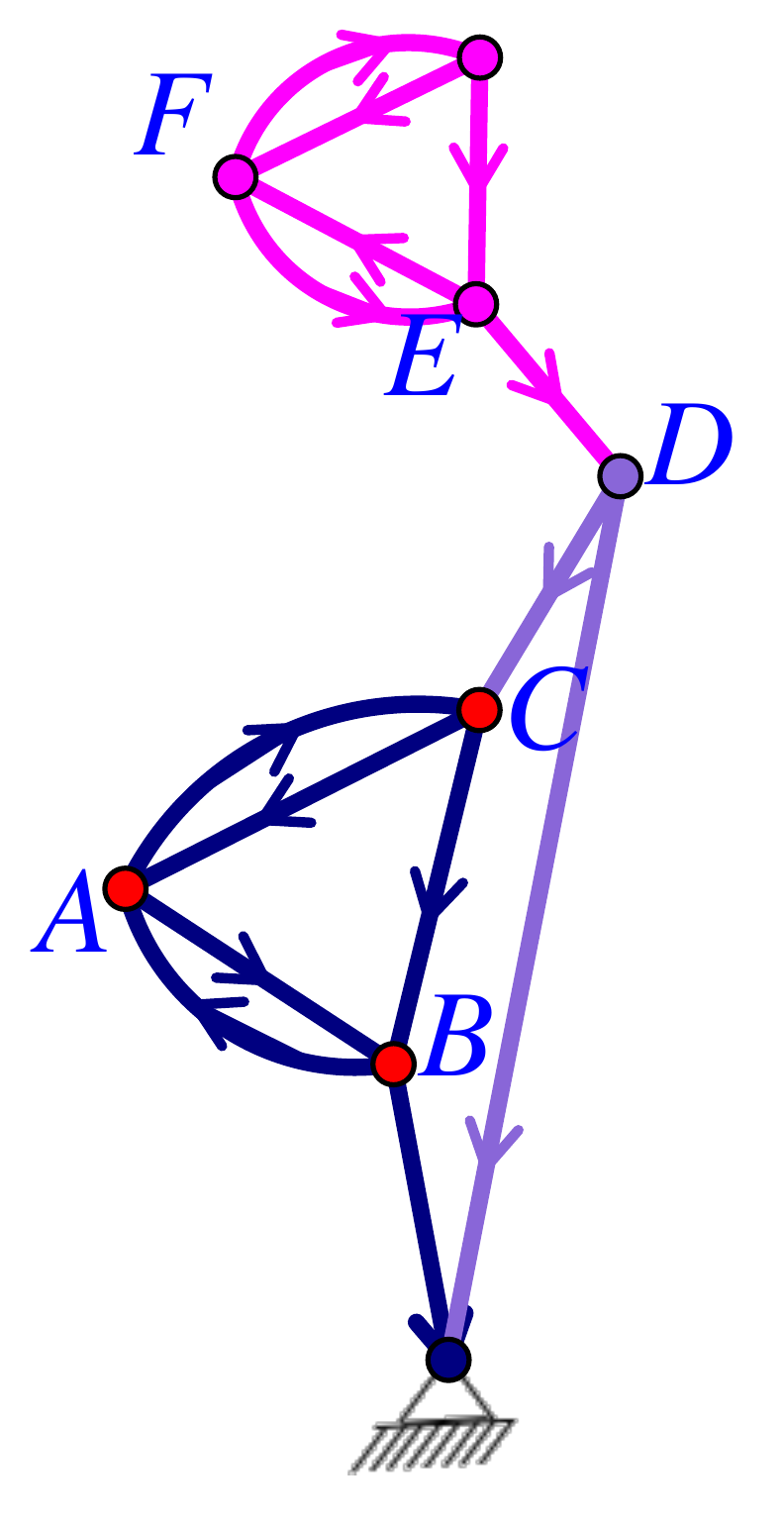}}\quad\qquad
  \subfigure[] {\includegraphics [width=.13\textwidth]{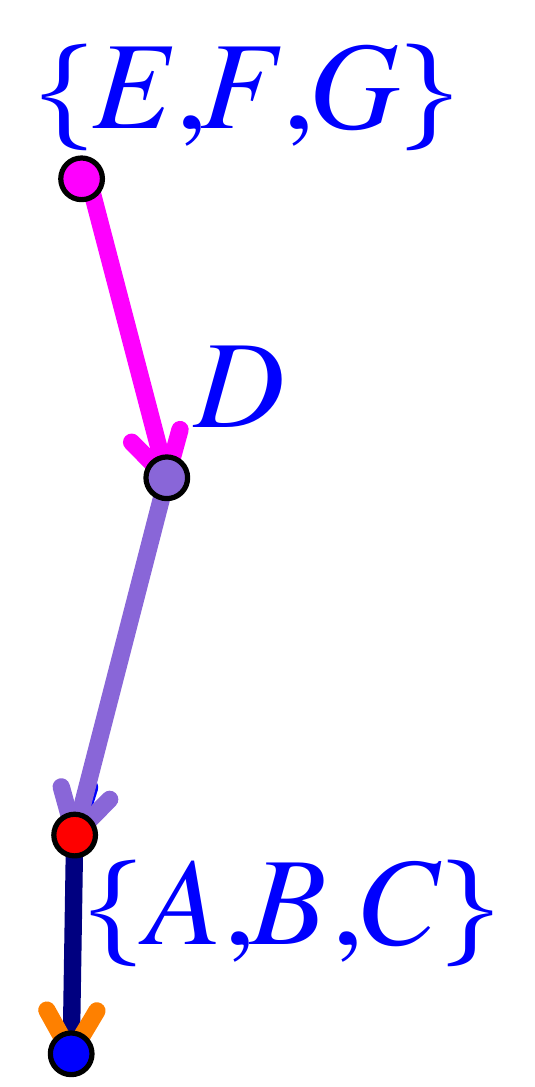}}
      \end{center}
    \caption{A mirror symmetric framework in the plane (a),  with its $\mathcal{C}_s$-directed gain graph (b), and the associated $\mathcal{C}_s$-Assur block graph (c) has redundant components and  flexible components as well as too few pins, when symmetry is relaxed.}
    \label{fig:Composite}
    \end{figure}

     Such composite structures can be synthesized by composing various \SG-Assur graphs with the same symmetry. We might even compose several components with different symmetries, provided the attachments from the upper component(s) to the lower ones had the symmetry of the upper components.  We have only scratched the surface of what can happen and what can be designed.
\end{Example}

\section{Inductive Constructions}
\label{sec:Inductions}

Inductive constructions can be a tool for synthesis of mechanisms, including synthesis of larger Assur components.  We present the basic idea for \SG-Assur decompositions, and the reader can pull this back to new results without symmetry, when \SG~ is the identity group. The reader interested in inductive constructions can consult \cite{N&R}.

The following inductive constructions  on quotient gain graphs, called \emph{extensions}, preserve $(2,3,m)$-gain-sparsity.
The first two operations are generalizations of the well-known Henneberg operations \cite{W1}.

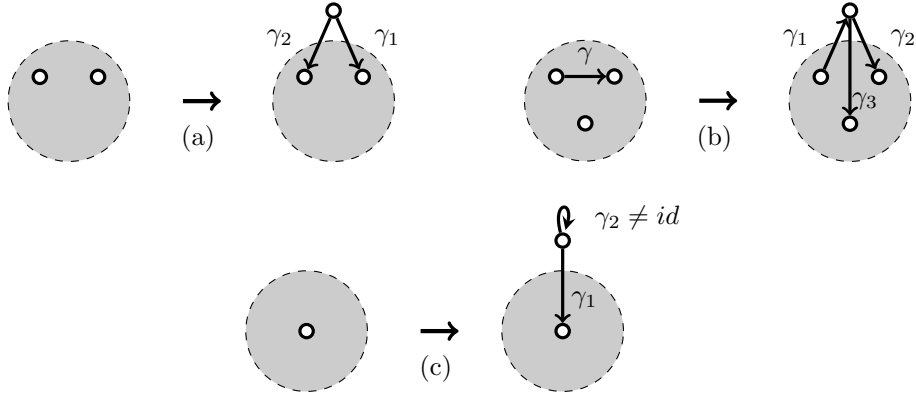
\begin{figure}[htp]
\begin{center}
\begin{tikzpicture}[very thick,scale=1]
\tikzstyle{every node}=[circle, draw=black, fill=white, inner sep=0pt, minimum width=5pt];
\filldraw[fill=black!20!white, draw=black, thin, dashed](0,0)circle(0.8cm);
\node (p1) at (40:0.5cm) {};
\node (p2) at (140:0.5cm) {};
  \draw[ultra thick, ->](1.5,0)--(2,0);
   \node [draw=white, fill=white] (a) at (1.7,-0.5)  {(a)};
\end{tikzpicture}
\hspace{0.3cm}
\begin{tikzpicture}[very thick,scale=1]
\tikzstyle{every node}=[circle, draw=black, fill=white, inner sep=0pt, minimum width=5pt];
\filldraw[fill=black!20!white, draw=black, thin, dashed](0,0)circle(0.8cm);
\node (p1) at (40:0.5cm) {};
\node (p2) at (140:0.5cm) {};
\node (p3) at (90:1.2cm) {};
\draw [->](p3)--(p1);
\draw [->](p3)--(p2);
\node [draw=white, fill=white] (a) at (50:1.1cm)  {$\gamma_1$};
\node [draw=white, fill=white] (a) at (130:1.1cm)  {$\gamma_2$};
\end{tikzpicture}
\hspace{1.35cm}
\begin{tikzpicture}[very thick,scale=1]
\tikzstyle{every node}=[circle, draw=black, fill=white, inner sep=0pt, minimum width=5pt];
\filldraw[fill=black!20!white, draw=black, thin, dashed](0,0)circle(0.8cm);
\node (p1) at (40:0.5cm) {};
\node (p2) at (140:0.5cm) {};
\node (p4) at (270:0.3cm) {};
 \draw [<-](p1)--(p2);
\draw[ultra thick, ->](1.5,0)--(2,0);
\node [draw=black!20!white, fill=black!20!white] (a) at (90:0.6cm)  {$\gamma$};
   \node [draw=white, fill=white] (a) at (1.7,-0.5)  {(b)};
\end{tikzpicture}
\hspace{0.3cm}
\begin{tikzpicture}[very thick,scale=1]
\tikzstyle{every node}=[circle, draw=black, fill=white, inner sep=0pt, minimum width=5pt];
\filldraw[fill=black!20!white, draw=black, thin, dashed](0,0)circle(0.8cm);
\node [draw=black!20!white, fill=black!20!white] (a) at (0:0.2cm)  {$\gamma_3$};
\node (p1) at (40:0.5cm) {};
\node (p2) at (140:0.5cm) {};
\node (p3) at (90:1.2cm) {};
\node (p4) at (270:0.3cm) {};
\draw [->](p3)--(p1);
\draw [<-](p3)--(p2);
\draw [->](p3)--(p4);
\node [draw=white, fill=white] (a) at (50:1.1cm)  {$\gamma_2$};
\node [draw=white, fill=white] (a) at (130:1.1cm)  {$\gamma_1$};
\end{tikzpicture}
\vspace{0.1cm}

\begin{tikzpicture}[very thick,scale=1]
\tikzstyle{every node}=[circle, draw=black, fill=white, inner sep=0pt, minimum width=5pt];
\filldraw[fill=black!20!white, draw=black, thin, dashed](0,0)circle(0.8cm);
\node (p4) at (90:0cm) {};
\draw[ultra thick, ->](1.5,0)--(2,0);
   \node [draw=white, fill=white] (a) at (1.7,-0.5)  {(c)};
\end{tikzpicture}
\hspace{0.3cm}   \begin{tikzpicture}[very thick,scale=1]
\tikzstyle{every node}=[circle, draw=black, fill=white, inner sep=0pt, minimum width=5pt];
\filldraw[fill=black!20!white, draw=black, thin, dashed](0,0)circle(0.8cm);
\node (p3) at (90:1.2cm) {};
\node (p4) at (90:0cm) {};
\draw [->](p3)--(p4);
\node [draw=black!20!white, fill=black!20!white] (a) at (55:0.5cm)  {$\gamma_1$};
\node [draw=white, fill=white] (a) at (57:1.8cm)  {$\gamma_2\neq id$};
\path
(p3) edge [loop above,->, >=stealth,shorten >=1pt,looseness=20] (p3);
\end{tikzpicture}
\vspace{-0.3cm}
\end{center}
\vspace{-0.2cm}
\caption{(a) 0-extension, where the  new edges may be parallel.
(b) 1-extension, where the removed edge may be a loop and the
new edges may be parallel. Note that for the gains $\gamma_1$ and $\gamma_2$, we have $\gamma_1 \cdot \gamma_2=\gamma$.
(c) loop-1-extension.}
\label{fig:inductive}
\end{figure}

Let $(H,\psi)$ be a $\Gamma$-gain graph with $H=(\tilde{V}, \tilde{E})$. The \emph{0-extension} adds a new vertex \(\tilde{v}\) and two new non-loop edges
$\te_{1}$ and $\te_{2}$ to $H$
such that the new edges are incident to $\tilde{v}$ and the other end-vertices are
two not necessarily distinct vertices of $\tilde{V}$.
If $\te_{1}$ and $\te_{2}$ are not parallel, then their labels can be arbitrary.
Otherwise the labels are assigned such that $\psi(\te_{1})\neq \psi(\te_{2})$,
assuming that $\te_1$ and $\te_2$ are directed to $\tilde{v}$ (see Fig.\ref{fig:inductive} (a)).

The \emph{1-extension} deletes an edge of $(H,\psi)$ and adds a new vertex and three new edges to $(H,\psi)$.
First, one chooses an edge $\te$ of $H$ (which will be deleted) and a vertex $\tilde z$ of $H$ which may be an end-vertex of $\te$.
Then one subdivides $\te$, with a new vertex $\tilde v$ and new edges $\te_1$ and $\te_2$, such that the tail of $\te_1$ is the tail
of $\te$ and the head of $\te_2$ is the head of $\te$. The gains of the new edges are assigned so that $\psi(e_1)\cdot \psi(e_2)=\psi(e)$.
 Finally, we add a third new edge, $e_3$, to $H$. This edge is oriented from $\tilde v$ to $\tilde z$ and its gain is such that every 2-cycle $\te_i \te_j$, if it exists, is unbalanced.

The \emph{loop 1-extension} (see Fig.\ref{fig:inductive} (c)). adds a new vertex $\tilde{v}$ to $H$
and connects it to a vertex $\tilde{z}\in \tilde{V}$ by a new edge with any label.
It also adds a new loop $\tilde{l}$ incident to $\tilde{v}$ with $\psi(\tilde{l})\neq id$.

In the covering graph these operations can be seen as graph operations that preserve
the underlying symmetry. Some of them can be recognized as performing standard -
non-symmetric - Henneberg operations simultaneously \cite{BS3,jkt,schtan,N&S}.


\begin{figure}[htp]
\begin{center}
\begin{tikzpicture}[very thick,scale=1]
\tikzstyle{every node}=[circle, draw=black, fill=white, inner sep=0pt, minimum width=5pt];
  \node [draw=white, fill=white] (c) at (0,-1.4)  {$\quad$};
\filldraw[fill=black!20!white, draw=black, thin, dashed](0,0)circle(1.3cm);
\node (p1) at (80:1cm) {};
\node (p2) at (140:1cm) {};
\node (p3) at (200:1cm) {};
\node (p4) at (260:1cm) {};
\node (p5) at (320:1cm) {};
\node (p6) at (20:1cm) {};

\node (p1a) at (100:1cm) {};
\node (p2a) at (160:1cm) {};
\node (p3a) at (220:1cm) {};
\node (p4a) at (280:1cm) {};
\node (p5a) at (340:1cm) {};
\node (p6a) at (40:1cm) {};

\draw [->, ultra thick](2,0)--(2.5,0);

\end{tikzpicture}
\hspace{0.3cm}
\begin{tikzpicture}[very thick,scale=1]
\tikzstyle{every node}=[circle, draw=black, fill=white, inner sep=0pt, minimum width=5pt];
\filldraw[fill=black!20!white, draw=black, thin, dashed](0,0)circle(1.3cm);
\node (p1) at (80:1cm) {};
\node (p2) at (140:1cm) {};
\node (p3) at (200:1cm) {};
\node (p4) at (260:1cm) {};
\node (p5) at (320:1cm) {};
\node (p6) at (20:1cm) {};

\node (p1a) at (100:1cm) {};
\node (p2a) at (160:1cm) {};
\node (p3a) at (220:1cm) {};
\node (p4a) at (280:1cm) {};
\node (p5a) at (340:1cm) {};
\node (p6a) at (40:1cm) {};

\node (p1n) at (98:1.5cm) {};
\node (p2n) at (158:1.5cm) {};
\node (p3n) at (218:1.5cm) {};
\node (p4n) at (278:1.5cm) {};
\node (p5n) at (338:1.5cm) {};
\node (p6n) at (38:1.5cm) {};

\draw [->] (p1n)--(p1);
\draw [->] (p1n)--(p1a);
\draw [->] (p2n)--(p2);
\draw [->] (p2n)--(p2a);
\draw [->] (p3n)--(p3);
\draw [->] (p3n)--(p3a);
\draw [->] (p4n)--(p4);
\draw [->] (p4n)--(p4a);
\draw [->] (p5n)--(p5);
\draw [->] (p5n)--(p5a);
\draw [->] (p6n)--(p6);
\draw [->] (p6n)--(p6a);

\end{tikzpicture}\\
(a)
\vspace{0.3cm}
\\
\begin{tikzpicture}[very thick,scale=1]
\tikzstyle{every node}=[circle, draw=black, fill=white, inner sep=0pt, minimum width=5pt];
\filldraw[fill=black!20!white, draw=black, thin, dashed](0,0)circle(1.5cm);
\node (p1) at (70:1.2cm) {};
\node (p2) at (130:1.2cm) {};
\node (p3) at (190:1.2cm) {};
\node (p4) at (250:1.2cm) {};
\node (p5) at (310:1.2cm) {};
\node (p6) at (10:1.2cm) {};

\node (p1a) at (110:1.2cm) {};
\node (p2a) at (170:1.2cm) {};
\node (p3a) at (230:1.2cm) {};
\node (p4a) at (290:1.2cm) {};
\node (p5a) at (350:1.2cm) {};
\node (p6a) at (50:1.2cm) {};

\node (p1i) at (98:0.6cm) {};
\node (p2i) at (158:0.6cm) {};
\node (p3i) at (218:0.6cm) {};
\node (p4i) at (278:0.6cm) {};
\node (p5i) at (338:0.6cm) {};
\node (p6i) at (38:0.6cm) {};
\draw  [<-] (p1a)--(p1);

\draw [<-] (p2a)--(p2);

\draw [<-] (p3a)--(p3);

\draw [<-] (p4a)--(p4);

\draw [<-] (p5a)--(p5);

\draw [<-] (p6a)--(p6);

\draw [->, ultra thick](2,0)--(2.5,0);
\end{tikzpicture}
\hspace{0.3cm}
\begin{tikzpicture}[very thick,scale=1]
\tikzstyle{every node}=[circle, draw=black, fill=white, inner sep=0pt, minimum width=5pt];
\filldraw[fill=black!20!white, draw=black, thin, dashed](0,0)circle(1.5cm);
\node (p1) at (70:1.2cm) {};
\node (p2) at (130:1.2cm) {};
\node (p3) at (190:1.2cm) {};
\node (p4) at (250:1.2cm) {};
\node (p5) at (310:1.2cm) {};
\node (p6) at (10:1.2cm) {};

\node (p1a) at (110:1.2cm) {};
\node (p2a) at (170:1.2cm) {};
\node (p3a) at (230:1.2cm) {};
\node (p4a) at (290:1.2cm) {};
\node (p5a) at (350:1.2cm) {};
\node (p6a) at (50:1.2cm) {};

\draw [<-] (p1a)--(p1);

\draw [<-] (p2a)--(p2);

\draw [<-] (p3a)--(p3);

\draw [<-] (p4a)--(p4);

\draw [<-] (p5a)--(p5);

\draw [<-] (p6a)--(p6);

\node (p1n) at (90:1.13cm) {};
\node (p2n) at (150:1.13cm) {};
\node (p3n) at (210:1.13cm) {};
\node (p4n) at (270:1.13cm) {};
\node (p5n) at (330:1.13cm) {};
\node (p6n) at (30:1.13cm) {};

\draw [->] (p1)--(p1n);

\draw [->] (p2)--(p2n);

\draw [->] (p3)--(p3n);

\draw [->] (p4)--(p4n);

\draw [->] (p5)--(p5n);

\draw [->] (p6)--(p6n);

\node (p1i) at (98:0.6cm) {};
\node (p2i) at (158:0.6cm) {};
\node (p3i) at (218:0.6cm) {};
\node (p4i) at (278:0.6cm) {};
\node (p5i) at (338:0.6cm) {};
\node (p6i) at (38:0.6cm) {};

\draw [->] (p1n)--(p1i);

\draw  [->]  (p2n)--(p2i);

\draw  [->]  (p3n)--(p3i);

\draw  [->]  (p4n)--(p4i);

\draw [->]  (p5n)--(p5i);

\draw  [->]  (p6n)--(p6i);

\end{tikzpicture}
\\
(b)
\\\hspace{0.3cm}
\begin{tikzpicture}[very thick,scale=1]
\tikzstyle{every node}=[circle, draw=black, fill=white, inner sep=0pt, minimum width=5pt];
\filldraw[fill=black!20!white, draw=black, thin, dashed](0,0)circle(1.3cm);
\node (p1i) at (90:1cm) {};
\node (p2i) at (150:1cm) {};
\node (p3i) at (210:1cm) {};
\node (p4i) at (270:1cm) {};
\node (p5i) at (330:1cm) {};
\node (p6i) at (30:1cm) {};
  \node [draw=white, fill=white] (c) at (0,-1.6)  {$\quad$};
\draw [->, ultra thick](2,0)--(2.5,0);
\end{tikzpicture}
\hspace{0.3cm}
\begin{tikzpicture}[very thick,scale=1]
\tikzstyle{every node}=[circle, draw=black, fill=white, inner sep=0pt, minimum width=5pt];
\filldraw[fill=black!20!white, draw=black, thin, dashed](0,0)circle(1.3cm);

\node (p1i) at (90:1cm) {};
\node (p2i) at (150:1cm) {};
\node (p3i) at (210:1cm) {};
\node (p4i) at (270:1cm) {};
\node (p5i) at (330:1cm) {};
\node (p6i) at (30:1cm) {};

\node (p1n) at (98:1.7cm) {};
\node (p2n) at (158:1.7cm) {};
\node (p3n) at (218:1.7cm) {};
\node (p4n) at (278:1.7cm) {};
\node (p5n) at (338:1.7cm) {};
\node (p6n) at (38:1.7cm) {};

\draw[->]  (p1n)--(p1i);

\draw[->]  (p2n)--(p2i);

\draw[->]  (p3n)--(p3i);

\draw[->]  (p4n)--(p4i);

\draw[->]  (p5n)--(p5i);

\draw[->]  (p6n)--(p6i);

\draw[->]   (p1n)--(p2n);

\draw[->]  (p2n)--(p3n);

\draw[->]  (p3n)--(p4n);

\draw[->]  (p4n)--(p5n);

\draw [->]  (p5n)--(p6n);

\draw[->]  (p6n)--(p1n);

\end{tikzpicture}
\\
(c)
\vspace{-0.3cm}
\end{center}
\vspace{-0.2cm}
\caption{(a) 0-extension (with $\gamma_1=\gamma_2=id$), (b) 1-extension  (with $\gamma_1=\gamma_2=\gamma_3=id$), (c) loop-1-extension  (with $\gamma_1=id$ and $\gamma_2=C_6$) in the covering graph.}
\label{fig:0extension_lifted}
\end{figure}
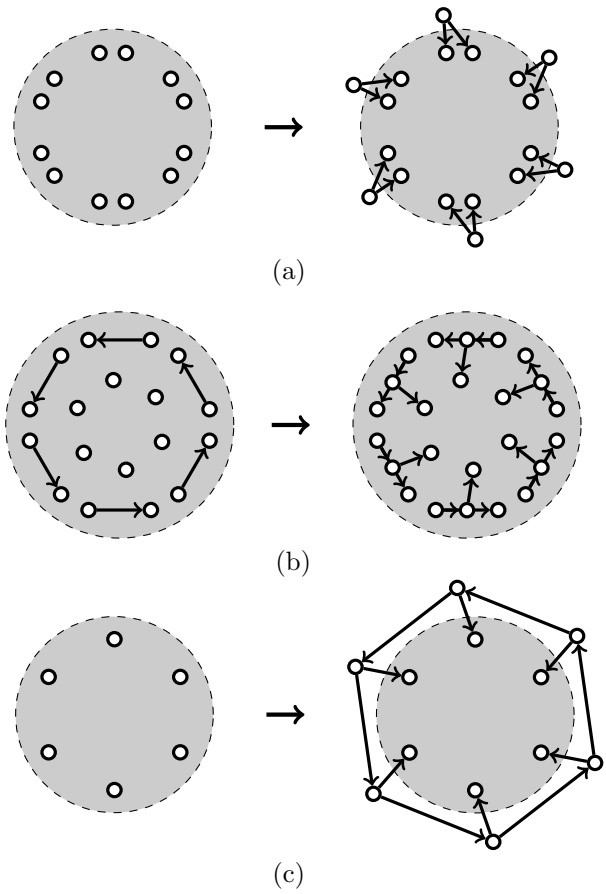

\subsection{Plane symmetry groups under free actions}

It is easy to observe that a 0-extension or loop-1-extension will create exactly one new component in the \SG-Assur decomposition and that this component will be the single new vertex in the gain graph. Moreover, this component will be above the component its edges join to in the partial order. For the 1-extension there are more possibilities as we describe in the next lemma. It is possible for a 1-extension to preserve the decomposition and the partial order, to preserve the decomposition but make incomparable components become comparable, to add a new single vertex component or to merge several components into a larger one.

\begin{Example}
Before proving this we illustrate the effect a 1-extension can have on the graph in Figure \ref{fig:C4}. In Figure \ref{fig:C4Induct} (a) we see a 1-extension applied completely within the component $A$ of the gain graph. This preserves the \SG-Assur decomposition. In (b) we see a 1-extension subdividing an edge between components resulting in a new single vertex component. In (c) we illustrate how two \SG-Assur components can be merged into one. This is easily observed by noting that the subgraph, of the gain graph, induced by $\{A,C,D\}$ has become strongly connected.
\end{Example}

\begin{figure}[ht]
    \begin{center}
  \subfigure[] {\includegraphics [width=.22\textwidth]{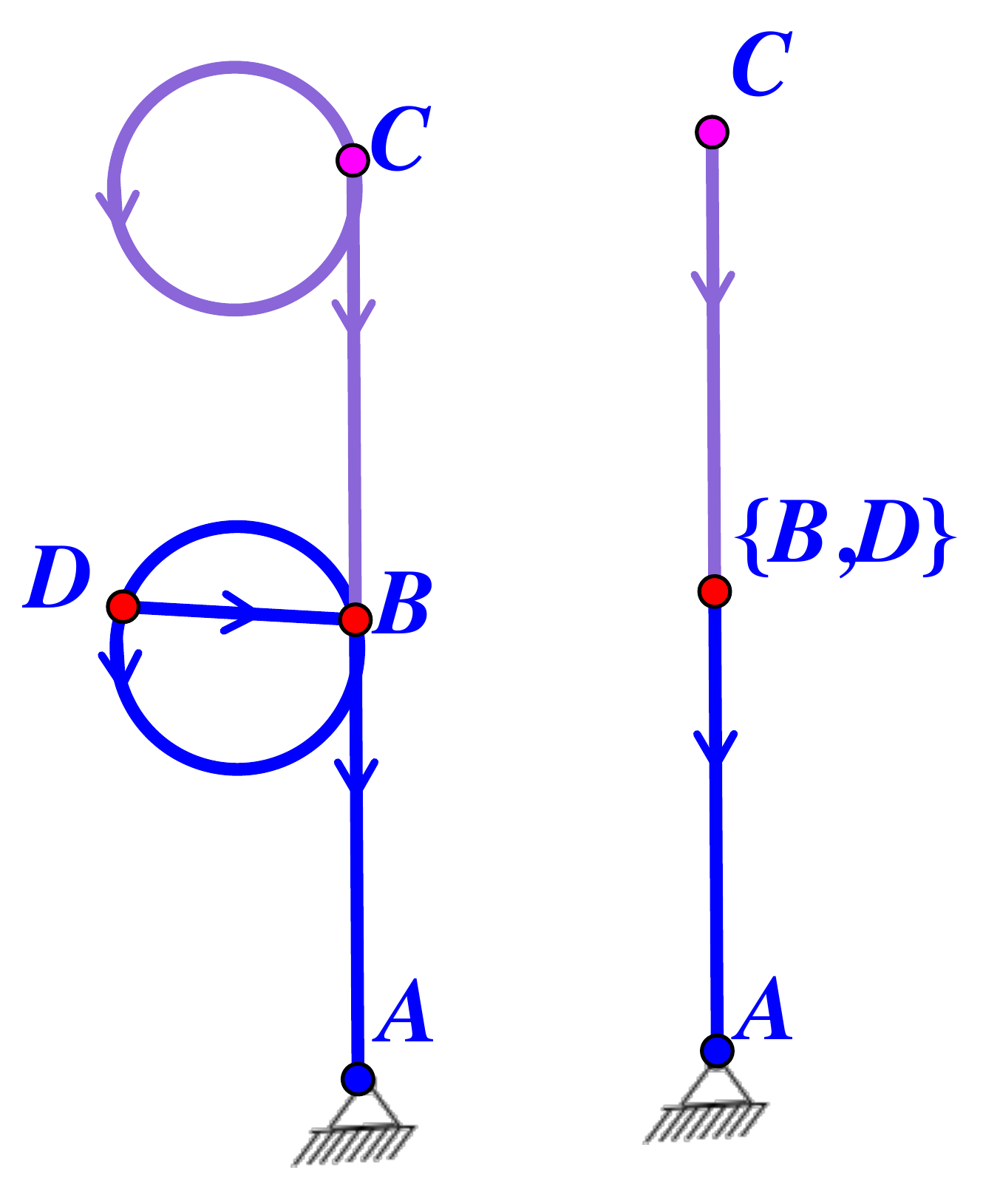}} \quad\quad\quad\quad
  \subfigure[] {\includegraphics [width=.18\textwidth]{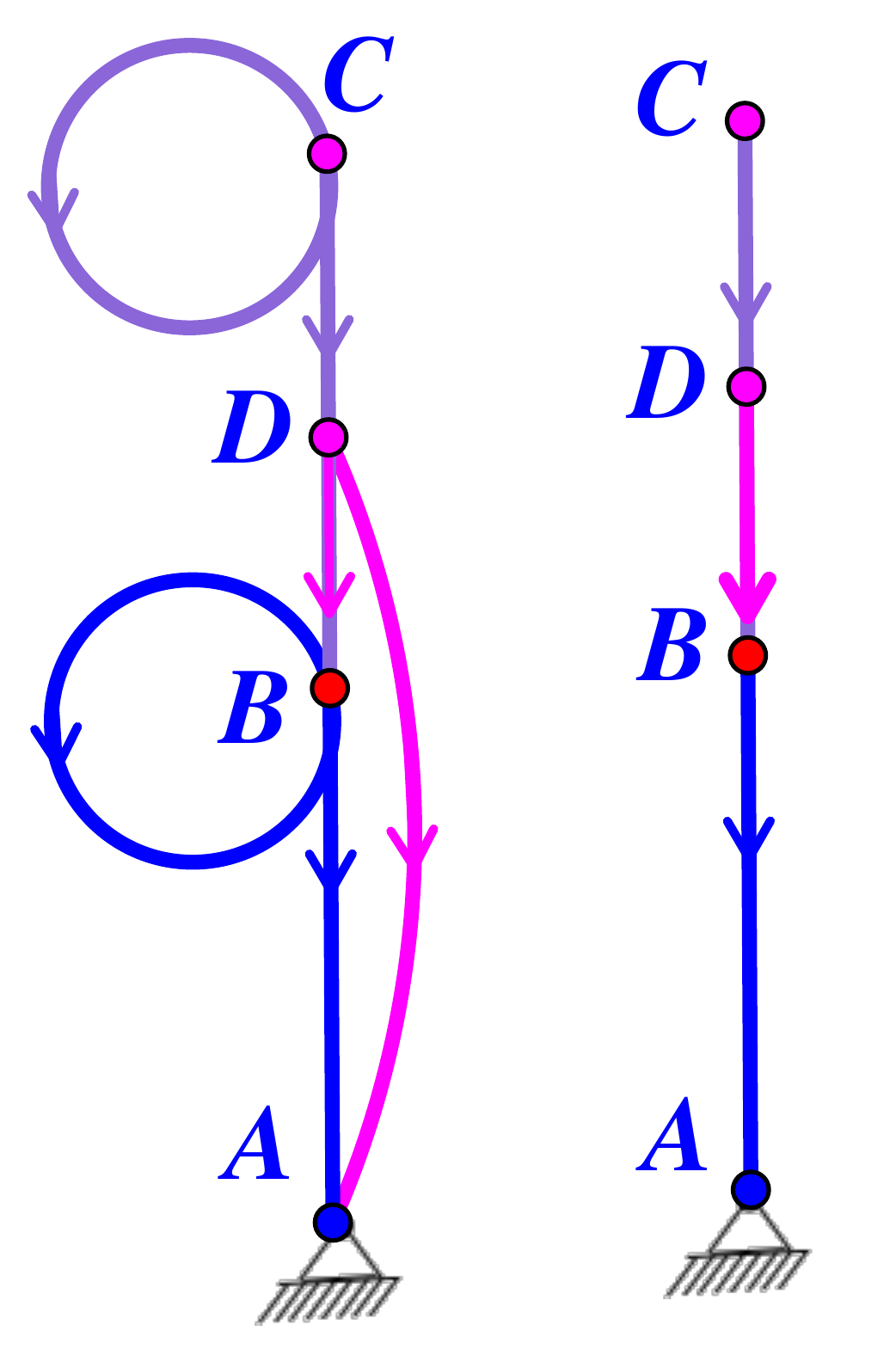}} \quad\quad\quad\quad
  \subfigure[] {\includegraphics [width=.19\textwidth]{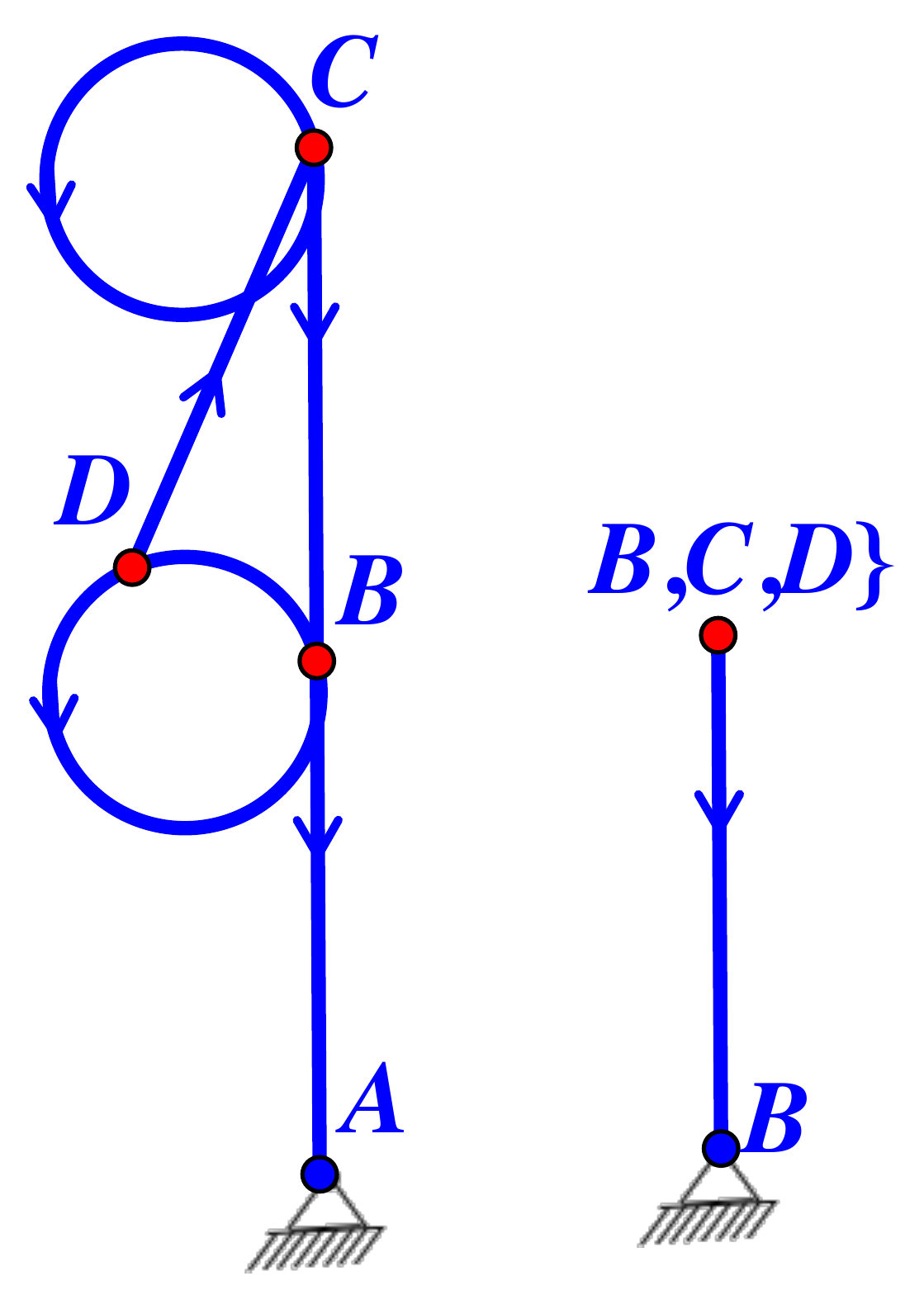}}\quad
    \subfigure[] {\includegraphics [width=.35\textwidth]{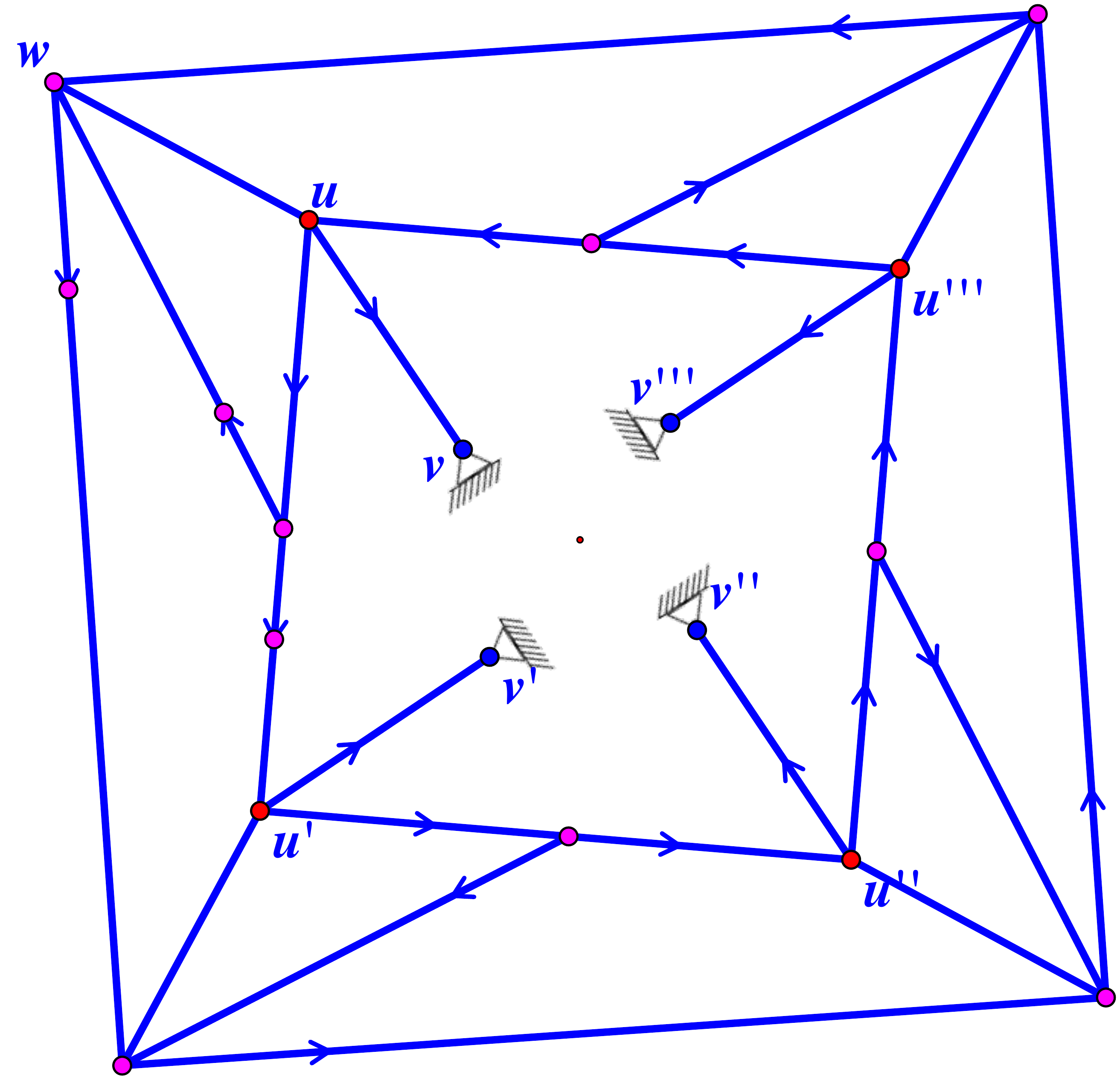}}
    \end{center}
    \caption{Three versions of the $1$-extension in the gain graph (a,b,c), illustrating Lemma~\ref{lem:1ext} - and their impact on the Assur decomposition.  (d) shows the induced graph $\hat{G}$ for figure (c).}
    \label{fig:C4Induct}
    \end{figure}

\begin{Lemma}
\label{lem:1ext}
Let \SG~be a plane symmetry group. Further, let \SG~act freely on the vertices and edges of a pinned \SG-isostatic graph $\hat G$ with associated gain graph $\hat H$ and \SG-Assur components $\hat H_1,\hat H_2,\dots, \hat H_n$ where $i<j$ implies $\hat H_i$ is not above $\hat H_j$ in the partial order ($i,j \in \{1,2,\dots, n\}$). Let $\hat H'$ be formed from $\hat H$ by a 1-extension deleting the edge $xy$ and adding a new vertex $v$ and edges $vx,vy,vw$. Then the covering graph $\hat G'$ of $\hat H'$ is pinned \SG-isostatic. Moreover
\begin{enumerate}
\item the \SG-Assur decomposition and its partial order are preserved if and only if (1) $xy \in \hat H_k$ and $w \in \hat H_k$, (2) $x \in V(\hat H_k)$, $y\in V(\hat H_j)$, $k>j$ and $w\in \hat H_k$ or (3) $x \in V(\hat H_k)$, $y\in V(\hat H_j)$, $k>j$ and $w\in \hat H_i$ for $i<k$ ($\hat H_i$ and $\hat H_k$ are comparable).
\item the \SG-Assur decomposition contains one new single vertex component if and only if (1) $x \in V(\hat H_k)$, $y\in V(\hat H_j)$, $k>j$ and $w\in \hat H_j$ or (2) $x \in V(\hat H_k)$, $y\in V(\hat H_j)$ and $k>j$ and $w\in \hat H_i$ ($\hat H_i$ and $\hat H_k$ are incomparable).
\item several comparable components in the \SG-Assur decomposition become one large component if and only if (1) $xy  \in E(\hat H_k)$ and $w\in \hat H_i$ $i>k$ ($\hat H_i$ and $\hat H_k$ are comparable) or (2) $x \in V(\hat H_k)$, $y\in V(\hat H_j)$ and $k>j$ and $w\in \hat H_i$ for $i>k$ ($\hat H_i$ and $\hat H_k$ are comparable).
\end{enumerate}
\end{Lemma}

\begin{proof}
That the covering graph is \SG-isostatic follows from \cite[Lemma 6.1]{jkt}.

There are two cases, with many sub-cases, for how the 1-extension may be applied.

\textbf{Case a} $xy  \in E(\hat H_k)$.
Subcases: i $w\in \hat H_k$; ii $w\in \hat H_i$ $i>k$ ($\hat H_i$ and $\hat H_k$ are comparable); iii $w\in \hat H_i$ $i<k$ ($\hat H_i$ and $\hat H_k$ are comparable); iv $w\in \hat H_i$ ($\hat H_i$ and $\hat H_k$ are incomparable).

\textbf{Case b} $x \in V(\hat H_k)$, $y\in V(\hat H_j)$ and $k>j$.
Subcases: i $w\in \hat H_k$; ii $w\in \hat H_j$; iii $w\in \hat H_i$ $i<k$ ($\hat H_i$ and $\hat H_k$ are comparable); iv $w\in \hat H_i$ $i>k$ ($\hat H_i$ and $\hat H_k$ are comparable); v $w\in \hat H_i$ ($\hat H_i$ and $\hat H_k$ are incomparable).

For case a i: without loss we may assume $xy$ is directed from $x$ to $y$. Direct the new edges from $x$ to $v$, $v$ to $y$ and $v$ to $w$. It is immediate that this makes $\hat H_k +v$ strongly connected (since $\hat H_k$ was) and including edges to components below it gives us a \SG-directed orientation. It follows from Theorem \ref{thm:decomp} that the \SG-Assur decomposition is the same.

Applying similar reasoning in each other case reveals that we are in the relevant case of the lemma except for cases a iv and b v which maintain the decomposition but change the partial order by making incomparable components comparable. In particular the choice of directions illustrated in Figures \ref{fig:inductive} and \ref{fig:0extension_lifted} show we always have a \SG-orientation of $\hat H'$.
\end{proof}

It is possible to state an analogue of the lemma above for X-replacement (this is defined, for example, in \cite{N&R}). However there are many more cases.
There are numerous additional operations including vertex splitting \cite{Wvsplit} and vertex-to-$K_4$ moves \cite{NOP2} that preserve the counts required and have previously been adapted to symmetric settings \cite{N&S}. These moves are more intricate from our perspective since the moves include more freedom for new edges. Hence it is harder to track the directions on these new edges and preserve the \SG-orientation. When that is possible similar case by case analysis will reveal the effect of the operation on the components of a \SG-Assur decomposition.


\subsection{$d$-dimensions and non-fixed actions}

The 0-, 1- and loop-1-extensions can be generalised to arbitrary dimensions. For example the analogue of 0-extension in dimension $d$ is simply to add a new vertex \(\tilde{v}\) and $d$ new non-loop edges
$\te_{1},\te_{2},\dots, \te_{d}$ taking care with the gain labels in the obvious way. Similarly a $d$-dimensional 1-extension removes a single edge and adds in a vertex of degree $d+1$ (with new edges and gains similarly chosen as specified explicitly for the $d=2$ case).

In dimension 3 we can repeat Lemma \ref{lem:1ext}, with many more cases, to control exactly when a $3$-dimensional 1-extension preserves components or reduces the number of components, etc. Moreover we can consider non-free actions. The quantity dim$U_{\bp(x)}$ is either $0,1,2,3$ so by choosing  $d$-dimensional 1-extensions for the appropriate $d\in \{0,1,2,3\}$ we can repeat the 2 and 3-dimensional arguments for non-free actions.

\section{Extensions and Conclusions} \label{sec:conclusions}

In the same way that orbit matrices have proven useful in a number of broader settings, the analysis extends to other settings where we can create square pinned orbit matrices.

\subsection{Extensions to `anti-symmetric' orbit matrices}

It was shown in \cite{BS2} that the rigidity matrix $R(G, \bp)$ of a \SG-symmetric framework $(G, \bp)$ can be
transformed into a block-decomposed form, where each block $R_i(G,\bp)$ corresponds to an irreducible
representation $\rho_i$ of the group $\mathcal{S}$. This breaks up the rigidity analysis of $(G, \bp)$ into a number
of independent subproblems \cite{FGsymmax,BS2}. In fact, the fully \SG-symmetric rigidity properties of $(G, \bp)$ are described
by the block matrix $R_1(G, \bp)$ corresponding to the trivial irreducible representation $\rho_1$ of $\mathcal{S}$. Thus, as shown in \cite{BSWWorbit}, this block matrix
is equivalent to the orbit matrix defined in Definition~\ref{orbitmatrixdef}.

In the recent paper \cite{schtan}, an orbit matrix was established for \emph{each} of the blocks $R_i(G, \bp)$ in the case where the group is abelian, and
these new tools were successfully used to characterize \SG-generic infinitesimally rigid graphs (i.e., graphs $G$ with the property that \SG-generic realisations of $G$ do not have \emph{any} non-trivial infinitesimal motions) in the plane for the groups  $\mathcal{C}_s$, $\mathcal{C}_2$ and $\mathcal{C}_3$.

Since the structure of the `anti-symmetric' orbit matrices is similar to the structure of the orbit matrix from  Definition~\ref{orbitmatrixdef}, we expect that the methods and results established in this paper can  be extended to these other orbit matrices in a straightforward fashion. This would allow us to analyse symmetric frameworks for
infinitesimal motions and stresses that break certain symmetries, but preserve others. However, note that an `anti-symmetric' infinitesimal motion typically does not extend to a finite motion \cite{KG1,BS2}.

\subsection{Extensions to matrices for other constraint systems}

All of the techniques (and even examples) of this paper extend directly to spherical linkages, which are studied in settings such as \cite{sphere} Chapters 7-10.  Papers such as \cite{BSWWorbit} carry out the transfer for all orbit matrices, including the original rigidity matrix (where the group is the identity).  There is no complication to transferring these transfer processes to pinned frameworks, and the \SG-Assur decompositions will also transfer.
All the techniques also extend to cones - provided the cone point (or the cone point and its symmetric images) has the appropriate symmetry.

Periodic frameworks in the plane have been studied by a number of groups.  Whether the lattice is fixed \cite{Ross}, partially variable or fully flexible \cite{M&Tper, B&S} there is a natural rigidity matrix which can be made square by pinning. Moreover periodic frameworks have been understood combinatorially using gain graphs.  Hence there is potential to use the techniques in this paper to generate Assur decompositions when  one vertex in the fixed lattice can substitute for the ground.

In \cite{NOP,NOP2} frameworks in 3-dimensions supported on fixed surfaces have been studied and symmetric analogues of Theorem \ref{thm:symmetry_reflection} have been obtained \cite{N&S}. The rigidity matrix has $3|V|$ columns and when the framework is isostatic $2|V|-k+|V|$ rows (where $k$ is the number of isometries of $\mathbb{R}^3$ admitted by the surface). When the surface admits no isometries of $\mathbb{R}^3$ then no pinning is required. However unlike the situation in Figure \ref{fig:D_2}, all frameworks on such a surface can be considered as pinned. Otherwise it seems like these frameworks should be decomposable into Assur components by pinning $k$ degrees of freedom to make the matrix square and then applying the techniques in this paper.

In \cite[Subsection 5.1]{3directed} it was outlined how to extend the equivalence of Assur decompositions to various alternative types of frameworks allowing bodies, bars, joints, pins, hinges, etc.  All of these kinds of mechanisms occur in the mechanical engineering literature.   All of the techniques and results of this paper will have analogous decompositions under symmetry, since the symmetry-adapted analogues of the basic rigidity results have recently been developed.

In \cite{LeeSid} general CAD systems in 3-dimensions were investigated with constraint matrices which are analogs of the rigidity matrices for bodies and bars.  It is natural to consider pinned CAD systems (essentially make one of the bodies the ground) and to also consider systems of symmetric CAD constraints.  We anticipate all of the decompositions described here will again extend to these more general sets of constraints.


%
%

\subsection{Uniquely stressed graphs}
\label{subsec:circuits}

In \cite{SSW1} many combinatorial properties of 2-Assur graphs were developed by considering minimal dependencies in the rigidity matrix. Since the rigidity matrix induces a linear matroid known as the rigidity matroid, these minimal dependencies are known as circuits.


This did not generalize to Assur graphs in 3-space.  In this subsection we conjecture that the symmetry analogue works for certain plane groups and again breaks down in higher dimensions.

Let us define a \emph{\SG-circuit} to be a graph $G$ for which the rows of $\mathcal{O}( G, \psi,  \bp)$ (with $\bp$ \SG-regular) induce a circuit in the forced-\SG-symmetric rigidity matroid. Then $\hat{G}$ is a \SG-circuit if and only if there is a unique vector in the cokernel of the orbit matrix which is non-zero in every coordinate. In other words $G$ is a \SG-circuit if and only if it has a unique \SG-symmetric self-stress which is non-zero on all edges.

Let $\hat{G}^-$ be the graph formed by collapsing all vertices in $P$ to a single vertex $c$ (and every edge of the form $xy$ for $x\in I,y\in P$ is replaced by the corresponding edge $xc$).
A key necessary condition for $\hat{G}^-$ to be a circuit when $\hat{G}$ is \SG-Assur is that the degree of freedom of the `ground' under $\cal{S}$ is $1$ larger than the degree of freedom of a single vertex fixed by the group $\cal{S}$.  This observation leads us to consider 3 candidates, although 2 of them (mirror symmetry or rotational symmetry in 3-space) are ruled out since flexible circuits such as the double banana translate to the symmetric setting. Consider Figure \ref{fig:Weakly}: the $\mathcal{C}_3$-gain graph, with ground shrunk, becomes \SG-dependent but not a circuit.

This leaves rotation symmetry in the plane where a full ground has $1$ degree of freedom, while a single vertex fixed on the axis has no degrees of freedom. Figure \ref{fig:Circuits} illustrate this, see also Figures~\ref{fig:desargues} and \ref{fig:desargues2}.

\begin{Conjecture}
Let $\mathcal{C}_n$, $n\geq 2$, be a rotation group, in the plane, centred at the origin. Let $(\hat{G},\hat \bp)$ be pinned $\mathcal{C}_n$-regular. For $\hat{G}^-$ choose $\hat \bp^-$ to equal $\hat \bp$ on all inner vertices and $\hat \bp^-(c)=(0,0)$. Then $\hat{G}$ is $\mathcal{C}_n$-Assur if and only if $\hat{G}^-$ is a $\mathcal{C}_n$-circuit.
\end{Conjecture}

The conjecture seems to require a version of Theorem \ref{thm:Scounts} for arbitrary group actions. Since there is little theory for \SG-circuits that we could make use of, we have not exerted much effort in trying to prove it by alternate means.

\begin{figure}[ht]
    \begin{center}
  \subfigure[] {\includegraphics [width=.30\textwidth]{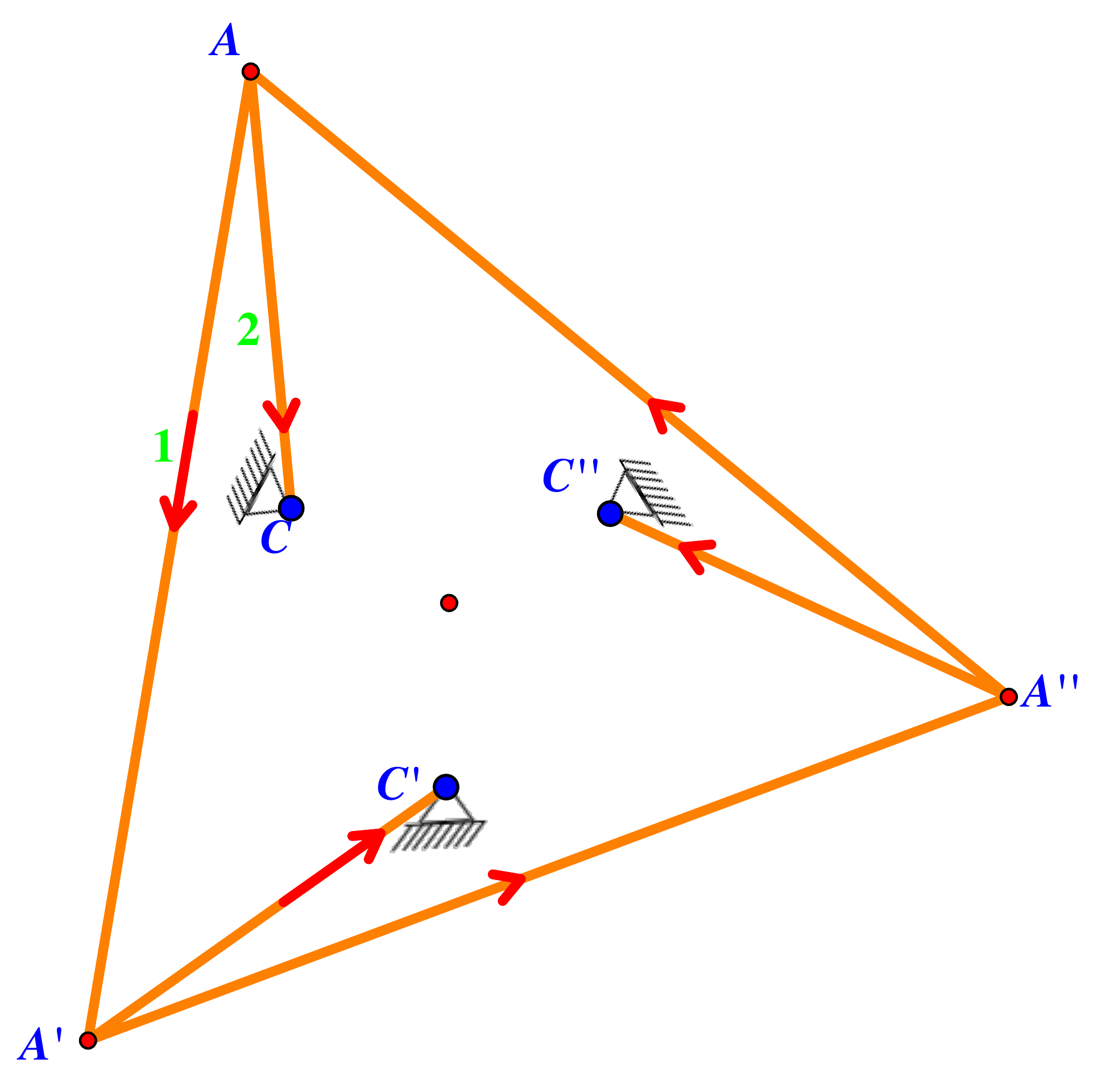}} \quad
\subfigure[] {\includegraphics [width=.14\textwidth]{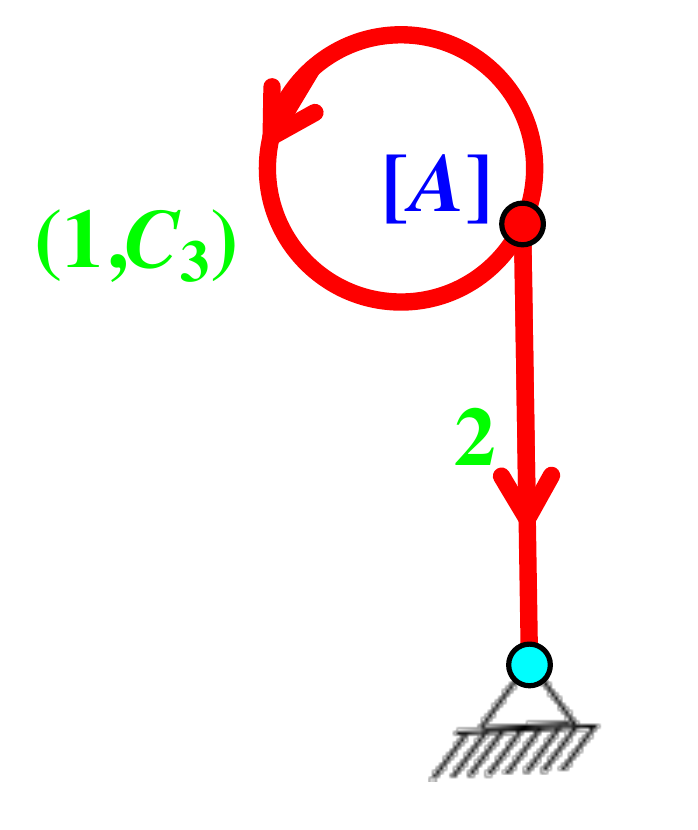}}\quad
  \subfigure[] {\includegraphics [width=.30\textwidth]{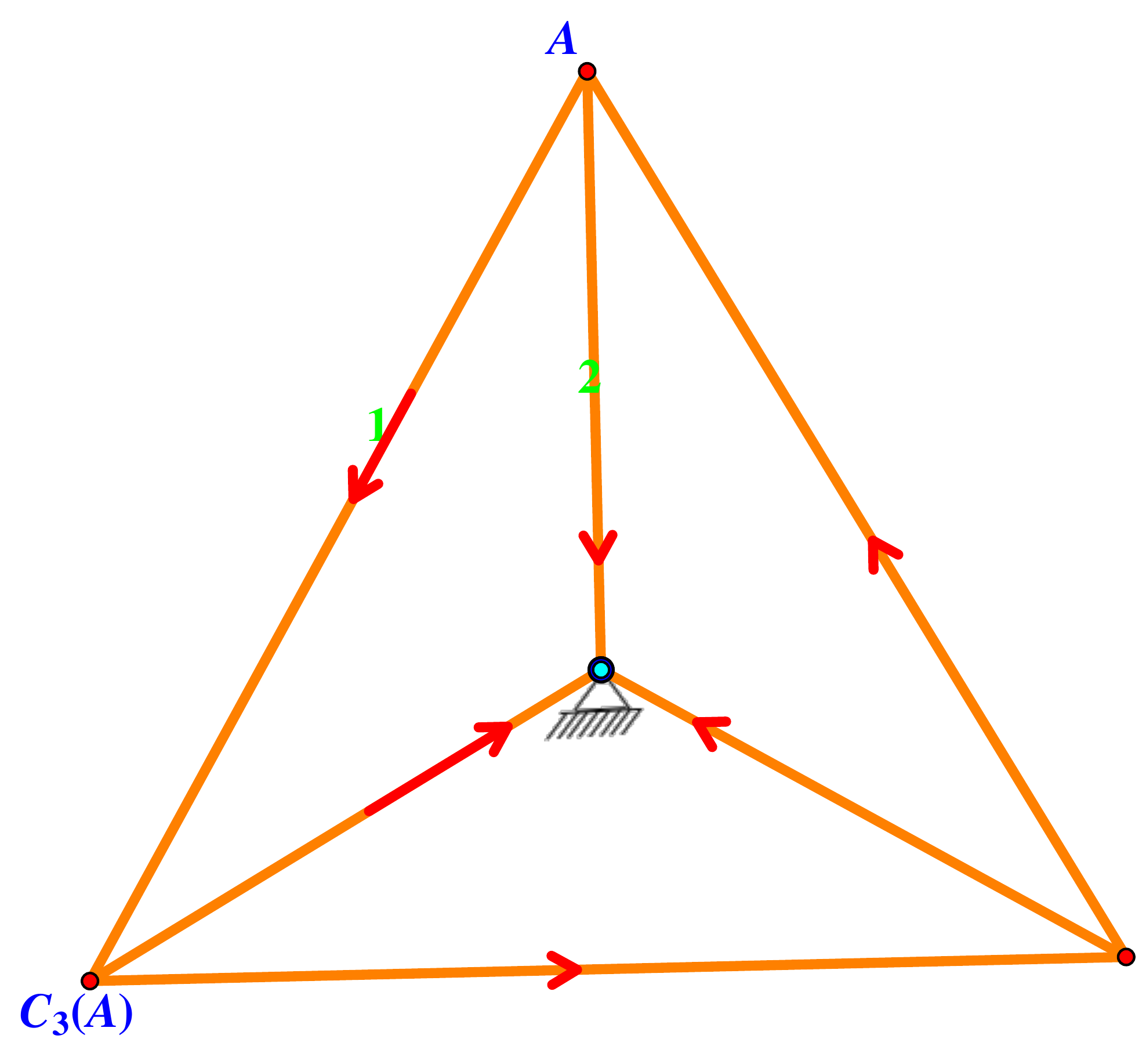}} \quad
\subfigure[] {\includegraphics [width=.14\textwidth]{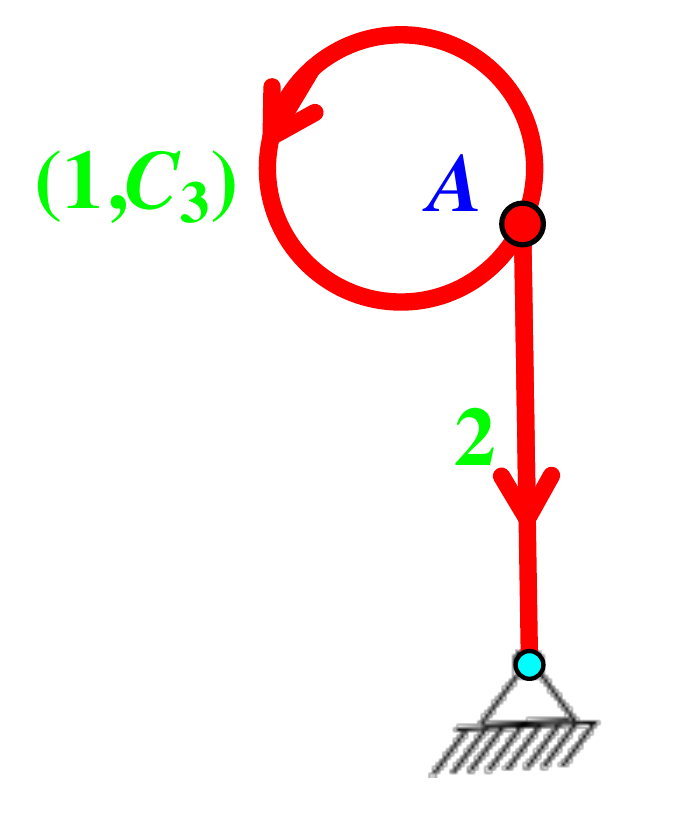}}\quad
      \end{center}
    \caption{A pinned framework in the plane with $\mathcal{C}_3$ symmetry (a), which is  $\mathcal{C}_3$-Assur (a,b), and becomes a  $\mathcal{C}_3$-circuit when the ground is shrunk to a point (c,d).}
    \label{fig:Circuits}
    \end{figure}

\subsection*{Acknowledgements}

Research supported in part by a grant from NSERC (Canada) to Walter Whiteley.


%


\bibliographystyle{mdpi}
\makeatletter
\renewcommand\@biblabel[1]{#1. }
\makeatother

\end{document}